\documentclass[english]{amsart}
\usepackage[T1]{fontenc}
\usepackage{geometry}
\usepackage{fancyhdr}
\usepackage{color}
\usepackage{babel}
\usepackage{amstext}
\usepackage{amsthm}
\usepackage{amssymb}
\usepackage{setspace}
\usepackage[unicode=true,pdfusetitle,
 bookmarks=true,bookmarksnumbered=false,bookmarksopen=false,
 breaklinks=true,pdfborder={0 0 1},backref=false,colorlinks=true]
 {hyperref}
\hypersetup{
 linkcolor=black, citecolor=black,  hyperfootnotes=false, unicode=true}

\makeatletter

\providecommand{\tabularnewline}{\\}

\numberwithin{equation}{section}
\numberwithin{figure}{section}
\theoremstyle{plain}
\newtheorem{thm}{\protect\theoremname}[section]
\theoremstyle{remark}
\newtheorem{claim}[thm]{\protect\claimname}
\theoremstyle{plain}
\newtheorem{lem}[thm]{\protect\lemmaname}
\theoremstyle{definition}
\newtheorem{problem}[thm]{\protect\problemname}
\theoremstyle{definition}
\newtheorem{defn}[thm]{\protect\definitionname}
\theoremstyle{plain}
\newtheorem{prop}[thm]{\protect\propositionname}
\theoremstyle{plain}
\newtheorem{cor}[thm]{\protect\corollaryname}
\theoremstyle{definition}
\newtheorem{example}[thm]{\protect\examplename}
\theoremstyle{plain}
\newtheorem*{thm*}{\protect\theoremname}
\theoremstyle{plain}
\newtheorem{conjecture}[thm]{\protect\conjecturename}

\usepackage{tikz-cd}
\UseRawInputEncoding

\newtheorem{theoremalpha}{Theorem}
\newtheorem{problemalpha}[theoremalpha]{Problem}
\newtheorem{lemmaalpha}[theoremalpha]{Lemma}
\newtheorem{propalpha}[theoremalpha]{Proposition}

\makeatother

\providecommand{\claimname}{Claim}
\providecommand{\conjecturename}{Conjecture}
\providecommand{\corollaryname}{Corollary}
\providecommand{\definitionname}{Definition}
\providecommand{\examplename}{Example}
\providecommand{\lemmaname}{Lemma}
\providecommand{\problemname}{Problem}
\providecommand{\propositionname}{Proposition}
\providecommand{\theoremname}{Theorem}

\begin{document}
\title{Computable analysis on the space of marked groups}
\author{Emmanuel Rauzy}
\begin{abstract}
We begin the systematic study of decision problems for finitely generated
groups given by a solution to their word problem. We relate this to
the study of computable analysis on the space of marked groups. We
point out that several distinct approaches to computable analysis,
some of which are sometimes considered obsolete, yield relevant results.
In particular, we give necessary and sufficient conditions in terms
of Banach-Mazur computability for the existence of a finitely presented
group with solvable word problem but whose subgroups with a certain
property cannot be recognized. 

We classify group properties in different effective Borel hierarchies.
For most common group properties, the classical and effective Borel
classifications coincide. However, we show that the set of LEF groups
is a closed set that is computably a $G_{\delta}$, but not computably
closed. 

Finally, we show that the space of marked groups is a Polish space
which is not \emph{computably Polish, }because it does not admit a
dense and computable sequence. This poses several interesting problems
in terms of computable topology. The space of marked groups is the
first natural example of this kind. 
\end{abstract}

\maketitle
\tableofcontents{}

\section{Introduction }

When Max Dehn introduced in 1911 \cite{Dehn1911,Dehn1987} the isomorphism
problem for groups, he was concerned solely with finitely presented
groups. The problem he considered was thus: given two finite presentations,
decide whether or not they define isomorphic groups. 

Max Dehn likely hoped for a positive solution to the isomorphism problem
for finitely presented groups. 

By the Adian-Rabin Theorem \cite{Rabin1958,NybergBrodda2022}, not
only the isomorphism problem for finitely presented groups, but even
the problem of deciding whether a finite presentation defines the
trivial group are undecidable. 

Because of this theorem, group theorists have learned to expect that
``most properties'' of finitely presented groups are undecidable,
and that ``almost nothing can be said of a group'' given a finite
presentation of it. 

Consider however the following theorem of Groves and Wilton: 
\begin{thm}
[\cite{Groves2009}] \label{thm:GrovesWilton-1}There exists an
algorithm that, given as input a presentation for a group $G$ and
a solution to the word problem in $G$, determines whether or not
$G$ is free.
\end{thm}

This theorem is non-trivial, the proof given in \cite{Groves2009}
relies on a deep understanding of the universal theory of free groups
and of its models, the limit groups. Another proof of that result
was given by Touikan in \cite{Touikan2018}. 

Our interpretation of this theorem is the following: 
\begin{claim}
\label{claim: Basic claim}The correct description to study decision
problems for finitely generated groups is the one associated to ``finite
presentations together with a solution to the word problem''. 
\end{claim}

We will discuss at length how this sentence can be formalized. 

In light of Claim \ref{claim: Basic claim}, the Adian-Rabin Theorem
does not show that infinite groups cannot be manipulated by algorithms,
but simply that we should consider finite presentations as incomplete
descriptions, lacking a full solution to the word problem. In \cite{Rauzy21},
we show that the usual proof of the Adian-Rabin Theorem provides strictly
no information about decision problems for groups when a solution
to the word problem is available. 

The description of finitely presented groups associated to Claim \ref{claim: Basic claim}
is a composite one: it can be expressed as the \emph{conjunction}
of two simpler descriptions: groups given by finite presentations,
and groups given by their solutions to the word problem. (We will
see that this ``conjunction'' has a precise meaning in the lattice
of subnumberings types of marked groups, it corresponds to the \emph{join}
\emph{of two subnumbering types}.) A natural and necessary step towards
understanding this composite description is to understand both of
its halves. While plenty is known about decision problems for finitely
presented groups, virtually nothing is known about decision problems
for groups given by their solution to the word problem. 

The present article is the first to address this problem. 

Our main method is to use tools that were developed in computable
analysis and to apply them to the space of marked groups. 

However, it turns out that many interesting phenomena occur in this
context, non-trivial interactions between computability, topology
and group theory, that do not naturally occur in computable analysis.
In particular, the space of marked groups is the first natural example
of a Polish space which is not computably Polish with respect to its
natural \emph{representation}, i.e. with respect to the natural notion
of computability that we have on it.

\subsection{Groups given by a solution to their word problem }

The first step of our study is to formalize the concept of ``a group
being given by a solution to its word problem''. 

For $k\in\mathbb{N}$, a $k$\emph{-marked group }is a countable group
$G$ together with a generating tuple $(s_{1},...,s_{k})\in G^{k}$.
Two $k$-marked groups $(G,(s_{1},...,s_{k}))$ and $(H,(s'_{1},...,s'_{k}))$
are isomorphic if the function $s_{i}\mapsto s_{i}'$ extends to a
group isomorphism. Most finite descriptions of finitely generated
groups (and in particular finite presentations) are actually descriptions
of marked groups. 

By fixing a free group $\mathbb{F}_{k}$ with basis $(x_{1},...,x_{k})$,
we can identify any marked group $(G,(s_{1},...,s_{k}))$ with the
kernel of the morphism $\mathbb{F}_{k}\rightarrow G$ defined by $x_{i}\mapsto s_{i}$.
In other words: a marked group is uniquely determined by its set of
\emph{relations}. 

Consider a bijection $\theta_{k}:\mathbb{N}\rightarrow\mathbb{F}_{k}$.
The set of relations satisfied by the $k$-marked group $(G,S)$ can
be seen as a binary sequence $(u_{n})_{n\in\mathbb{N}}\in\{0,1\}^{\mathbb{N}}$,
defined by 
\[
u_{n}=1\iff\theta_{k}(n)\text{ is a relation in }(G,S).
\]
We call $(u_{n})_{n\in\mathbb{N}}$ the \emph{binary expansion} of
$(G,S)$. 

It now appears that, in order to define functions that are ``computable
for groups given by a solution to the word problem'', it is necessary
and sufficient to be able to define computability on the Cantor space
$\{0,1\}^{\mathbb{N}}$. 

We will thus turn our attention to the mathematical field where computability
on Cantor space and on other spaces with cardinality that of the continuum
is studied: \emph{computable analysis}. 

\medskip

Defining computable functions on the Cantor space and on the set of
real numbers is a problem that goes back to Turing himself, who, in
\cite{Turing1936}, right after having defined the machines that now
bear his name, introduced those real numbers whose decimal expansion
can be output by such a machine, the \emph{computable real numbers},
and proposed a notion of a computable function from the computable
reals to the computable reals. 

Several other authors have proposed notions of computable functions
on the real numbers or on the Cantor space. Contrary to what happens
for functions defined on the natural numbers, for which there is a
single unanimously acclaimed notion of computable function, in the
case of functions defined on $\{0,1\}^{\mathbb{N}}$ or on $\mathbb{R}$,
there isn't a single definition that we will be able to choose and
use throughout. 

We will present here three notions: Banach-Mazur computability, Markov
computability (also called Type 1 computability), and Type 2 computability.
A detailed historical account can be found in \cite{Avigada}.

\subsection*{Type 2 computability}

Type 2 computability is a concept that goes back to Kleene. The general
framework we now present is the one developed by Weihrauch between
1985 and 2000 \cite{Kreitz1985,Weihrauch2000}, it was later on extended
and rendered more robust by Schröder \cite{Schroeder2001}. Modern
references are \cite{Pauly2016,Iljazovic2021,Schroeder2021,Brattka2021}. 

To define Type 2 computability, we start by defining computability
on Baire space $\mathbb{N}^{\mathbb{N}}$. 

A partial function $f:\,\subseteq\mathbb{N}^{\mathbb{N}}\rightarrow\mathbb{N}^{\mathbb{N}}$
is called \emph{Type 2 computable }if there is a partial functions
$F:\,\subseteq\mathbb{N}^{*}\rightarrow\mathbb{N}^{*}$ (where $\mathbb{N}^{*}$
designates the set of words over $\mathbb{N}$) such that 
\begin{itemize}
\item $F$ is computable (in the usual Church-Turing sense);
\item $F$ is prefix increasing: if $u$ is a prefix of $v$, $F(u)$ is
a prefix of $F(v)$; 
\item And $F(u_{0}...u_{n})\underset{n\to\infty}{\rightarrow}f((u_{n})_{n\in\mathbb{N}})$. 
\end{itemize}
Notice that it follows immediately from the definition that a Type
2 computable function $f:\,\subseteq\mathbb{N}^{\mathbb{N}}\rightarrow\mathbb{N}^{\mathbb{N}}$
is continuous in the Baire space topology, and even effectively continuous:
there is a Turing machine which, given a finite sequence of natural
numbers $w\in\mathbb{N}^{*}$, produces a sequence $(u_{i})_{i\in\mathbb{N}}\in(\mathbb{N}^{*})^{\mathbb{N}}$
such that 
\[
f^{-1}(w\mathbb{N}^{\mathbb{N}})=\text{dom}(f)\cap\bigcup_{i\in\mathbb{N}}u_{i}\mathbb{N}^{\mathbb{N}}.
\]
Thus the preimage of a basic clopen set of Baire space can be computably
written as a union of basic clopen sets. 

Once computability is defined on Baire space, we can extend it to
other sets by encoding their elements using sequences of natural numbers.
A \emph{representation} of a set $X$ is a partial surjection $\rho:\,\subseteq\mathbb{N}^{\mathbb{N}}\rightarrow X$.
A function $f:X\rightarrow Y$ between represented spaces $(X,\rho)$
and $(Y,\tau)$ is called \emph{computable, }or\emph{ $(\rho,\tau)$-computable,
}if there is a computable function $F:\,\subseteq\mathbb{N}^{\mathbb{N}}\rightarrow\mathbb{N}^{\mathbb{N}}$
such that for all $p\in\text{dom}(\rho)$, $f(\rho(p))=\tau(F(p))$.
The situation is summed up in the following diagram:

\begin{center}
\begin{tikzcd} 
X \arrow[r, "f"] & Y \\ 
\mathbb{N}^\mathbb{N} \arrow[u, "\rho"] \arrow[r, "F" swap] & \mathbb{N}^\mathbb{N} \arrow[u, "\tau"] 
\end{tikzcd}
\end{center}

In this case, the function $F$ is called a \emph{computable realizer}
of $f$. When $\rho(p)=x$, $p$ is called a $\rho$\emph{-name} of
$x$. 

Type 2 computability is strongly related to topology, because, as
was said before, computable functions on Baire space are continuous,
and conversely, continuous functions on Baire space are computable
modulo some oracle. This relationship (i.e. continuity being ``computability
modulo some oracle'') extends to all topological spaces that admit
an \emph{admissible representation}, these were characterized by Schröder
\cite{Schroeder2001,Schroeder2021} as the $\text{T}_{0}$ quotients
of countably based spaces, they include all second countable spaces. 

In particular, every represented set $(X,\rho)$ is naturally equipped
with the final topology of the representation $\rho$, and we get
the following relation between computability and topology: the open
sets of the final topology of $\rho$ are exactly the sets that are
\emph{semi-decidable modulo some oracle. }

\subsection*{Markov computability}

Markov computability, sometimes also called Type 1 computability,
is the notion of computability that was considered by Turing in \cite{Turing1936}.
It is named after Andreï Andreïevitch Markov\footnote{Markov was the son of the famous probabilist after whom Markov chains
are named. Both are called Andreï Andreïevitch Markov.}, who was the leader of the ``Russian school of constructivism''
who considered that a mathematical object exists only if we can exhibit
a Turing machine that can, in some sense, represent this object. 

The idea of Markov computability is very similar to that of Type 2
computability, except that we only allow finite descriptions of objects.
We thus consider as base set $\mathbb{N}$, with computability notion
that of Church and Turing, and extend computability to other sets
via \emph{numberings}. A \emph{numbering} of a set $X$ is a partial
surjection $\nu:\,\subseteq\mathbb{N}\rightarrow X$. As for representations,
if $\nu(n)=x$, we say that $n$ is a $\nu$-name of $x$. A function
$f:X\rightarrow Y$ between numbered sets $(X,\nu)$ and $(Y,\mu)$
is called \emph{Markov computable, }or\emph{ $(\nu,\mu)$-computable,
}if there is a computable function $F:\,\subseteq\mathbb{N}\rightarrow\mathbb{N}$
such that for all $n\in\text{dom}(\nu)$, $f(\nu(n))=\mu(F(n))$.

While, in Type 2 computable analysis, computable functions are by
definition effectively continuous, some of the most important theorems
in Markovian computable analysis are effective continuity theorems:
theorems that state that Markov computable functions between sets
with some specified properties are automatically effectively continuous.
These include results due to Myhill and Shepherdson \cite{Myhill1955},
Kreisel, Lacombe and Schoenfield \cite{KLS57}, Ceitin \cite{Ceitin1967},
Moschovakis \cite{Moschovakis1964} and Spreen \cite{Spreen1998}. 

The easiest one to quote is the following one:
\begin{thm}
[Kreisel-Lacombe-Schoenfield, \cite{KLS57}]Let $\text{Tot}$ be
the subset of $\mathbb{N}^{\mathbb{N}}$ consisting of total computable
functions. Any Markov computable function $f:\text{Tot}\rightarrow\text{Tot}$
is effectively continuous. And this statement is uniform: from the
code of a computable function $f$, one can recover a code for a function
that witnesses effective continuity of $f$. 
\end{thm}

The theorem of Ceitin extends this result to Type 1 computable Polish
spaces, that of Moschovakis to spaces that admit a certain Computable
Choice Axiom (see Definition \ref{def:Moschovakis (B) condition}).
Finally, Spreen gives in \cite{Spreen1998} a unified proof of the
results of Ceitin and of Myhill and Shepherdson, generalizing them
both. 

Other references on Markov computability are: \cite{Aberth1980,Kushner1984,Markov1963,Markov54}. 

\subsection*{Banach-Mazur computability}

Banach-Mazur computability was invented by Banach and Mazur in Lwów
before the second world war, but their results were not published
before much later (see \cite{Mazur63}). A modern account of their
results is given in \cite{Hertling2001}. 

The general setting of Banach-Mazur computability is the same one
as that of Markov computability: the basic notion of computability
is given by Church-Turing computability on $\mathbb{N}$, and is then
transferred to other sets via numberings. However, the definition
of a computable function changes: a function $f:X\rightarrow Y$ between
numbered sets $(X,\nu)$ and $(Y,\mu)$ is called \emph{Banach-Mazur
computable }if it maps $\nu$-computable sequences to $\mu$-computable
sequences. In other words: for every $\nu$-computable sequence\footnote{A computable sequence is a $(\text{id}_{\mathbb{N}},\nu)$-computable
function $g:\mathbb{N}\rightarrow X$. There is only only one notion
of computable sequence, whether working in Type 2, Markov or Banach-Mazur
computability.} $(u_{n})_{n\in\mathbb{N}}$, the sequence $(f(u_{n}))_{n\in\mathbb{N}}$
is computable. 

Markov computable functions are Banach-Mazur computable. Proving that
the converse does not hold is difficult, two examples were given,
one by Friedberg \cite{Friedberg1958a} (for a function $\text{Tot}\rightarrow\mathbb{N}$),
one by Hertling in \cite{Hertling2005} (for a function $\mathbb{R}_{c}\rightarrow\mathbb{R}_{c}$
defined on the computable reals), and Bauer and Simpson have shown
in \cite{Bauer2004} how to extend these results to computable Polish
spaces that are ``effectively without isolated points''. 

\subsection{Choosing a framework}

It is by now widely accepted that Type 2 computability is the better
approach to computable analysis. In particular, it can be used to
study computability on sets with cardinality that of the continuum,
without having to restrict one's attention to points that have finite
descriptions. Type 2 computability literature is much more abundant
than Markov and Banach-Mazur literature. 

On the other hand, Banach-Mazur computability is by many aspects not
a good notion. Contrary to what happens in Type 2 computability and
in Type 1 computability, we cannot define a Cartesian closed category
using numbered sets as objects and Banach-Mazur computable functions
as morphisms: Banach-Mazur computable functions are not associated
to finite descriptions. Even worse: there exist some sets on which
continuously many Banach-Mazur computable functions exist. Consider
an infinite subset $A$ of $\mathbb{N}$ which contains no infinite
c.e. set. Then every function from $A$ to $\mathbb{N}$ is Banach-Mazur
computable: a computable sequence in $A$ has finitely many elements.
In this case, the notion of Banach-Mazur function is vacuous. 

In accordance, much less has been written about Banach-Mazur computability
than about Markov computability.

However, we prove the following result:

\begin{theoremalpha}	\label{thm: Theorem A Banach Mazur Iff }Let
$P$ be a property of marked groups. The following are equivalent: 
\begin{itemize}
\item There exists a finitely presented group $G$ with solvable word problem,
but where the problem of determining if a finitely generated subgroup
of $G$ has $P$ is not semi-decidable;
\item $P$ is not Banach-Mazur semi-decidable (with respect to the numbering
associated to groups given by word problem algorithms). 
\end{itemize}
\end{theoremalpha}

This result is very interesting in that it shows that Banach-Mazur
computability, despite all of its flaws, remains a valid notion that
has to be studied. 

This is the first general theorem of its kind, it deals at once with
a wide range of group properties. This result can be used in conjunction
with Markov's Lemma (Lemma \ref{lem:Markov's-Lemma-for Groups Intro}),
it then applies to all properties that are ``effectively not open''
in the space of marked groups. Such a result was asked for in \cite{Duda_2022}.
In \cite{Duda_2022}, a finitely presented group in which amenability
of subgroups is not semi-decidable was obtained thanks to an application
of results in the theory of intrinsically computable relations \cite{ASH1998167}.
The fact that we obtain necessary and sufficient conditions shows
that computable analysis methods are the correct tools to tackle such
problems. 

\subsection{Relationship between the three frameworks}

Type 2, Markov and Banach-Mazur computability are related via the
following implications:
\[
\text{Type 2 computable\ensuremath{\implies} Markov computable \ensuremath{\implies} Banach-Mazur computable}.
\]
Notice that because of this sequence of implications, it is natural
to expect that proving a function Type 2 computable is harder than
proving it Type 1 computable which should be harder than proving it
Banach-Mazur computable. However, this is not the case, and what happens
in practice is that proving a function Markov computable without in
fact proving that it is Type 2 computable is very hard. The main tool
to do this is to use Kolmogorov complexity, see Theorem \ref{thm:Hoyrup Continuous Iff }
of this introduction. Similarly, proving a function Banach-Mazur computable
without in fact proving that it is Markov computable is difficult,
this was achieved by Friedberg \cite{Friedberg1958a} and Hertling
\cite{Hertling2005}. 

Let us now look at the contrapositive implications: 
\[
\text{Not Banach-Mazur computable\ensuremath{\implies} Not Markov computable \ensuremath{\implies} Not Type 2 computable}.
\]
Looking at this sequence of implications, we expect that proving that
a function is not Banach-Mazur computable is harder than proving that
it is not Markov computable which is harder than proving that it is
not Type 2 computable. In this case, what happens is the following:
most of the time, when a function is proved not Markov computable,
it is in fact also proved not Banach-Mazur computable. The reason
for this is simple: proving that a function is not computable by using
a reduction to the halting problem amounts to finding a computable
sequence on which this function is not computable, thereby proving
that it is not Banach-Mazur computable. However, proving that a function
is not Type 2 computable is not only in theory but also in practice
easier than proving that it is not Markov computable. Indeed, Type
2 computable functions being continuous, topological arguments are
sufficient to prove that a function is not Type 2 computable. 

A typical instance of this phenomenon is the problem of recognizing
a fixed group $G$. By compactness, as soon as $G$ is infinite, there
must be a group that is different from $G$ which is adherent to the
set of markings of $G$. This is sufficient to say that the function
which is $1$ on markings of $G$ and $0$ elsewhere is not Type 2
computable. However, this argument far from proves that that same
function is not Markov computable: for it to be useful, we have to
know that in the closure of the set of markings of $G$, there is
a group, different from $G$, \emph{and which has solvable word problem.
}There is no automatic way of effectivizing this topological argument.\emph{
}Theorem \ref{thm:LEF groups not computably closed } gives another
instance for which establishing a Type 1/Banach-Mazur undecidability
result is harder than establishing its Type 2 counterpart. 

\subsection{Partial conclusion }

The conclusion of the above considerations is the following. In the
study of decision problems for groups given by their solution to the
word problem, 
\begin{itemize}
\item decidability results have to be established in terms of Type 2 computability,
\item undecidability results have to be established in terms of Banach-Mazur
computability. 
\end{itemize}
In both cases, we ask for the stronger result. We stress that obtaining
Banach-Mazur undecidability results is important because of Theorem
\ref{thm: Theorem A Banach Mazur Iff }. 

\bigskip

Despite the above conclusion, the present article is mostly written
in the context of Markov computability: this is a middle ground between
Type 2 computability and Banach-Mazur computability which provides
a good solution in order to not work with several frameworks at the
same time. But it will be clearly indicated throughout which undecidability
results hold for Banach-Mazur computability, and all of our decidability
results in fact hold for Type 2 computability. 

We now present in more details our results. We start by introducing
more precisely the objects we will talk about. 

\subsection{Representation and numbering associated to word problems}

For each $k$, consider a free group $\mathbb{F}_{k}$ with basis
$(x_{1},...,x_{k})$. Denote by $\theta_{k}:\mathbb{N}\rightarrow\mathbb{F}_{k}$
the bijection associated to the shortlex order on $\mathbb{F}_{k}$. 

In this context, as explained already, each marked group $(G,(s_{1},...,s_{k}))$
is determined by a unique binary expansion $(u_{n})_{n\in\mathbb{N}}\in\{0,1\}^{\mathbb{N}}$,
given by 
\[
u_{n}=1\iff\theta_{k}(n)\text{ is a relation of }(G,S).
\]
We thus define a representation $\rho_{WP}:\,\subseteq\mathbb{N}^{\mathbb{N}}\rightarrow\mathcal{G}$
of the set $\mathcal{G}$ of isomorphism classes of marked groups
as follows:
\[
\rho_{WP}((v_{n})_{n\in\mathbb{N}})=(G,(s_{1},...,s_{k}))\text{ if and only if }v_{0}=k\text{ and (\ensuremath{v_{n})_{n\ge1}} is the binary expansion of \ensuremath{(G,(s_{1},...,s_{k}))}}
\]
Let $(n,m)\mapsto\langle n,m\rangle$ denote Cantor's pairing function.
Denote by $(\varphi_{0},\varphi_{1},\varphi_{2},...)$ a standard
enumeration of all partial computable functions \cite{Rogers1987}.
We now define the numbering induced by the representation $\rho_{WP}$,
denoted $\nu_{WP}$, and which will interest us throughout. 

We define $\nu_{WP}$ by 
\[
\nu_{WP}(\langle k,i\rangle)=(G,S)\iff(k,\varphi_{i}(0),\varphi_{i}(1),\varphi_{i}(2)...)\in\mathbb{N}^{\mathbb{N}}\text{ is a \ensuremath{\rho_{WP}}-name of }(G,(s_{1},...,s_{k})).
\]
We denote by $\mathcal{G}^{+}$ the set of marked groups with solvable
word problem. 

\subsection{Topology of the space of marked groups}

As we have seen in the previous paragraph, the set of $k$-marked
groups can be seen as a subset of $\{0,1\}^{\mathbb{N}}$. It thus
inherits the product topology of the Cantor space. This topology comes
from the ultrametric distance $d$ given by: for $(u_{n})_{n\in\mathbb{N}}$
and $(v_{n})_{n\in\mathbb{N}}$ elements of $\{0,1\}^{\mathbb{N}}$,
let $n_{0}=\inf(n,u_{n}\ne v_{n})\in\mathbb{N}\cup\{+\infty\}$, and
put $d((u_{n})_{n\in\mathbb{N}},(v_{n})_{n\in\mathbb{N}})=2^{-n_{0}}$. 

We denote $\mathcal{G}_{k}$ the topological space defined this way,
and $\mathcal{G}$ the disjoint union of the $\mathcal{G}_{k}$, $k\ge1$,
equipped with the disjoint union topology. The topology of $\mathcal{G}$
is also metrizable, as the distance $d$ can be extended to $\mathcal{G}$
by imposing that groups marked by families of different cardinalities
be far away\footnote{Note that we do not embed $\mathcal{G}_{n}$ into $\mathcal{G}_{n+1}$
by identifying a marking with the marking obtained by adding the identity
as a redundant generator, as is usually done, see for instance \cite{Champetier2005}.
This is inconsequential.}, say at distance exactly $2$. 

The topology of the space of marked groups admits a natural basis.
If $(r_{1},...,r_{m};\,s_{1},...,s_{m'})$ is a pair of tuples of
elements of $\mathbb{F}_{k}$, we denote by $\Omega_{r_{1},...,r_{m};s_{1},...,s_{m'}}^{k}$
the set of $k$-marked groups in which $(r_{1},...,r_{m})$ are indeed
relations while $(s_{1},...,s_{m'})$ are not relations. Sets of this
form are called the \emph{basic clopen sets.}

We call a group an \emph{abstract group }when we want to emphasize
the fact that it is not a marked group. 

\subsection{A non-computably Polish space}

The space of marked groups is a Polish space, i.e. it has a complete
metric and is separable. It is also $\sigma$-compact, as the set
of $k$-marked groups is compact. 

By equipping the space of marked groups with a representation and
an induced numbering, we have equipped it with a notion of computability.
We can then ask which of the above facts hold effectively: is the
space of $k$-marked groups computably compact? Is its metric computable?
Is it computably separable? Is it computably complete? 

In fact, each of these properties is easily seen to hold computably,
except for separability: the space of marked groups is not computably
separable. 

Let us first quickly define the effective terms above, precise definitions
are found in Section \ref{sec:Main-results-in}: 
\begin{itemize}
\item A represented set has a \emph{computable metric} if there is an algorithm
that, given the names of two points $x$ and $y$ and some $n\in\mathbb{N}$,
produces a rational approximation of $d(x,y)$ within $2^{-n}$.
\item A represented set is \emph{computably complete} if there is a program
which, given a Cauchy sequence $(x_{n})_{n\in\mathbb{N}}$ and a \emph{regulator}
for it, i.e. a function $g$ which, given $n$, produces $N$ such
that $\forall p,q>N,\,d(x_{p},x_{q})<2^{-n}$, produces a name for
its limit. 
\item A second countable space $X$ with numbered basis $(B_{i})_{i\in\mathbb{N}}$
is \emph{computably compact} if there is a program that, given a sequence
of basic sets whose union covers $X$, outputs a finite subset which
already covers $X$. (See \cite{Pauly2016} for a general definition
that does not rely on second countability.)
\item A represented space is \emph{computably separable} if it admits a
computable and dense sequence.
\end{itemize}
The above definitions are set in Type 2, the corresponding definition
in Type 1 computability are obtained by replacing represented spaces
by numbered sets. Note however that there is no convincing notion
of computable compactness that works well in Type 1 computability.
In particular, the set of computable elements of $\{0,1\}^{\mathbb{N}}$
is not closed, thus not classically compact. We could however have
expected a result such as ``a \emph{computable} union of basic clopen
sets that covers the set of computable points $\{0,1\}^{\mathbb{N}}$
must contain a finite sub-cover'', but this is contradicted by a
famous example of Kleene: there is a computable union of basic clopen
sets whose union contains all computable points of $\{0,1\}^{\mathbb{N}}$,
but whose complement is infinite. See also \cite{Bauer2012} where
constructive notions of compactness for metric spaces are discussed. 

Following \cite{Moschovakis1964}, a numbered set with a computable
metric is called \emph{a recursive metric space}. We could maybe also
call this a ``Type 1 computable metric space''. However, the commonly
used Type 2 notion of \emph{computable metric space} \cite{Brattka2003,Iljazovic2021}
asks strictly more than a computable metric: a computable metric space
is a computably separable subset of a computably Polish space.
\begin{lem}
The represented space $(\mathcal{G},\rho_{WP})$ has a computable
metric, it is computably complete, and $\mathcal{G}_{k}$ is a computably
compact subset of it.
\end{lem}

The following lemma gives the corresponding Type 1 statement, which
is an immediate consequence of the above. 
\begin{lem}
The triple $(\mathcal{G}^{+},d,\nu_{WP})$ is a computably complete
recursive metric space. 
\end{lem}

The Boone-Novikov Theorem, which states that there exists a finitely
presented group with unsolvable word problem, easily translates to
the following: 
\begin{thm}
[Boone-Novikov,  reformulated]\label{thm:No--computable-sequence Type 2}No
$\rho_{WP}$-computable sequence is dense in $\mathcal{G}$, and thus
$(\mathcal{G},d,\rho_{WP})$ is not a computable Polish space. 
\end{thm}

Note that the above result is a Type 2 computability result. In the
vocabulary of \cite{Brecht2020}, the represented space $(\mathcal{G},\rho_{WP})$
is a precomputable quasi-Polish space which is not overt. The corresponding
result in Type 1 computability is obtained by reformulating the Boone-Rogers
Theorem \cite{Boone1966}, which says that the word problem is not
uniformly solvable on the set of finitely presented groups with solvable
word problem. 
\begin{thm}
[Boone-Rogers,  reformulated]\label{thm:No--computable-sequence Type 1}No
$\nu_{WP}$-computable sequence is dense in $\mathcal{G}^{+}$, and
thus $(\mathcal{G}^{+},d,\nu_{WP})$ is not a recursive Polish space. 
\end{thm}

Theorem \ref{thm:No--computable-sequence Type 2} relies on the fact
that the problem ``is $\Omega_{R,T}^{k}$ the empty set'' is not
co-semi-decidable, Theorem \ref{thm:No--computable-sequence Type 1}
on the fact that the problem ``is $\Omega_{R,T}^{k}\cap\mathcal{G}^{+}$
the empty set'' is not co-semi-decidable. However, we show that while
``is $\Omega_{R,T}^{k}$ the empty set'' is semi-decidable, ``is
$\Omega_{R,T}^{k}\cap\mathcal{G}^{+}$ the empty set'' is not. This
appears as Theorem \ref{thm:Use Miller Wisely -1} in the text.

\begin{theoremalpha}	

No algorithm can stop exactly on the basic clopen sets $\Omega_{R,T}^{k}$
that do not contain a group with solvable word problem. 

\end{theoremalpha}	

A common approach to studying computable Polish spaces is via notions
of ``computable presentations of Polish spaces''. The first of these
is due to Moschovakis \cite{Moschovakis1980}. There are several essentially
equivalent notions of computable presentations, see \cite{GREGORIADES2016}
for the different definitions and their equivalence. We will use the
following definition: a \emph{computable presentation} of a Polish
space $X$ with distance $d$ is a dense sequence $(u_{n})_{n\in\mathbb{N}}\in X^{\mathbb{N}}$
for which the map 
\begin{align*}
d:\, & \mathbb{N}\times\mathbb{N}\rightarrow\mathbb{R}\\
 & (n,m)\mapsto d(u_{n},u_{m})
\end{align*}
is computable. In other words, there is an algorithm that, given $(n,m,p)\in\mathbb{N}^{3}$,
produces a rational approximation of $d(u_{n},u_{m})$ within $2^{-p}$.
The use of presentations can be understood as a means of defining
computable Polish spaces without going through the steps of defining
spaces with a computable metric, computably complete and computably
separable: all this is done in one single concise step. This approach
is thus not adapted to our setting, since the space of marked groups,
while being not computably separable, has a computable metric which
is computably complete: we are making distinctions that are finer
than what notions of computable presentations would allow us to do. 

Let us quote a sentence of Moschovakis from \cite[Chapter 3]{Moschovakis1980},
who, after defining a recursive presentation of a Polish space, states:
\begin{onehalfspace}
\begin{center}
``Not every Polish space admits a recursive presentation---but every
interesting space certainly does.''
\par\end{center}
\end{onehalfspace}

We prove in Theorem \ref{thm:NoRecPres}: 

\begin{theoremalpha}	

\label{thm: No Polish Presentation}The space of marked groups $\mathcal{G}$
associated to its ultrametric distance $d$ does not have a recursive
presentation. 

\end{theoremalpha}	

Moschovakis' statement should be understood as the belief that the
only examples of Polish spaces that do not admit recursive presentations
will be artificially built counterexamples, while the Polish spaces
that occur naturally in mathematics, and whose definitions do not
involve Turing machines, always have those presentations. Thus the
fact that a naturally arising Polish space exists for which this fails
is very interesting. 

The difference between Theorem \ref{thm: No Polish Presentation}
and Theorems \ref{thm:No--computable-sequence Type 2} and \ref{thm:No--computable-sequence Type 1}
is that for this last result we do not demand a priori that groups
should be given by their solutions to the word problem. We do not
know how to obtain the above result with respect to any metric that
induces the topology of the space of marked groups, so as to have
a result that concerns the space of marked groups purely as a topological
space. Note however that this question has no bearing on the present
study: whether or not the topological space of marked groups has a
certain representation $\rho$ that makes it computably Polish has
no consequence at all on the algorithmic problem ``what can be said
about a group, given a solution to its word problem''. The representation
$\rho_{WP}$ is the only \emph{realistic} one, in the sense that it
is the only one which tells us what can on cannot be performed on
actual computer algebra systems like GAP or Magma. 

\subsection{Continuity problems on the space of marked groups }

We will now discuss what we know about continuity and effective continuity
of computable functions on the space of marked groups. 

A function $f:X\rightarrow Y$ between numbered sets with computable
metrics is called \emph{effectively metric continuous}\footnote{Note that in the context of recursive Polish spaces, this notion was
shown by Moschovakis \cite[Theorem 11]{Moschovakis1964} to be equivalent
to the one obtained by effectivizing the idea that the preimage of
an open set is open. In general it could be weaker. This is investigated
in detail in \cite{Rauzy2023}. } if there is a program that, given $x\in X$ and $n\in\mathbb{N}$,
produces $p\in\mathbb{N}$ such that if $d(y,x)<2^{-p}$, $d(f(x),f(y))<2^{-n}$. 

\subsection*{Type 2 computability }

In Type 2 computability, computable functions are by definition effectively
continuous, and thus this remains true on the space of marked groups. 

\subsection*{Markov computability }

We ask both whether Markov computable functions are continuous or
effectively metric continuous on the space of marked groups. 

Markov computable functions are effectively metric continuous on Type
1 Computable Polish spaces by Ceitin's theorem, but by Theorem \ref{thm:No--computable-sequence Type 1}
this theorem does not apply to $(\mathcal{G}^{+},d,\nu_{WP})$. 

In \cite{Moschovakis1964}, Moschovakis proved a generalization of
Ceitin's Theorem, which gives, as of today, the weakest known conditions
on a recursive metric space, that are sufficient in order for the
functions defined on this space to be effectively metric continuous.
The hypotheses of this theorem are detailed in Section \ref{subsec:KLS-Ceitin-Theorem,}. 

We however prove that the hypotheses of Moschovakis' Effective Continuity
Theorem are not satisfied by $(\mathcal{G}^{+},d,\nu_{WP})$. This
a consequence of the following result (which appears as Theorem \ref{thm:No-Completion-Theorem}
in the text):

\begin{theoremalpha}	[Failure of an Effective Axiom of Choice for
groups] No algorithm can, on input a basic clopen set $\Omega_{r_{1},...,r_{m};s_{1},...,s_{m'}}^{k}$
such that $\Omega_{r_{1},...,r_{m};s_{1},...,s_{m'}}^{k}\cap\mathcal{G}^{+}\neq\emptyset$,
produce a word problem algorithm for a group in this basic clopen
set. 

\end{theoremalpha}	

The following theorem, given to us by Mathieu Hoyrup, gives necessary
and sufficient conditions for existence of a discontinuous Markov
computable function defined on the space of marked groups. It is based
on results of \cite{Hoyrup2016} and appears as Theorem \ref{thm:Continuity IFF Kolmogorov }
in the text. 
\begin{thm}
\label{thm:Hoyrup Continuous Iff }There exists a $\nu_{WP}$-computable
and discontinuous function $f:\mathcal{G}^{+}\rightarrow\{0,1\}$
if and only if there is a marked group $(G,S)$ with solvable word
problem, which is not isolated in $\mathcal{G}^{+}$, and a computable
function $g:\mathbb{N}\rightarrow\mathbb{N}$ such that for each $n\in\mathbb{N}$
bigger than the smallest $\nu_{WP}$-name of $(G,S)$, $(G,S)$ is
the only marked group in $B((G,S),2^{-g(n)})$ to have a $\nu_{WP}$-name
inside $\{0,1,...,n\}.$

Or again: 
\[
\forall n\in\text{dom}(\nu_{WP}),\,\nu_{WP}(n)\neq(G,S)\implies d(\nu_{WP}(n),(G,S))>2^{-g(n)}.
\]
\end{thm}

In other words, $(G,S)$, while not isolated, can only be approached
by groups of very high Kolmogorov complexity. This is a strong negation
of ``$(G,S)$ is the limit of a computable sequence of marked groups
that are distinct from it''. 

\subsection*{Banach-Mazur computability }

Banach-Mazur computable functions defined on a recursive Polish space
are continuous \cite{Mazur63,Hertling2001}. 

However, even on the computable reals or on the computable points
of $\{0,1\}^{\mathbb{N}}$, continuous but not effectively metric
continuous Banach-Mazur functions can exist. In \cite{Hertling2006},
Hertling has even constructed a Banach-Mazur computable function on
the computable reals which is at no point effectively metric continuous.
We extend this to the space of marked groups by using a computable
retraction on $\{0,1\}^{\mathbb{N}}$. The following appears as Corollary
\ref{cor:BM not Markov on G}:

\begin{propalpha}	There exists a Banach-Mazur computable function
$\mathcal{G}^{+}\rightarrow\mathbb{N}$ which is continuous but not
effectively continuous. \end{propalpha}	

Whether or not Banach-Mazur computable functions must be continuous
on the space of marked groups remains an open problem. There is no
known characterization similar to that of Theorem \ref{thm:Hoyrup Continuous Iff }
for continuity of Banach-Mazur computable functions. It is not clear
whether the methods of \cite{Hoyrup2016} can be applied to Banach-Mazur
computability. 

Note also that there are countably many Banach-Mazur computable functions
$\mathbb{R}_{c}\rightarrow\mathbb{R}_{c}$: by continuity, each such
function is determined by its behavior on $\mathbb{Q}$, and because
$\mathbb{Q}$ can be computably enumerated, each Banach-Mazur computable
function $f$ on $\mathbb{Q}$ has a finite description: the code
of a machine that produces a computable enumeration of $f(\mathbb{Q})$.
On the other hand, if $A$ is an infinite subset of $\mathbb{N}$
with no infinite c.e. subset, there are continuously many Banach-Mazur
computable functions $f:A\rightarrow\mathbb{N}$. What of Banach-Mazur
functions of the space of marked groups? 

\subsection*{Summary }

We summarize the situation in the following table: 

 	\begin{table}[h] 		\centering 		\begin{tabular}{|c|c|c|c|c|c|c|c|c|c|} 			\hline 	Notion of computability:	&	\multicolumn{3}{|c|}{\textbf{Banach-Mazur}}&	\multicolumn{3}{|c|}{\textbf{Markov}}&	\multicolumn{3}{|c|}{\textbf{Type 2}} 			\\ \hline 			&$C^0$ & Eff. $C^0$ & Cardinality &$C^0$ & Eff. $C^0$ & Cardinality &$C^0$ & Eff. $C^0$ & Cardinality  				\\ \hline 			Arbitrary set &$\times$&$\times$&Can be $2^{\aleph_0}$&$\times$&$\times$&$\aleph_0$&\checkmark&\checkmark&$\aleph_0$ 				\\ \hline 			Space of marked groups &?&$\times$&?&?&?&$\aleph_0$&\checkmark&\checkmark&$\aleph_0$ 							\\ \hline 			Computable Polish space &\checkmark &$\times$&$\aleph_0$&\checkmark&\checkmark&$\aleph_0$&\checkmark&\checkmark&$\aleph_0$ 							\\ \hline
	\end{tabular} \end{table}

The main problem arising from our study is thus the following one:

\begin{problemalpha}	[Main Problem]\label{conj:Main-1} Are Banach-Mazur
computable functions defined on $\mathcal{G}^{+}$ continuous? How
many of these are there? Are Markov computable functions defined on
$(\mathcal{G}^{+},\nu_{WP})$ continuous or effectively continuous? 

\end{problemalpha}	

A very interesting consequence of a positive answer to the first question
would be that it would be possible to apply Theorem \ref{thm: Theorem A Banach Mazur Iff }
to any non-clopen property of $\mathcal{G}^{+}$. 

\subsection{Markov's Lemma }

As discussed above, the more advanced continuity theorems of Type
1 and of Banach-Mazur computable analysis cannot be applied to the
space of marked groups.

However, the fact that $(\mathcal{G}^{+},d,\nu_{WP})$ is an effectively
complete recursive metric space provides us with some basic results
coming from computable analysis that can be used to obtain group theoretical
results. The most useful result in this regard is Markov's Lemma \cite[Theorem 4.2.2]{Markov1963},
as applied to the space of marked groups. 

\begin{lemmaalpha}	[Markov's Lemma for groups]\label{lem:Markov's-Lemma-for Groups Intro}If
$((G_{n},S_{n}))_{n\in\mathbb{N}}$ is a $\nu_{WP}$-computable sequence
of marked groups that converges to a marked group $(H,S)$ with solvable
word problem, with $(G_{n},S_{n})\ne(H,S)$ for each $n$, then there
is a $\nu_{WP}$-computable sequence $((\Gamma_{n},T_{n}))_{n\in\mathbb{N}}\in(\{(H,S)\}\cup\{(G_{n},S_{n}),n\in\mathbb{N}\})^{\mathbb{N}}$
such that $\{n\in\mathbb{N},\,(\Gamma_{n},T_{n})=(H,S)\}$ is not
a c.e. subset of $\mathbb{N}$. \end{lemmaalpha}	

In other words: if $(H,S)$ is the limit of a $\nu_{WP}$-computable
sequence of groups in some set $A$ with $(H,S)\notin A$, then $\{(H,S)\}$
is not a Banach-Mazur $\nu_{WP}$-semi-decidable subset of $\{(H,S)\}\cup A$. 

This result is elementary, and in fact one can find in the literature
results that are obtained by ``manual'' applications of this lemma
to different effectively converging sequences. See for instance \cite{McCool_1970}
and \cite{Lockhart1981}. 

Markov's Lemma is our main tool to study decidability on the space
of marked groups. 

\subsection{Effective Borel hierarchies }

The study of the Borel hierarchy on $\mathcal{G}$ is a difficult
topic related to many deep questions. For instance the fact that the
set of groups with polynomial growth is open cannot, to the best of
our knowledge, be proved without Gromov's Theorem. And the problem
of deciding which identities of the form $\forall w,\,w^{n}=1$ define
a clopen set is related to the Burnside problem: for $n\in\{1,2,3,4,6\}$
the set of groups satisfying $\forall w,\,w^{n}=1$ is clopen, and
for $n\gg1$ the construction of infinite Burnside groups shows that
it is closed but not open. The remaining cases are not classified. 

Some results about the Borel hierarchy on $\mathcal{G}$ were obtained
in \cite{Benli2019}, but, again, many group properties remain unclassified. 

By using Type 2 computability and the representation of $\mathcal{G}$,
we can refine its Borel hierarchy. Indeed, a set is open in $\mathcal{G}$
if and only if it is $\rho_{WP}$-semi-decidable modulo ``some oracle''.
In concrete examples, it is natural to ask which oracle is required.
 (Note that, by Theorem \ref{thm: No Polish Presentation}, Moschovakis'
approach to defining an effective Borel hierarchy \cite{Moschovakis1980}
is not applicable to the space of marked groups.)

Because of Theorem \ref{thm: Theorem A Banach Mazur Iff }, we must
also investigate Markov and Banach-Mazur computability. This is done
by restricting our attention to the set $\mathcal{G}^{+}$ of groups
with solvable word problem, and by using the numbering $\nu_{WP}$
instead of the representation $\rho_{WP}$. 

And thus each group property has four classifications: in the classical
Borel hierarchy, in the effective Borel hierarchy for $\rho_{WP}$,
in terms of the arithmetical hierarchy with respect to $\nu_{WP}$
(Type 1 classification), and in terms of a possible ``Banach-Mazur
arithmetical hierarchy''. It seems that such an object has yet to
be defined, but this is not a problem for us since we will only consider
the first levels of such a hierarchy: semi-decidable, co-semi-decidable,
and both or neither of these. 

Properties are classified as clopen, open, closed, or ``above'',
and similarly, properties are classified as decidable, semi-decidable,
co-semi-decidable, or ``above''. Note that the main reason why we
restrict our attention to these levels is that we are \emph{unaware
of differences (or conjectured differences) between the Borel hierarchy
and its effective counterparts that occur above the very first levels.
}We know that the set of LEF groups is a closed but not computably
closed set, and we have several more candidates of such properties,
but we are not aware of a natural $G_{\delta}$ which would not be
computably a $G_{\delta}$, and so on.

\[ \begin{tabular}{|c|c|c|c|} 	\hline 	 Borel hierarchy& Effective Borel hierarchy on $\mathcal{G}$& Type 1 hierarchy on $\mathcal{G}^{+}$& Banach-Mazur hierarchy on $\mathcal{G}^{+}$ 	\\ 	\hline 	\text{Clopen} &  \text{\ensuremath{\rho_{WP}}-decidable}  	& \text{\ensuremath{\nu_{WP}}-decidable}  &\text{Banach-Mazur \ensuremath{\nu_{WP}}-decidable} \\ 	\text{Open} &  \text{\ensuremath{\rho_{WP}}-semi-decidable}  &\text{\ensuremath{\nu_{WP}}-semi-decidable}  &\text{Banach-Mazur \ensuremath{\nu_{WP}}-semi-decidable} \\ 	\text{Closed} &  \text{\ensuremath{\rho_{WP}}-co-semi-decidable}  &\text{\ensuremath{\nu_{WP}}-co-semi-decidable}  &\text{Banach-Mazur \ensuremath{\nu_{WP}}-co-semi-decidable}\\ \hline \end{tabular} \]
\medskip

In most explicit cases, the four classifications will perfectly coincide.
We thus expect a natural property $P\subseteq\mathcal{G}$ to be open
if and only if it is $\rho_{WP}$-semi-decidable if and only if it
is $\nu_{WP}$-semi-decidable if and only if it is Banach-Mazur $\nu_{WP}$-semi-decidable.
Notice that establishing such correspondences will also involve proving
negative results: the correspondence works only if a non-open property
is also not $\nu_{WP}$-semi-decidable. 

The implications that are known to hold between the four classifications
are represented below: 
\begin{align*}
\text{clopen}\impliedby\text{\ensuremath{\rho_{WP}}-decidable} & \implies\text{\ensuremath{\nu_{WP}}-decidable}\implies\text{Banach-Mazur \ensuremath{\nu_{WP}}-decidable}\\
\text{open}\impliedby\text{\ensuremath{\rho_{WP}}-semi-decidable} & \implies\text{\ensuremath{\nu_{WP}}-semi-decidable}\implies\text{Banach-Mazur \ensuremath{\nu_{WP}}-semi-decidable }\\
\text{closed}\impliedby\text{\ensuremath{\rho_{WP}}-co-semi-decidable} & \implies\text{\ensuremath{\nu_{WP}}-co-semi-decidable}\implies\text{Banach-Mazur \ensuremath{\nu_{WP}}-co-semi-decidable }
\end{align*}

Partial solutions to Problem \ref{conj:Main-1} could add more implications.
Indeed, a non-effective continuity theorem, either for Banach-Mazur
or for Markov-computable functions, would permit to establish the
corresponding implication:
\[
\text{Banach-Mazur \ensuremath{\nu_{WP}}-decidable}\implies\text{clopen in }\text{\ensuremath{\mathcal{G}^{+}}},
\]
\[
\text{\ensuremath{\nu_{WP}}-decidable}\implies\text{clopen in }\text{\ensuremath{\mathcal{G}^{+}}}.
\]
An effective continuity theorem for Markov-computable functions would
permit to establish the implication:
\[
\text{\ensuremath{\nu_{WP}}-decidable}\implies\text{\ensuremath{\rho_{WP}}-decidable in }\text{\ensuremath{\mathcal{G}^{+}}}.
\]
There are no other actual implications between the four notions. (In
particular, establishing a Markov effective continuity result would
\emph{not} imply that a $\nu_{WP}$-semi-decidable property is open
in $\mathcal{G}^{+}$, this is known to be false.) 

However, there is the following heuristic:
\begin{itemize}
\item A natural $\nu_{WP}$-semi-decidable property is (\emph{very much})
expected to be $\rho_{WP}$-semi-decidable and open in $\mathcal{G}^{+}$,
and a natural $\nu_{WP}$-co-semi-decidable property is (\emph{very
much}) expected to be $\rho_{WP}$-co-semi-decidable and closed in
$\mathcal{G}^{+}$.
\end{itemize}
Semi-decidable properties that are not open do exist, but they are
built using Kolmogorov complexity, those are sets one does not expect
to run into when dealing with properties defined by algebraic or geometric
constructions. (See Section \ref{subsec:Differences-with-Borel} for
an example.)
\begin{itemize}
\item A natural open property in $\mathcal{G}$ is (\emph{a little}) expected
to be $\rho_{WP}$-semi-decidable, and a natural closed property in
$\mathcal{G}$ is (\emph{a little}) expected to be $\rho_{WP}$-co-semi-decidable.
\end{itemize}
This second fact is rather of an empirical nature, and is justified
by the results given in Section \ref{subsec:The Table}. Indeed,
there is given a table containing around 30 group properties for which
we prove that the four classifications coincide perfectly. Theorem
\ref{thm: Theorem A Banach Mazur Iff } can be used with all those
that are not clopen/decidable. 

\subsection{Differences between the Borel hierarchy and its effective counterpart}

The main interest of studying effective versions of the Borel hierarchy
on the space of marked groups lies in the fact that there are natural
group properties which, while open, are not computably open, or while
closed, are not computably closed. 

We in fact know of only a single example of such a group property,
but we have several other candidates. 

The set of LEF groups (for Locally Embeddable into Finite groups)
is the closure in $\mathcal{G}$ of the set of finite groups. It was
introduced in \cite{VershikGordon1998}. 

\begin{theoremalpha}	\label{thm:LEF groups not computably closed }The
set of LEF groups is a closed set which is not computably closed,
in the sense that its complement cannot be written as a computable
union of basic clopen sets. However, it is computably a $G_{\delta}$.
\end{theoremalpha}	

The main ingredient of the proof of this theorem is Slobodskoi's Theorem
stating that the universal theory of finite groups is undecidable
\cite{Slobodskoi1981}. 

Remark the following. By definition, a marked group $(G,S)$ is LEF
if every ball in the labeled Cayley graph of $(G,S)$ is also a ball
in the labeled Cayley graph of a certain finite group. This condition
is naturally expressed by a $\forall\exists$ statement:
\[
\forall n\in\mathbb{N},\exists(F,S')\text{ finite marked group,}\,B_{(G,S)}(n)=B_{(F,S')}(n).
\]
Upon closer examination, one realizes that the existential quantifier
above is a bounded one: there are finitely many labeled graphs of
diameter at most $2n$ and with edges having degree at most $\vert S\vert$,
some of them are a part of the Cayley graph of a finite group, others
are not, and there is a certain number $g(n)$ which is the least
number such those that do belong to some finite group belong to a
finite group of size at most $g(n)$. And thus the existential quantifier,
being bounded, defines a finite union of clopen sets, itself clopen,
the universal quantification then gives an intersection of clopen
sets, which is thus closed. 

However, Slobodskoi's Theorem implies, in particular, that the function
$g$ defined above cannot be computable (or even majored by a computable
function). And thus from the point of view of effective mathematics,
the $\forall\exists$ statement that defines the set of LEF groups
cannot be reduced to a single $\forall$ statement. The fact that
$g$ cannot be majored by a computable function was remarked by Bradford
in \cite{Bradford2022}. 

We want to confront Theorem \ref{thm:LEF groups not computably closed }
with the following problems: 
\begin{problem}
[Gromov]Is there a non-LEF hyperbolic group? 
\begin{problem}
[Zelmanov]Are infinite free Burnside groups LEF?
\end{problem}

\end{problem}

The computably closed properties are exactly those which it is algorithmically
possible to refute. Theorem \ref{thm:LEF groups not computably closed }
can thus be seen as explaining why these two problems are so difficult.

Note finally that Theorem \ref{thm:LEF groups not computably closed }
is a Type 2 result: it says that the set of LEF groups is not $\rho_{WP}$-co-semi-decidable.
We do not know how to obtain the corresponding Type 1 or Banach-Mazur
results: 

\begin{problemalpha}	The set of LEF groups with solvable word problem
is not $\nu_{WP}$-co-semi-decidable nor Banach-Mazur $\nu_{WP}$-co-semi-decidable.
\end{problemalpha}	

Other candidates for properties that would be open/closed but not
computably so are given in the following conjecture: 

\begin{problemalpha}	The sets of finitely presented simple groups
and of isolated groups are open but not computably open, and also
not $\nu_{WP}$-semi-decidable or Banach-Mazur $\nu_{WP}$-semi-decidable. 

The closure of the set of hyperbolic groups, the closure of the set
of finite nilpotent groups are closed but not computably closed, and
also not $\nu_{WP}$-co-semi-decidable nor Banach-Mazur $\nu_{WP}$-co-semi-decidable.
\end{problemalpha}	

The set of sofic groups could also be a closed but not computably
closed subset of $\mathcal{G}$. 

\subsection{Contents of this paper}

In Section \ref{sec:The-topological-space}, we describe the space
of marked groups, and give a few results related to computability:
impossibility of deciding whether or not a basic clopen set is empty,
impossibility of deciding whether a basic clopen set contains a group
with solvable word problem. 

In Section \ref{sec:Vocabulary-about-numberings}, we fix the vocabulary
about numberings that is required to present concepts from Type 1
and Banach-Mazur computable analysis. 

In Section \ref{sec:The-word-problem-Numbering-Type}, we give several
equivalent definitions that formalize the concept that a group is
described by a word problem algorithm, or by its labeled Cayley graph. 

In Section \ref{sec:Main-results-in}, we describe some results of
Markovian computable analysis, giving proofs for some of them and
references for the others. We prove Markov's Lemma and Mazur's Continuity
Theorem, and quote the Kreisel-Lacombe-Schoenfield-Ceitin Theorem,
and Moschovakis' extension of this theorem. We also include examples
of computable but discontinuous functions built using Kolmogorov complexity. 

In Section \ref{sec:Limits-of-applicability}, we start investigating
the space of marked groups as a recursive metric space. We prove that
none of the continuity results given in the previous section can be
applied to the space of marked groups. 

In Section \ref{sec:First-results-for}, we give a wide range of examples
of group properties whose classical and effective Borel classifications
coincide. 

In Section \ref{sec:Two-candidates-for failure}, we prove that the
sets of LEF groups is closed but not computably closed. We propose
several other properties as possible failures of the correspondence
between the Borel hierarchy and its effective counterparts. 

In Section \ref{sec:Witnessing-results-in FP groups}, we prove Theorem
\ref{thm: Theorem A Banach Mazur Iff } and give examples of its applications.

\section{\label{sec:The-topological-space}The topological space of marked
groups}

\subsection{Definitions}

Let $k$ be natural number. A\emph{ $k$-marked group} is a finitely
generated group $G$ together with a $k$-tuple $S=(s_{1},...,s_{k})$
of elements of $G$ that generate it. We call $S$ a \emph{generating
family}. Note that repetitions are allowed in $S$, the order of the
elements matters, and $S$ could contain the identity element of $G$.
A \emph{morphism of marked groups} between $k$-marked groups $(G,(s_{1},...,s_{k}))$
and $(H,(t_{1},...,t_{k}))$ is a group morphism $\varphi$ between
$G$ and $H$ that additionally satisfies $\varphi(s_{i})=t_{i}$.
It is an \emph{isomorphism of marked groups} if $\varphi$ is a group
isomorphism. Marked groups are considered up to isomorphism. We call
a group an \emph{abstract group }when we want to emphasize the fact
that it is not a marked group. 

It is in fact convenient, when studying $k$-marked groups, to fix
a free group $\mathbb{F}_{k}$ of rank $k$, together with a basis
$S$ for $\mathbb{F}_{k}$. A $k$-marking of a group $G$ can then
be seen as an epimorphism $\varphi:\mathbb{F}_{k}\rightarrow G$,
the image of $S$ by $\varphi$ defines a marking with respect to
the previous definition. Two $k$-marked groups are then isomorphic
if they are defined by morphisms with identical kernels: the isomorphism
classes of $k$-marked groups are classified by the normal subgroups
of a rank $k$ free group. The set $S$ can be thought of as a set
of generating symbols, and we often consider that all groups are generated
by those letters. 

Remark that a word problem algorithm for a group $G$ is thus a description
of a \emph{marking} of $G$, and similarly, a presentation of a group
defines a marked group. 

We note $\mathcal{G}_{k}$ the set of isomorphism classes of $k$-marked
groups, and $\mathcal{G}$ the disjoint union of the $\mathcal{G}_{k}$. 

Note that some authors consider that each set $\mathcal{G}_{k}$ is
embedded in the set $\mathcal{G}_{k+1}$, identifying a marking $(G,(g_{1},...,g_{k}))$
with the same marking where the identity $e_{G}$ of $G$ is added
as a last generator: the $k$-marking $(G,(g_{1},...,g_{k}))$ is
identified with the $k+1$-marking $(G,(g_{1},...,g_{k},e_{G}))$.
We do not adhere to this convention, for reasons that appear clearly
in \cite[Proposition 59]{Rauzy21}: adding generators that define
the identity to a marking can change whether or not a marked group
is recognizable from finite presentations. It is thus detrimental
in the study of decision problems for groups to identify a marking
to the markings obtained by adding trivial generators. 

For an abstract group $G$, we denote $\left[G\right]_{k}$ the set
of all its markings in $\mathcal{G}_{k}$, and $\left[G\right]$ the
set of all its markings in $\mathcal{G}$ (as in \cite{Champetier2005}).
If $(G,S)$ is a marking of $G$, we also define $\left[(G,S)\right]_{k}$
and $\left[(G,S)\right]$, those are identical to $\left[G\right]_{k}$
and $\left[G\right]$ respectively. 

\subsection{Topology on $\mathcal{G}$}

We define a topology on $\mathcal{G}$ by equipping each separate
space $\mathcal{G}_{k}$ with a topology, the topology we then consider
on $\mathcal{G}$ is the disjoint union topology of the $\mathcal{G}_{k}$. 

For each $k$, consider a finite set $\{s_{1},...,s_{k}\}$, choose
arbitrarily an order on the set $\{s_{1},...,s_{k}\}\cup\{s_{1}^{-1},...,s_{k}^{-1}\}$,
and enumerate by length and then lexicographically the elements of
the free group $\mathbb{F}_{k}$ over $S$ (i.e. following the \emph{shortlex
order}). 

Denote by $\theta_{k}(n)$ the $n$th element obtained in this enumeration,
$\theta_{k}$ is thus a bijection between $\mathbb{N}$ and $\mathbb{F}_{k}$. 

To a normal subgroup $N$ of $\mathbb{F}_{k}$ we can associate its
characteristic function $\chi_{N}:\mathbb{F}_{k}\rightarrow\{0,1\}$,
and composing it with the bijection $\theta_{k}$, we obtain an element
of the Cantor space $\mathcal{C}=\{0,1\}^{\mathbb{N}}$. This defines
an embedding:
\[
\Phi_{k}:\begin{cases}
\mathcal{G}_{k} & \longrightarrow\,\,\,\,\left\{ 0,1\right\} ^{\mathbb{N}}\\
N\vartriangleleft\mathbb{F}_{k} & \longmapsto\,\,\,\,\chi_{N}\circ\theta_{k}
\end{cases}
\]
of the space of $k$-marked groups into the Cantor space. We call
the image $\Phi_{k}((G,S))$ of a $k$-marked group $(G,S)$ the \emph{binary
expansion} of $(G,S)$. 

With the Cantor set being equipped with its usual product topology,
the topology we will study on $\mathcal{G}_{k}$ is precisely the
topology induced by this embedding. 

It is easy to see that $\Phi_{k}(\mathcal{G}_{k})$ is a closed subset
of $\left\{ 0,1\right\} ^{\mathbb{N}}$ of empty interior. It is thus
compact. 

The product topology on $\left\{ 0,1\right\} ^{\mathbb{N}}$ admits
a basis which consists of clopen sets: given any finite set $A\subseteq\mathbb{N}$
and any function $f:A\rightarrow\{0,1\}$, define the set $\Omega_{f}$
by: $(u_{n})_{n\in\mathbb{N}}\in\Omega_{f}\Longleftrightarrow\forall n\in A;\,u_{n}=f(n)$.
Sets of the form $\Omega_{f}$ are clopen and form a basis for the
topology of $\left\{ 0,1\right\} ^{\mathbb{N}}$. 

Sets of the form $\mathcal{G}_{k}\cap\Omega_{f}$ thus define a basis
for the topology of $\mathcal{G}_{k}$. 

The set $\mathcal{G}_{k}\cap\Omega_{f}$ is defined as a set of marked
groups that must satisfy some number of imposed relations, while on
the contrary a fixed set of elements must be different from the identity. 

We fix the following notation. For $m$ and $m'$ natural numbers,
and elements $r_{1},...,r_{m};\,s_{1},...,s_{m'}$ of $\mathbb{F}_{k}$,
we note $\Omega_{r_{1},...,r_{m};s_{1},...,s_{m'}}^{k}$ the set of
$k$-marked groups that satisfy the relations $r_{1},...,r_{m}$,
while they do not satisfy $s_{1},...,s_{m'}$. We call $s_{1},...,s_{m'}$
\emph{irrelations. }

The sets $\Omega_{r_{1},...,r_{m};s_{1},...,s_{m'}}^{k}$ are called
the \emph{basic clopen sets. }

The following lemma from \cite{Cornulier2007} is very useful in many
situation. For instance, it implies that if $A$ is a set of marked
groups closed under group isomorphism, then so are its closure and
its interior. 
\begin{lem}
[\cite{Cornulier2007}, Lemma 1]\label{Ycor prop-1} Consider two
markings $(G,S_{1})\in\mathcal{G}_{m_{1}}$ and $(G,S_{2})\in\mathcal{G}_{m_{2}}$
of a same group. Then there are clopen neighborhoods $V_{i}$, $i=1,2$
of $(G,S_{i})$ in $\mathcal{G}_{m_{i}}$ and a homeomorphism $\phi:V_{1}\rightarrow V_{2}$
mapping $(G,S_{1})$ to $(G,S_{2})$ and preserving isomorphism classes,
i.e. such that, for every $(H,T)\in V_{1}$, $\phi((H,T))=(H,T')$
for some $T'$.
\end{lem}

In what follows, we call the set $r_{1},...,r_{m};\,s_{1},...,s_{m'}$
of relations and irrelations \emph{coherent }if $\Omega_{r_{1},...,r_{m};s_{1},...,s_{m'}}^{k}$
is not empty. The Boone-Novikov theorem, which implies that there
exists a finitely presented group with unsolvable word problem, directly
implies the following:
\begin{thm}
[Boone-Novikov reformulated]\label{thm:Boone-Novikov-reformulated}No
algorithm can decide whether or not a given finite set of relations
and irrelations is coherent. More precisely, there is an algorithm
that stops exactly on incoherent sets of relations and irrelations,
but no algorithm can stop exactly on coherent sets of relations and
irrelations. 
\end{thm}

\begin{proof}
It is always possible, given a finite set $r_{1},...,r_{m}$ of relations,
to enumerate their consequences, thus it can be found out if one of
the irrelations $s_{1},...,s_{m'}$ is a consequence of the relations
$r_{1},...,r_{m}$, in which case $r_{1},...,r_{m};\,s_{1},...,s_{m'}$
is incoherent. 

A finite set of relations and irrelations $r_{1},...,r_{m};\,s_{1},...,s_{m'}$
is coherent if and only if $\Omega_{r_{1},...,r_{m};s_{1},...,s_{m'}}^{k}$
is non-empty, but $\Omega_{r_{1},...,r_{m};s_{1},...,s_{m'}}^{k}$
is non-empty if and only if it contains the finitely presented group
$\langle S\,\vert\,r_{1},...,r_{m}\rangle$, i.e. if and only if the
relations $s_{1},...,s_{m'}$ are not satisfied by $\langle S\,\vert\,r_{1},...,r_{m}\rangle$. 

Thus solving the word problem in $\langle S\,\vert\,r_{1},...,r_{m}\rangle$
is equivalent to producing an algorithm that stops on non-empty basic
clopen sets of the form $\Omega_{r_{1},...,r_{m};w}^{k}$, as $w$
varies. By the Boone-Novikov Theorem there exists a group for which
no such algorithm exists. 
\end{proof}
We will call a set $r_{1},...,r_{m};\,s_{1},...,s_{m'}$ of relations
and irrelations \emph{word problem coherent}, or \emph{wp-coherent},
if the basic clopen set $\Omega_{r_{1},...,r_{m};s_{1},...,s_{m'}}^{k}$
contains a group with solvable word problem. The remarkable fact that
the notions of coherence and wp-coherence differ follows from a theorem
of Miller \cite[Corollary 3.9]{Miller1992} which we present in details
in Section \ref{subsec:Two Applications of Miller } (see Theorem
\ref{thm:Miller,}). 

The fact that coherence and wp-coherence differ can be equivalently
formulated as: ``groups with solvable word problem are not dense
in $\mathcal{G}$''. 

This remark calls for the following theorem:
\begin{thm}
[Boone-Rogers reformulated] \label{thm:Boone-Rogers-reformulated-}No
algorithm can stop exactly on wp-coherent sets of relations and irrelations. 
\end{thm}

\begin{proof}
This follows from the Boone-Rogers theorem \cite{Boone1966} which
states that there is no uniform solution to the word problem on the
set of finitely presented groups with solvable word problem. Indeed,
an effective way of recognizing wp-coherent sets of relations and
irrelations would provide a uniform algorithm for the word problem
on finitely presented groups with solvable word problem, by an argument
similar to that of Theorem \ref{thm:Boone-Novikov-reformulated}. 
\end{proof}
We show in Theorem \ref{thm:Use Miller Wisely -1} that detecting
wp-coherence is strictly more complex than detecting coherence: it
is neither semi-decidable nor co-semi-decidable. 

We finally include a lemma which will be useful in Section \ref{sec:Two-candidates-for failure}:
\begin{lem}
\label{lem:Inclusion-between-basic sets is semi-decidable}Inclusion
between basic clopen sets is semi-decidable. And inclusion between
finite unions of basic clopen sets is semi-decidable. 
\end{lem}

This means that there is an algorithm that takes as inputs two finite
sets of tuples of relations and irrelations, and stops if and only
if the union of the basic clopen sets defined by the first set is
a subset of the union of the basic clopen sets defined by the second
set. 
\begin{proof}
This lemma is in fact a simple corollary of the fact that emptiness
of basic open sets is semi-decidable. We first give an example. 

Consider two generators $a$ and $b$. To determine whether the basic
open set $\Omega_{(ab)^{2}\ne1}$ is a subset of $\Omega_{b^{2}=1}$,
we write both basic open sets as disjoint unions of basic open sets
with the same elements of the free groups appearing as relations and
irrelations:
\[
\Omega_{(ab)^{2}\ne1}=\Omega_{b^{2}=1;\,(ab)^{2}\ne1}\sqcup\Omega_{b^{2}\ne1,\,(ab)^{2}\ne1},
\]
\[
\Omega_{b^{2}=1}=\Omega_{b^{2}=1;\,(ab)^{2}\ne1}\sqcup\Omega_{b^{2}=1,\,(ab)^{2}=1}.
\]

It is then clear that $\Omega_{(ab)^{2}\ne1}\subseteq\Omega_{b^{2}=1}$
if and only if $\Omega_{(ab)^{2}\ne1,\,b^{2}\ne1}$ is empty. 

In general, given two tuples of basic open sets, define the total
support to be the set of all elements of the free group that appear
either as relations or irrelations in one of these basic open sets.
Decompose all the basic open sets as disjoint union as above. Inclusion
of the union of the first tuple inside the second tuple is then equivalent
to emptiness of a certain finite number of basic sets, which is semi-decidable
by Theorem \ref{thm:Boone-Novikov-reformulated}. 
\end{proof}

\subsection{Different distances }

The topology defined above on the space of marked groups is metrizable.
We describe here two possible distances which generate this topology,
and which we may most of the time use interchangeably throughout this
paper. 

Those distances are defined on each $\mathcal{G}_{k}$ separately,
the distances between groups marked by generating families that have
different cardinalities can be fixed arbitrarily, as long as there
remains a strictly positive lower bound to those distances. For convenience,
we will adopt throughout the convention that the distance between
groups marked by families with different cardinalities is exactly
$2$. 

\subsubsection{Ultrametric distance}

For sequences $(u_{n})_{n\in\mathbb{N}}$ and $(v_{n})_{n\in\mathbb{N}}$
in $\left\{ 0,1\right\} ^{\mathbb{N}}$ that are different, denote
$n_{0}$ the least number for which $u_{n_{0}}\ne v_{n_{0}}$, and
set $d((u_{n})_{n\in\mathbb{N}},(v_{n})_{n\in\mathbb{N}})=2^{-n_{0}}$.
If the sequences $(u_{n})_{n\in\mathbb{N}}$ and $(v_{n})_{n\in\mathbb{N}}$
are equal, set $d((u_{n})_{n\in\mathbb{N}},(v_{n})_{n\in\mathbb{N}})=0$. 

This defines an ultrametric distance on $\left\{ 0,1\right\} ^{\mathbb{N}}$
which generates its topology, and which induces a distance on $\mathcal{G}_{k}$
via the embedding $\mathcal{G}_{k}\hookrightarrow\{0,1\}^{\mathbb{N}}$. 

\subsubsection{Cayley Graph distance}

Yet another way to define a distance on $\mathcal{G}_{k}$ is by using
labeled Cayley graphs. A labeled Cayley graph defines uniquely a marked
group. And a word problem algorithm can be seen as an algorithm that
produces arbitrarily large (finite) portions of the labeled Cayley
graph of a given group, this is explained precisely in Section \ref{sec:The-word-problem-Numbering-Type}.
We can define a new distance $d_{\text{Cay}}$ as follows. 

For two $k$-marked groups $(G,S)$ and $(H,S')$, consider the respective
labeled Cayley graphs, $\Gamma_{G}$ and $\Gamma_{H}$. The natural
bijection between $S$ and $S'$ induces a relabeling of $\Gamma_{G}$,
denoted $\Gamma_{G}'$, replacing edge labels which belong to $S$
by the corresponding elements in $S'$. The balls centered at the
identity in $\Gamma_{G}'$ and $\Gamma_{H}$ agree up to a certain
radius, call $r$ the radius for which the balls of radius $r$ of
$\Gamma_{G}'$ and $\Gamma_{H}$ are identical, while their balls
of radius $r+1$ differ. Then put $d_{\text{Cay}}((G,S),(H,S))=2^{-r}$.
If $\Gamma_{G}'$ and $\Gamma_{H}$ are identical, $r$ is infinite,
and we put $d_{Cay}((G,S),(H,S))=0$. It is easy to check that $d_{\text{Cay}}$
is an ultrametric distance which induces the topology of the space
of marked groups. 

The distance $d_{\text{Cay}}$ could be preferred to $d$, as it brings
a visual dimension to proofs, however in most cases it is just less
precise than $d$: the only difference between $d$ and $d_{\text{Cay}}$
is that, in the computation of $d_{\text{Cay}}$, the relations are
considered ``in packs'', corresponding to their length as elements
of the free groups, while they are considered each one by one when
using $d$. The choice of an order of the free group is precisely
what allows to give more or less importance to relations of the same
length.

\section{\label{sec:Vocabulary-about-numberings}Vocabulary about numberings }

\subsection{Numberings, subnumberings and numbering types }

\subsubsection{Paring functions }

Throughout, we denote by $(n,m)\mapsto\langle n,m\rangle$ Cantor's
computable bijection between $\mathbb{N}^{2}$ and $\mathbb{N}$.
We denote by $\text{fst}$ and $\text{snd}$ functions that define
its inverse: for all $n$ and $m$ natural numbers, we have: $\text{fst}(\langle n,m\rangle)=n$
and $\text{snd}(\langle n,m\rangle)=m$. We extend this bijection
to tuples by the formula $\langle n_{1},...,n_{m}\rangle=\langle n_{1},\langle n_{2},\langle...,n_{m}\rangle...\rangle$. 

\subsubsection{First definitions }

We will now introduce numbered spaces and numbering types, which will
be useful throughout this paper. For more details, see the chapter
on numberings in Weihrauch's book \cite{Weihrauch1987}. The expressions
``numbering type'' and ``subnumbering'' are ours. 
\begin{defn}
Let $X$ be a set. A \emph{numbering} of $X$ is a surjective function
$\nu$ that maps a subset $A$ of $\mathbb{N}$ onto $X$. A \emph{subnumbering}
of $X$ is a numbering of a subset of $X$. We denote this by: $\nu:\,\subseteq\mathbb{N}\rightarrow X$.
\end{defn}

The pair $(X,\nu)$ is a \emph{subnumbered set}. The domain of $\nu$
is a subset of $\mathbb{N}$ denoted by $\text{dom}(\nu)$. The set
of all subnumberings of $X$ is denoted $\mathcal{N}_{X}$. 

The image $\nu(\text{dom}(\nu))$ of $\nu$ is called the set of $\nu$-computable
points of $X$, and denoted $X_{\nu}$. Given a point $x$ in $X$,
an integer $n$ such that $\nu(n)=x$ is called a \emph{$\nu$-name}
of $x$. 
\begin{defn}
\label{def:MARKOV COMP-1}Let $(X,\nu)$ and $(Y,\mu)$ be subnumbered
spaces. A function $f:X\rightarrow Y$ is called $(\nu,\mu)$-computable
if there exists a partial recursive function $F:\,\subseteq\mathbb{N}\rightarrow\mathbb{N}$
such that for all $n$ in the domain of $\nu$, $f\circ\nu(n)=\mu\circ F(n)$.
That is to say, there exists $F$ which renders the following diagram
commutative:

\begin{center}
\begin{tikzcd}
X \arrow[r, "f"]& Y \\ \mathbb{N}\arrow[u, "\nu"]\arrow[r, "F"]& \mathbb{N}\arrow[u, "\mu"]\end{tikzcd}
\end{center}
\end{defn}

Whether a function $f$ between subnumbered spaces $(X,\nu)$ and
$(Y,\mu)$ is computable only depends on its behavior on the set of
$\nu$-computable points of $X$. The following is a simple consequence
of the fact that the composition of computable functions (defined
on $\mathbb{N}$) is computable:
\begin{lem}
If $(X,\nu)$, $(Y,\mu)$ and $(Z,\tau)$ are subnumbered sets, the
composition of a $(\nu,\mu)$-computable function with a $(\mu,\tau)$-computable
function is $(\nu,\tau)$-computable.
\end{lem}

The identity function $\text{id}_{\mathbb{N}}$ on $\mathbb{N}$ defines
its most natural numbering. 
\begin{defn}
If $(X,\nu)$ is a subnumbered set, a\emph{ $\nu$-computable sequence}
is an $(\text{id}_{\mathbb{N}},\nu)$-computable function from $\mathbb{N}$
to $X$. 
\end{defn}

It was explained in the introduction why we need Banach-Mazur computability. 
\begin{defn}
\label{def:Banach-Mazur Comp}Let $(X,\nu)$ and $(Y,\mu)$ be subnumbered
spaces. A function $f:X\rightarrow Y$ is called \emph{Banach-Mazur
$(\nu,\mu)$-computable} if the image of any $\nu$-computable sequence
is a $\mu$-computable sequence. 
\end{defn}

We consider a partial order on the subnumberings of a set $X$:
\begin{defn}
A subnumbering $\nu$ of a space $X$ is \emph{stronger} than a subnumbering
$\mu$ of this same space if the identity on $X$ is $(\nu,\mu)$-computable.
We denote this $\nu\succeq\mu$. Those subnumberings are \emph{equivalent}
if $\nu\succeq\mu$ and $\mu\succeq\nu$ both hold. We denote this
by $\nu\equiv\mu$. 
\end{defn}

The relation $\nu\succeq\mu$ exactly means that there is an algorithm
that, given a $\nu$-name of a point in $X$, produces a $\mu$-name
of it, and thus can be interpreted as: a $\nu$-name of a point $x$
contains more information about it than a $\mu$-name of this same
point. 

The following lemma is a direct consequence of the fact that composition
preserves computability of functions. 
\begin{lem}
The relation $\succeq$ is transitive and reflexive. In particular,
the relation $\equiv$ is an equivalence relation. 
\end{lem}

\begin{defn}
The \emph{subnumbering types} on $X$ are the equivalence classes
of $\equiv$. 
\end{defn}

If $(X,\nu)$ and $(Y,\mu)$ are subnumbered spaces, and if $f:X\rightarrow Y$
is a $(\nu,\mu)$-computable function between $X$ and $Y$, then
$f$ will be computable with respect to any pair of subnumberings
of $X$ and $Y$ which are $\equiv$-equivalent respectively to $\nu$
and $\mu$. Thus $f$ can be considered computable with respect to
the subnumbering types associated to $\nu$ and $\mu$. Subnumbering
types are in fact the objects that we want to be studying, rather
than subnumberings. We denote $\mathcal{NT}_{X}$ the set of subnumbering
types on $X$. Given a subnumbering, we denote $\left[\nu\right]$
its $\equiv$-equivalence class. We usually denote subnumbering types
by capital greek letters, $\Lambda$, $\Delta$, ... 

The relation $\succeq$ defines an order on the set of subnumbering
types, whose description is an important part of the study of the
different subnumberings of a set. 
\begin{defn}
Let $(X,\nu)$ be a subnumbered set. Then the equivalence relation
defined on $\text{dom}(\nu)$ by $n\sim m\iff\nu(n)=\nu(m)$ is called
the \emph{subnumbering equivalence induced} \emph{by} $\nu$ \cite{Ershov1999},
and denoted $\eta_{\nu}$. 

The subnumbering $\nu$ is called \emph{positive} \cite{Maltsev1971}
when equality on $X$ is a $\nu$-semi-decidable relation, i.e. when
there is a recursively enumerable set $L\subseteq\mathbb{N}\times\mathbb{N}$
such that $\eta_{\nu}=L\cap\text{dom}(\nu)\times\text{dom}(\nu)$. 

The subnumbering $\nu$ is called \emph{negative} when equality on
$X$ is a $\nu$-co-semi-decidable relation, i.e. when there is a
co-recursively enumerable set $L\subseteq\mathbb{N}\times\mathbb{N}$
such that $\eta_{\nu}=L\cap\text{dom}(\nu)\times\text{dom}(\nu)$. 

It is called \emph{decidable} when if it is both positive and negative. 
\end{defn}

\begin{prop}
If $\nu\succeq\mu$, and if $\mu$ is any of negative, positive or
decidable, then so is $\nu$. 
\end{prop}

\begin{proof}
As $\nu$-names provide more information than $\mu$-names, if $\mu$-names
allow to (partially) distinguish points, then so do $\nu$-names. 
\end{proof}

\subsubsection{Constructions and examples}

There are several useful constructions that allow one to build subnumberings
of sets using subnumberings of simpler sets. 
\begin{defn}
Given a subnumbered set $(X,\nu)$ and a subset $Y$ of $X$, define
\emph{the} \emph{restriction of $\nu$ to $Y$} to be the subnumbering
$\nu_{\vert Y}$ defined by the following:
\[
\text{dom}(\nu_{\vert Y})=\text{dom}(\nu)\cap\nu^{-1}(Y),
\]
\[
\forall n\in\text{dom}(\nu_{\vert Y}),\,\nu_{\vert Y}(n)=\nu(n).
\]
\end{defn}

We define the product of two subnumberings. 
\begin{defn}
If $(X,\nu)$ and $(Y,\nu)$ are subnumbered sets, \emph{the} \emph{product
of the subnumberings} $\nu$ \emph{and} $\mu$ is the subnumbering
$\nu\times\mu$ of $X\times Y$ defined by the following:
\[
\text{dom}(\nu\times\mu)=\{\langle n,m\rangle\in\mathbb{N},\,n\in\text{dom}(\nu),\,m\in\text{dom}(\mu)\};
\]
\[
\forall n\in\text{dom}(\nu\times\mu),\,\nu\times\mu(\langle n,m\rangle)=(\nu(n),\mu(m)).
\]
\end{defn}

We do not prove the following easy proposition:
\begin{prop}
If $\nu_{1}\equiv\nu_{2}$ and if $\mu_{1}\equiv\mu_{2}$ then $\nu_{1}\times\mu_{1}\equiv\nu_{2}\times\mu_{2}$.
\end{prop}

Finally, we give the definition of the natural numbering of the set
of computable functions between numbered sets. Denote by $(\varphi_{0},\varphi_{1},\varphi_{2}...)$
a standard Gödel enumeration of all recursive functions \cite{Rogers1987}.
(We do not describe here how to obtain such an enumeration: numberings
only allow to translate any enumeration of the partial recursive functions
to an enumeration of the computable functions between numbered sets.)
\begin{defn}
If $(X,\nu)$ and $(Y,\mu)$ are subnumbered sets, we define a subnumbering
$\mu^{\nu}:\,\subseteq\mathbb{N}\rightarrow Y_{\mu}^{X_{\nu}}$ of
the set of functions that map $\nu$-computable points of $X$ to
$\mu$-computable points of $Y$ as follows: 
\begin{align*}
\text{dom}(\mu^{\nu}) & =\{i\in\mathbb{N}\vert\text{ dom}(\nu)\subseteq\text{dom}(\varphi_{i}),\\
 & \forall n,m\in\text{dom}(\nu),\,\nu(n)=\nu(m)\implies\mu(\varphi_{i}(n))=\mu(\varphi_{i}(m))\},
\end{align*}
\[
\forall i\in\text{dom}(\mu^{\nu}),\forall x\in X_{\nu},\forall k\in\text{dom}(\nu),\,(x=\nu(k))\implies(\mu^{\nu}(i))(x)=\mu(\varphi_{i}(k)).
\]
The image of $\mu^{\nu}$ is exactly the set of $(\nu,\mu)$-computable
functions. The following commutative diagram renders this definition
clearer: 
\end{defn}

\begin{center}
\begin{tikzcd}
X_{\nu} \arrow[r, "\mu^{\nu}(i)"]& Y_{\mu} \\ \mathbb{N}\arrow[u, "\nu"]\arrow[r, "\varphi_{i}"]& \mathbb{N}\arrow[u, "\mu"]\end{tikzcd}
\end{center}

Several examples follow from those constructions: 
\begin{itemize}
\item The Baire space $\mathcal{N}=\mathbb{N}^{\mathbb{N}}$ is naturally
equipped with the subnumbering $\text{id}_{\mathbb{N}}^{\text{id}_{\mathbb{N}}}$,
which we usually denote $c_{\mathcal{N}}$. Denote by $\mathcal{N}_{c}=\mathbb{N}_{c}^{\mathbb{N}}$
the set of computable points of the Baire space. 
\item The Cantor space $\mathcal{C}=\{0,1\}^{\mathbb{N}}$ admits a subnumbering
induced by its natural embedding into $\mathcal{N}$, we denote it
$c_{\mathcal{C}}$. Denote by $\mathcal{C}_{c}=\{0,1\}_{c}^{\mathbb{N}}$
the set of computable points of the Cantor space. 
\end{itemize}

\subsection{Lattice operations on $\mathcal{NT}_{X}$}

We now introduce the lattice structure on $\mathcal{NT}_{X}$. Here,
by lattice, we mean a partially ordered set that admits meet and join
operations: we will thus show that any pair of subnumbering types
in $\mathcal{NT}_{X}$ admits both a greatest lower bound and a least
upper bound for the order $\succeq$. 

The lattice operations of $\mathcal{NT}_{X}$ are the conjunction
and the disjunction. Given two subnumberings $\nu$ and $\mu$ of
a set $X$, we will define new subnumberings $\nu\wedge\mu$ and $\nu\vee\mu$
by saying respectively that a $\nu\wedge\mu$-name for a point $x$
is a $\nu$-name for $x$ together with a $\mu$-name for it, and
that a $\nu\vee\mu$-name for a point $y$ is either a $\nu$-name
for it, or a $\mu$-name for it. Those definitions are explained below. 
\begin{defn}
Let $\nu$ and $\mu$ be subnumberings of $X$. Define a subnumbering
$\nu\vee\mu$ (the \emph{disjunction} of $\nu$ and $\mu$) by setting,
for any natural number $k$, $\nu\vee\mu(2k)=\nu(k)$ and $\nu\vee\mu(2k+1)=\mu(k)$.
The domain of $\nu\vee\mu$ is the set $\{2k,k\in\text{dom}(\nu)\}\cup\{2k+1,k\in\text{dom}(\mu)\}$.
\end{defn}

Thus a $\nu\vee\mu$-name for a point $x$ of $X$ is either a $\nu$-name
for it, or a $\mu$-name for it. 
\begin{prop}
Let $\nu$ and $\mu$ be subnumberings of $X$. Then $\nu\succeq\nu\vee\mu$
and $\mu\succeq\nu\vee\mu$, and for any $\tau$ in $\mathcal{N}_{X}$,
if $\nu\succeq\tau$ and $\mu\succeq\tau$, then $\nu\vee\mu\succeq\tau$. 
\end{prop}

\begin{proof}
Left to the reader. 
\end{proof}
\begin{prop}
Let $\nu$, $\mu$ and $\tau$ be subnumberings of $X$. 
\begin{itemize}
\item If $\nu\equiv\mu$, then $\nu\vee\tau\equiv\mu\vee\tau$,
\item $\nu\vee\mu\equiv\mu\vee\nu$,
\item $\nu\vee\nu\equiv\nu$
\end{itemize}
\end{prop}

\begin{proof}
Left to the reader.
\end{proof}
As a corollary to this proposition we can define the following: 
\begin{defn}
Let $\Lambda$ and $\Theta$ be subnumbering types of $X$. The subnumbering
type $\Lambda\vee\Theta$ is defined as being the $\equiv$-class
of $\nu\vee\mu$ for any $\nu$ in $\Lambda$ and $\mu$ in $\Theta$. 
\end{defn}

We now define the subnumbering obtained by giving as a name for a
point in $x$ both a $\nu$-name and a $\mu$-name for it, the ``conjunction''
of the subnumberings $\nu$ and $\mu$. Recall that $(n,m)\mapsto\langle n,m\rangle$
denotes a pairing function. The decoding functions are denoted $\text{fst}$
and $\text{snd}$. 
\begin{defn}
Let $\nu$ and $\mu$ be numberings of $X$. Define a subnumbering
$\nu\wedge\mu$ (the \emph{conjunction} of $\nu$ and\emph{ }$\mu$)
by the following:
\[
\text{dom}(\nu\wedge\mu)=\{\langle n,m\rangle\in\mathbb{N},n\in\text{dom}(\nu),m\in\text{dom}(\mu),\nu(n)=\mu(m)\},
\]
\[
\forall\langle n,m\rangle\in\text{dom}(\nu\wedge\mu),\,\nu\wedge\mu(\langle n,m\rangle)=\nu(n).
\]
\end{defn}

As we have already said, a $\nu\wedge\mu$-name for a point $x$ of
$X$ is constituted of both a $\nu$-name and a $\mu$-name for it. 
\begin{prop}
Let $\nu$ and $\mu$ be subnumberings of $X$. Then $\nu\wedge\mu\succeq\nu$
and $\nu\wedge\mu\succeq\mu$, and for any $\tau$ in $\mathcal{N}_{X}$,
if $\tau\succeq\nu$ and $\tau\succeq\mu$, then $\tau\succeq\nu\wedge\mu$. 
\end{prop}

\begin{proof}
We first show that $\nu\wedge\mu\succeq\nu$ and $\nu\wedge\mu\succeq\mu$.
But given $x$ in $X$ and $\langle n,m\rangle$ in $\mathbb{N}$
such that $\nu\wedge\mu(\langle n,m\rangle)=x$, by definition of
$\nu\wedge\mu$ one must have $x=\nu(n)=\mu(m)$, and thus the functions
$\text{fst}$ and $\text{snd}$ which form the inverse of the pairing
function are computable witnesses respectively for $\nu\wedge\mu\succeq\nu$
and $\nu\wedge\mu\succeq\mu$. 

Suppose now that $\tau\in\mathcal{N}_{X}$ is such that $\tau\succeq\nu$
and $\tau\succeq\mu$. This means that there are computable functions
$F:\,\subseteq\mathbb{N}\rightarrow\mathbb{N}$ and $G:\,\subseteq\mathbb{N}\rightarrow\mathbb{N}$
such that $\forall n\in\text{dom}(\tau),\tau(n)=\nu(F(n))$ and $\forall n\in\text{dom}(\tau),\tau(n)=\mu(G(n))$.
Let $n$ be a $\tau$-name for a point $x$ in $X$. Then $\langle F(n),G(n)\rangle$
is a $\nu\wedge\mu$-name for $x$, since $\nu(\text{fst}(\langle F(n),G(n)\rangle))=\nu(F(n))=\tau(n)$
and $\mu(\text{snd}(\langle F(n),G(n)\rangle))=\mu(G(n))=\tau(n)$.
Thus $\tau\succeq\nu\wedge\mu$. 
\end{proof}
The following proposition is straightforward. 
\begin{prop}
Let $\nu$, $\mu$ and $\tau$ be subnumberings of $X$. 
\begin{itemize}
\item If $\nu\equiv\mu$, then $\nu\wedge\tau\equiv\mu\wedge\tau$;
\item $\nu\wedge\mu\equiv\mu\wedge\nu$;
\item $\nu\wedge\nu\equiv\nu$;
\end{itemize}
\end{prop}

\begin{proof}
Left to the reader. 
\end{proof}
This proposition allows us to define the conjunction of subnumbering
types:
\begin{defn}
Let $\Lambda$ and $\Theta$ be subnumbering types of $X$. The numbering
type $\Lambda\wedge\Theta$ is defined as being the $\equiv$-class
of $\nu\wedge\mu$ for any $\nu$ in $\Lambda$ and $\mu$ in $\Theta$. 
\end{defn}

We thus obtain the following result: 
\begin{thm}
\label{Thm Lattice of numbering types }$(\mathcal{NT}_{X},\succeq,\wedge,\vee)$
is a lattice. 
\end{thm}

\subsection{Computably enumerable and semi-decidable sets }

Let $(X,\nu)$ be a numbered set. 
\begin{defn}
A subset $Y$ of $X$ is called \emph{$\nu$-computably enumerable
}($\nu$-c.e.) if there is a computably enumerable subset $A$ of
$\mathbb{N}$, such that $A\subseteq\text{dom}(\nu)$ and $Y=\nu(A)$. 
\end{defn}

Remark that a code for the c.e. set $A$ that appears in this definition
constitutes via $\nu$ a description of the $\nu$-c.e. set $Y$.
This allows us to define a subnumbering $\nu_{c.e.}:\,\subseteq\mathbb{N}\rightarrow\mathcal{P}(X)$
of the powerset of $X$. 

Recall that if $(\varphi_{0},\varphi_{1},\varphi_{2},...)$ is a standard
enumeration of the partial computable functions, we obtain an effective
enumeration $(W_{0},W_{1},W_{2},...)$ of computably enumerable subsets
of $\mathbb{N}$ by setting $W_{i}=\text{dom}(\varphi_{i})$. Define
$\nu_{c.e.}$ by the following:

\[
\text{dom}(\nu_{c.e.})=\{i\in\mathbb{N},\,W_{i}\subseteq\text{dom}(\nu)\};
\]
\[
\forall i\in\text{dom}(\nu_{c.e.}),\,\nu_{c.e.}(i)=\nu(W_{i}).
\]

Of course we have the following lemma:
\begin{lem}
If $\nu\equiv\mu$, then $\nu_{c.e.}\equiv\mu_{c.e.}$. 
\end{lem}

\begin{proof}
This follows from the standard fact that the image of a c.e. subset
of $\mathbb{N}$ by a computable function is again c.e., and that
this statement is effective: for each computable $f$, there is a
computable map $g$ such that for all $i\in\mathbb{N}$, $f(W_{i})=W_{g(i)}$.
This follows directly from the smn theorem. 
\end{proof}
We now define decidable, semi-decidable and co-semi-decidable subsets. 
\begin{defn}
A set $Y$ is \emph{$\nu$-semi-decidable} if there is a computably
enumerable subset $A$ of $\mathbb{N}$ such that $A\cap\text{dom}(\nu)=\nu^{-1}(Y)$.
The set $A$ is \emph{$\nu$-co-semi-decidable} if there is a co-computably
enumerable subset $B$ of $\mathbb{N}$ such that $B\cap\text{dom}(\nu)=\nu^{-1}(Y)$. 

A set is $\nu$\emph{-decidable} if it is both $\nu$-semi-decidable
and $\nu$-co-semi-decidable.
\end{defn}

An equivalent formulation is that a set $Y$ is \emph{$\nu$}-decidable
if there exists a procedure that, given a $\nu$-name of an element
in $X$, decides whether or not it belongs to $Y$. It is \emph{$\nu$}-semi-decidable
if there exists a procedure that stops exactly on the $\nu$-names
of elements of $Y$, and \emph{$\nu$}-co-semi-decidable if there
exists a procedure that stops exactly on the $\nu$-names of elements
that do not belong to $Y$. 

Each of the definitions above allows to define subnumberings of the
set $\mathcal{P}(X_{\nu})$ of all subsets of the $\nu$-computable
points of $X$. We detail those definitions. 

Associated to the subnumbering $\nu$ of $X$, we define three subnumberings
$\nu_{D}:\,\subseteq\mathbb{N}\rightarrow\mathcal{P}(X_{\nu})$, $\nu_{SD}:\,\subseteq\mathbb{N}\rightarrow\mathcal{P}(X_{\nu})$
and $\nu_{co-SD}:\,\subseteq\mathbb{N}\rightarrow\mathcal{P}(X_{\nu})$
associated respectively to $\nu$-decidable, $\nu$-semi-decidable
and $\nu$-co-semi-decidable subsets of $X_{\nu}$.

As before, denote $W_{i}=\text{dom}(\varphi_{i})$. 
\begin{itemize}
\item Define $\nu_{D}$ by 
\begin{align*}
\text{dom}(\nu_{D}) & =\{i\in\mathbb{N}\vert\text{ dom}(\nu)\subseteq W_{i},\varphi_{i}(\text{dom}(\nu))\subseteq\{0,1\},\\
 & \forall n,m\in\text{dom}(\nu),\nu(n)=\nu(m)\implies(\varphi_{i}(n)=\varphi_{i}(m))\};
\end{align*}
\[
\forall i\in\text{dom}(\nu_{D}),\,\nu_{D}(i)=\nu(\varphi_{i}^{-1}(\{1\})).
\]
\item Define $\nu_{SD}$ by 
\[
\text{dom}(\nu_{SD})=\{i\in\mathbb{N}\vert\,\forall n,m\in\text{dom}(\nu),\nu(n)=\nu(m)\implies(n\in W_{i}\iff m\in W_{i})\};
\]
\[
\forall i\in\text{dom}(\nu_{SD}),\,\nu_{SD}(i)=\nu(\text{dom}(\nu)\cap W_{i}).
\]
\item Define $\nu_{co-SD}$ by 
\[
\text{dom}(\nu_{co-SD})=\{i\in\mathbb{N}\vert\,\forall n,m\in\text{dom}(\nu),\nu(n)=\nu(m)\implies(n\notin W_{i}\iff m\notin W_{i})\};
\]
\[
\forall i\in\text{dom}(\nu_{co-SD}),\,\nu_{co-SD}(i)=\nu(\text{dom}(\nu)\setminus W_{i}).
\]
\end{itemize}
We do not prove the following easy lemmas: 
\begin{lem}
If $\nu\equiv\mu$, then $\nu_{D}\equiv\mu_{D}$, $\nu_{SD}\equiv\mu_{SD}$
and $\nu_{co-SD}\equiv\mu_{co-SD}$. 
\begin{lem}
The subnumberings $\nu_{D}$, $\nu_{SD}$ and $\nu_{co-SD}$ satisfy
$\nu_{D}\equiv\nu_{SD}\wedge\nu_{co-SD}$. 
\end{lem}

\end{lem}

A $\nu_{SD}$-name of a $\nu$-semi-decidable set is usually simply
called its ``code'', the same goes for c.e., decidable and co-semi-decidable
sets, and the four subnumberings $\nu_{c.e.}$, $\nu_{D}$, $\nu_{SD}$
and $\nu_{co-SD}$ defined here are usually not explicitly called
upon: the sentence ``given a decidable subset of $X$'' is usually
used instead of ``given the $\nu_{D}$-name of a subset of $X$''. 

\section{\label{sec:The-word-problem-Numbering-Type}The word problem subnumbering
type}

We will give four equivalent definitions of the word problem subnumbering
type, which we denote $\Lambda_{WP}$. $\Lambda_{WP}$ is the $\equiv$-equivalence
class of the numbering $\nu_{WP}$ defined in the introduction. It
is in fact the object we want to study: the definition of $\nu_{WP}$
relies on arbitrary choices that are not $\equiv$-invariants. 

We present four definitions to illustrate the fact that the definition
of $\Lambda_{WP}$ is robust. One is the usual definition, one uses
the notion of computable group of Malcev and Rabin, one uses labelled
Cayley graphs, the last one uses the embedding of the space of marked
group in a disjoint union of countably many Cantor spaces. These definitions
and their equivalence are well known, although this fact is seldom
expressed in terms of subnumberings. 

Before defining subnumberings of marked groups, we first include a
paragraph on the numberings of the elements of fixed a marked group. 

\subsection{\label{subsec:Numberings-of-the-elements }Numberings of the elements
of a marked group. }

In a marked group $(G,S)$, it is customary to describe group elements
by words on $S\cup S^{-1}$. This description is in fact canonical,
in a sense that can be made precise using numberings. 

We denote by $\Lambda_{(G,S)}$ the numbering type of $(G,S)$ associated
to the idea that elements are described by words on $S\cup S^{-1}$,
we define it formally in the next definition. Denote by $(p_{n})_{n\in\mathbb{N}}$
the sequence of prime numbers. 
\begin{defn}
If $(G,S)$ is a marked group, and $S=(s_{0},s_{2},...,s_{k-1})$,
we define the numbering $\nu_{(G,S)}$ on $\mathbb{N}$ as follows.
For $i$ between $k$ and $2k-1$, denote by $s_{i}$ the element
$s_{i-k}^{-1}$ of $G$. Given a natural number $n$, decompose it
as a product of primes $n=p_{0}^{\alpha_{0}}...p_{m}^{\alpha_{m}}$.
Then, for each $i$ between $1$ and $m$, denote $\widetilde{\alpha_{i}}$
the remainder in the Euclidian division of $\alpha_{i}$ by $2k$.
We then put:
\[
\nu_{(G,S)}(n)=s_{\widetilde{\alpha_{1}}}s_{\widetilde{\alpha_{2}}}...s_{\widetilde{\alpha_{m}}}\in G.
\]

The numbering type $\Lambda_{(G,S)}$ is the $\equiv$-equivalence
class of $\nu_{(G,S)}$. 
\end{defn}

The arbitrary choices that are made in this definition are unimportant,
as is shown by the following proposition: 
\begin{prop}
\label{prop:The-numbering-type v(G,S) }The numbering type $\Lambda_{(G,S)}$
is the greatest subnumbering type in the lattice $\mathcal{NT}_{G}$
which satisfies the following conditions:
\begin{itemize}
\item All elements of $G$ are $\Lambda_{(G,S)}$-computable; 
\item The group law and the inverse function on $G$ are respectively $(\Lambda_{(G,S)}\times\Lambda_{(G,S)},\Lambda_{(G,S)})$
and $(\Lambda_{(G,S)},\Lambda_{(G,S)})$ computable. 
\end{itemize}
(In particular, any subnumbering type which satisfies these conditions
can be compared to $\Lambda_{(G,S)}$ for the order $\succeq$.)
\end{prop}

Note that the first condition of this proposition could be replaced
equivalently by: ``The elements of the generating tuple $S$ are
$\Lambda_{(G,S)}$-computable''. 
\begin{proof}
Suppose that $\mu$ is any numbering of $G$ for which the group law
and the inverse function are computable. 

We show that $\nu_{(G,S)}\succeq\mu$, where $\nu_{(G,S)}$ is the
numbering which was used to define the numbering type $\Lambda_{(G,S)}$.
The generating set of $G$ is denoted $S=(s_{0},s_{2},...,s_{k-1})$.
As $\mu$ is surjective, there are numbers $u_{0}$, ...,$u_{k-1}$
such that $\mu(u_{i})=s_{i}$. 

As the group law is be computable for $\mu$, there is a computable
function $F$ such that $\mu(F(i,j))=\mu(i)\mu(j)\in G$, and a computable
function $I$ that computes the inverse for $\mu$. 

Given an integer $n$, which we decompose as a product of primes,
$n=p_{0}^{\alpha_{0}}...p_{m}^{\alpha_{m}}$, recall that $\nu_{(G,S)}(n)=s_{\widetilde{\alpha_{1}}}s_{\widetilde{\alpha_{2}}}...s_{\widetilde{\alpha_{m}}}$.
Consider the function $H:\mathbb{N}\rightarrow\mathbb{N}$ defined
as follows: map an integer $\alpha$ to its reminder modulo $2k$,
which we denote $\widetilde{\alpha}$, then, if $\widetilde{\alpha}\le k-1$,
map it to $u_{\widetilde{\alpha}}$, otherwise, if $\widetilde{\alpha}\ge k$,
map it to $I(u_{\widetilde{\alpha}-k})$. 

Then $F(H(\alpha_{1}),F(H(\alpha_{2}),...)))$ gives a $\mu$-name
for $\nu_{(G,S)}(n)$, and the procedure which produces this name
from $n$ is clearly computable. 
\end{proof}
This proposition is in fact a simple application of a well known fact
about subnumberings: if $(X,\Delta)$ is a subnumbered set, and if
$\{f_{1},...,f_{n}\}$ are functions defined on cartesian powers of
$X$ to $X$, then there is a greatest subnumbering type $\Lambda$,
that is below $\Delta$ ($\Delta\succeq\Lambda$), and for which all
the functions $f_{i}$ are computable. This was first detailed in
\cite[Section 2.2]{Weihrauch1987}. The numbering type $\Lambda_{(G,S)}$
is obtained following this principle, using for $\Delta$ the equivalence
class of the subnumbering $\nu_{0}$ that describes only $S$: $\nu_{0}$
is defined on $\{0,...,k-1\}$ and it maps $i$ to $s_{i}$. 

Proposition \ref{prop:The-numbering-type v(G,S) } shows that describing
group elements as products of the generators is the description that
gives as much information as can be given on group elements if we
want the group operations to be computable. The previous proposition
thus has an easy corollary:
\begin{cor}
The numbering type $\Lambda_{(G,S)}$ has a decidable equality if
and only if there exists a numbering type $\Delta$ of $G$ which
has decidable equality and which makes the group operations of $G$
computable. 
\end{cor}

\begin{proof}
Suppose that $\Delta$ is as in the hypotheses of the corollary. Then,
by Proposition \ref{prop:The-numbering-type v(G,S) }, we have $\Lambda_{(G,S)}\succeq\Delta$.
Thus if $\Delta$ has a decidable equality, then so has $\Lambda_{(G,S)}$. 
\end{proof}
Another corollary of Proposition \ref{prop:The-numbering-type v(G,S) }
is: 
\begin{cor}
If $S$ and $S'$ are two generating families of $G$, then $\Lambda_{(G,S)}=\Lambda_{(G,S')}$. 
\end{cor}

\begin{proof}
Proposition \ref{prop:The-numbering-type v(G,S) } gives a characterization
of $\Lambda_{(G,S)}$ which is independent of $S$. 
\end{proof}
We can finally define solvability of the word problem: 
\begin{defn}
A marked group $(G,S)$ is said to have \emph{solvable word problem}
if equality is decidable for $\Lambda_{(G,S)}$. 
\end{defn}

In this case, what we call a \emph{word problem algorithm }is the
computable function that witnesses for the fact that $\Lambda_{(G,S)}$
is decidable. 

\subsection{First definition of $\Lambda_{WP}$}

We will now define the subnumbering type associated to word problem
algorithms. Our first definition is based on the numbering $\nu_{(G,S)}$
defined above. 

The numbering $\nu_{(G,S)}$ has a decidable equality if and only
if there is a computable function of two variables $H$ such that
$H(i,j)=0$ if $\nu_{(G,S)}(i)=\nu_{(G,S)}(j)$ and $H(i,j)=1$ otherwise.
In this case, $H$ is said to witness for the fact that $\nu_{(G,S)}$
has decidable equality.

Let $(\varphi_{0},\varphi_{1},...)$ be a standard enumeration of
all partial computable functions. We can consider that those functions
depend on two variables using an encoding of pairs of natural numbers. 

Define as follows a subnumbering $\nu_{WP}$. 

Pose $\nu_{WP}(n)=(G,S)$ if and only if $n$ encodes a pair $(k,m)$,
i.e. $n=\langle k,m\rangle$, $S$ is a family with $k$ elements,
$\nu_{(G,S)}$ is decidable, and $\varphi_{m}$ is a computable function
that witnesses for the fact that $\nu_{(G,S)}$ is decidable. 

To check that this is a correct definition, we must check that the
marked group $(G,S)$ is uniquely defined by any of its $\nu_{WP}$-names.
This is to say: we must check that a word problem algorithm defines
uniquely a marked group. 

But this is straightforward: in the definition above, if $n=\langle k,m\rangle$
codes for two marked groups $(G,S)$ and $(H,S')$, first it must
be that $S$ and $S'$ have the same cardinality $k$, and, secondly,
that $\varphi_{m}$ is a computable function that witnesses for the
decidability of both $\nu_{(G,S)}$ and $\nu_{(H,S')}$, but then
$(G,S)$ and $(H,S')$ satisfy exactly the same relations, and thus
they are isomorphic as marked groups. 

We then define $\Lambda_{WP}$ to be the $\equiv$-equivalence class
of $\nu_{WP}$. 

\subsection{Labelled Cayley graph definition of $\Lambda_{WP}$}

Describing a marked group by a word problem algorithm is equivalent
to describing it by its labelled Cayley graph. And thus, as we investigate
decision problems for groups described by word problem algorithms,
we are also studying ``what can be deduced about a marked group,
given its labeled Cayley graph''. The graph should be suitably encoded
into a finite amount of data. We detail this now. 
\begin{defn}
If $(G,S=(s_{1},...,s_{n}))$ is a marked group, the labelled Cayley
graph associated to it is the graph whose vertexes are elements of
$G$, and whose (directed) edges are defined as follows: there is
an edge with label $s_{i}$ from the vertex $g_{1}\in G$ to the vertex
$g_{2}\in G$ if and only if $g_{1}s_{i}=g_{2}$. 
\end{defn}

We can then use the standard way to encode infinite graphs (computable
graphs) to define a subnumbering of Cayley graphs. This again uses
a standard enumeration $(\varphi_{0},\varphi_{1},\varphi_{2},...)$
of partial computable functions. 

An oriented edge-labelled graph is a quadruple $(V,E,C,c:E\rightarrow C)$,
where: $V$ is a set, the set of vertices, $E$ is the set of oriented
edges, i.e. a subset of $V\times V$, $C$ is a set of colors, which
here we suppose finite, and $c$ is a function which defines the color
of a given edge. 

Such a graph $\Gamma=(V,E,C,c:E\rightarrow C)$ is called computable
if it is isomorphic to a graph $\Gamma_{1}=(V_{1},E_{1},C_{1},\,c_{1}:E_{1}\rightarrow C_{1})$,
which satisfies additionally that: $V_{1}$ is a computable subset
of $\mathbb{N}$, $E_{1}$ is a computable subset of $\mathbb{N}\times\mathbb{N}$,
$C_{1}=\{1,...,k\}$ for some $k\in\mathbb{N}$, and $c_{1}:E_{1}\rightarrow C_{1}$
is a computable function. In this case, $\Gamma_{1}$ is called a
\emph{computable model} of $\Gamma$. 

Notice that each element of the tuple $(V_{1},E_{1},C_{1},c_{1}:E_{1}\rightarrow C_{1})$
is associated to some finite data that can be encoded: the characteristic
function of $V_{1}$, denoted $\chi_{V_{1}}$, which is computable,
the characteristic function of $E_{1}$, denoted $\chi_{E_{1}}$,
the natural number $k$ such that $C_{1}=\{1,...,k\}$, and the code
of the function $c_{1}$. 

We can thus define a subnumbering $\nu_{\Gamma}$ of edge-labelled
graphs by saying that a computable model $\Gamma_{1}=(V_{1},E_{1},C_{1},\,c_{1}:E_{1}\rightarrow C_{1})$
of a graph $\Gamma$ is encoded by a tuple $(i,j,k,l)$, where: $\varphi_{i}=\chi_{V_{1}}$,
$\varphi_{j}=\chi_{E_{1}}$, $C_{1}=\{1,...,k\}$, $\varphi_{l}=c_{1}$. 

One easily checks that the tuple $(i,j,k,l)$ defines a unique graph,
and thus the definition given above is sound. 

We now have a subnumbering $\nu_{\Gamma}$ of edge-labelled graphs,
we can restrict it to the set of Cayley graphs, and, because a labelled
Cayley graph defines uniquely a marked group, we can consider that
this new subnumbering is a subnumbering of the set of marked groups
(we compose the subnumbering of graphs to the function that maps a
labelled Cayley graph to the group it defines). 

Note that, in the computable model of a Cayley graph, we can always
suppose that there is a vertex at $0$, and that it is associated
to the identity element of the group whose graph it is. 

This defines a subnumbering that we denote $\nu_{\text{Cay}}$, which
is associated to the idea ``a marked group is described by algorithms
that describe its Cayley graph''. 

We can now show:
\begin{thm}
\label{thm:Numbering-Cay-Wp}The subnumbering $\nu_{\text{Cay}}$
is $\equiv$-equivalent to $\nu_{WP}$. 
\end{thm}

\begin{proof}
Given a $\nu_{\text{Cay}}$-name for a marked group $(G,S)$, i.e.
given a Cayley graph $\Gamma_{1}$ for it, we can solve the word problem
in $(G,S)$ by following edges along a word: given a word $w=a_{1}a_{2}...a_{n}$
on $S\cup S^{-1}$, starting from any vertex $v_{1}$ in $\Gamma_{1}$,
we can find (by an exhaustive search) a sequence of vertices $v_{2}$,
..., $v_{n+1}$ such that $v_{i+1}=v_{i}a_{i}$. We then solve the
word problem by checking whether $v_{1}=v_{n+1}$, i.e. by checking
whether the word $w$ defines a loop in the Cayley graph of $G$. 

Conversely, given a word problem algorithm for $(G,S)$, we can build
a computable model $\Gamma_{1}=(V_{1},E_{1},C_{1},c_{1}:E_{1}\rightarrow C_{1})$
of the Cayley graph of $(G,S)$ as follows: 
\begin{itemize}
\item Consider an enumeration of all words on $S\cup S^{-1}$ following
a given order, say by length and then lexicographically. We can then
delete, using the word problem algorithm of $G$, any element that
is redundant in this list. We obtain a list $w_{0}$, $w_{1}$, $w_{2}$,
... which contains a single word on $S\cup S^{-1}$ for each element
of $G$. 
\item Define a numbering $\mu$ of $G$ by saying that $\mu(i)=w_{i}$.
We put $V_{1}=\text{dom}(\mu)$. 
\item $V_{1}$ is $\mathbb{N}$ if $G$ is infinite, it is $\{0,...,\text{card}(G)-1\}$
otherwise. A computable characteristic function for $V_{1}$ can be
obtained as follows: the list $w_{0}$, $w_{1}$, ... can be enumerated,
and thus, given some number $i$, if it was found that $G$ contains
more than $i+1$ elements, $i$ belongs to $V_{1}$. On the contrary,
while the list $w_{0}$, $w_{1}$, $w_{2}$,~... is enumerated, we
can search for an initial segment of it of length less than $i$,
and that is stable by multiplication by any generator. If such a segment
is found, $G$ must be finite, and we know it has cardinality less
than $i$. In this case, we can conclude that $i\notin V_{1}$. 
\item We define $E_{1}$ and $c_{1}$ by saying that $(i,j)$ is an edge
labelled by $s\in S$ if and only if $w_{i}s=w_{j}$. This can be
effectively checked thanks to the word problem algorithm for $(G,S)$,
and thus we can produce the computable functions that define $E_{1}$
and $c_{1}$. 
\end{itemize}
\vspace{-0,5cm}
\end{proof}

\subsection{Computable groups definition of $\Lambda_{WP}$}

Another point of view on $\Lambda_{WP}$ follows (more or less) the
point of view of Malcev in \cite[Chapter 18]{Maltsev1971} and Rabin
in \cite{Rabin1960}. 
\begin{defn}
A countable group $G$ is \emph{computable} if there are computable
function $P:\mathbb{N}\times\mathbb{N}\rightarrow\mathbb{N}$ and
$I:\mathbb{N}\rightarrow\mathbb{N}$ such that either $(\mathbb{N},P,I,0)$
or $(\{0,..,\text{card}(G)-1\},P,I,0)$ (if $G$ is finite) is a group
that is isomorphic to $(G,\cdot,^{-1},e)$. 
\end{defn}

If $G$ is a finitely generated group, an isomorphism $\Theta:(G,\cdot,^{-1},e)\rightarrow(\mathbb{N},P,I,0)$
can be described by giving the images of a generating family of $G$
in $\mathbb{N}$. 

We define a new subnumbering of marked groups, denoted $\nu_{MR}$,
as follows. 

The description of a marked group $(G,S)$ for $\nu_{MR}$ is an encoded
quintuple $(p,k,m,i,j)$, such that:
\begin{itemize}
\item The function $\varphi_{p}$ has domain $\mathbb{N}$ if $G$ is infinite,
$\{0,...,\text{card}(G)-1\}$ otherwise; 
\item There exists a group isomorphism $\Theta:(G,\cdot,^{-1}e)\rightarrow(\text{dom}(\varphi_{p}),P,I,0)$,
where $P$ and $I$ are computable functions;
\item $k$ gives the cardinality of $S$;
\item $m$ can be decoded as a $k$-tuple of elements of $\mathbb{N}$,
which give the images of the elements of $S$ by $\Theta$;
\item $i$ and $j$ define the computable functions $P$ and $I$, i.e.
$\varphi_{i}=P$ and $\varphi_{j}=I$. 
\end{itemize}
As before, one can easily check that this definition is sound by checking
that there is no ambiguity as to which isomorphism $\Theta$ is encoded
by a number $n$. 

The following theorem is well known (for instance it is contained
in Theorem 7.1 in \cite{Cannonito_1966}, or it is Theorem 4 of \cite{Rabin1960}). 
\begin{thm}
The subnumberings $\nu_{WP}$ and $\nu_{MR}$ are $\equiv$-equivalent. 
\end{thm}

\begin{proof}
The proof is essentially the same as Theorem \ref{thm:Numbering-Cay-Wp}.
\end{proof}

\subsection{\label{subsec:Numbering induced by Cantor space}Subnumberings of
$\mathcal{G}^{+}$ induced by subnumberings of the Cantor space }

Recall that in Section \ref{sec:The-topological-space} we defined
an embedding $\Phi_{k}$ of the space $\mathcal{G}_{k}$ of $k$-marked
groups into the Cantor space. 

The natural subnumbering $\nu_{\mathcal{C}}$ of the Cantor space
is obtained by seeing it as the set of functions from $\mathbb{N}$
to $\{0,1\}$. . 

We can extend the subnumberings $\nu_{\mathcal{C}}$ to a countable
disjoint union of Cantor spaces. Consider a countable set of Cantor
spaces, denoted $\mathcal{C}_{i}$, $i\in\mathbb{N}$. 

Define a subnumbering $\hat{\nu_{\mathcal{C}}}$ of $\underset{n\in\mathbb{N}}{\bigcup}\mathcal{C}_{i}$
by the following: 
\[
\hat{\nu_{\mathcal{C}}}(\langle i,j\rangle)\in\mathcal{C}_{i},
\]
\[
\hat{\nu_{\mathcal{C}}}(\langle i,j\rangle)=\nu_{\mathcal{C}}(j).
\]

The following proposition proves that this definition gives yet another
way to introduce the subnumbering type $\Lambda_{WP}$. 
\begin{prop}
The restriction of $\hat{\nu_{\mathcal{C}}}$ to the space of marked
groups (seen as a subset of a countable union of Cantor spaces via
the maps $\Phi_{k}$) is equivalent to $\nu_{WP}$.
\end{prop}

(The embeddings $\Phi_{k}:\mathcal{G}_{k}\longrightarrow\left\{ 0,1\right\} ^{\mathbb{N}}$
were chosen so that this proposition would hold: so that $\Phi_{k}$
would be $(\nu_{WP},\nu_{\mathcal{C}})$-computable. Those embeddings
were defined thanks to bijections $\theta_{k}:\mathbb{N}\rightarrow\mathbb{F}_{k}$. 
\begin{proof}
We show that $\nu_{WP}$ is equivalent to $\hat{\nu_{\mathcal{C}}}$
on $\mathcal{G}$. 

If $\nu_{WP}(n)=(G,S)$, then $n$ encodes both the cardinality of
$S$ and the code for a function that, given two natural numbers,
decides whether or not they encode the same element in $G$ with respect
to the encoding of the elements of a marked group defined in Section
\ref{subsec:Numberings-of-the-elements }. Denote here $c_{k}:\mathbb{N}\rightarrow\mathbb{F}_{k}$
this encoding applied to the elements of the free groups. 

If $\hat{\nu_{\mathcal{C}}}(n)=(G,S)$, then $n$ encodes the cardinality
$k$ of $S$ and the code for a function that, given a natural number
$m$, indicates whether the elements $\theta_{k}(m)$ of $\mathbb{F}_{k}$
is a relation satisfied by $(G,S)$. 

The result then follows from the easy fact that each function $c_{k}\circ\theta_{k}^{-1}:\mathbb{N}\rightarrow\mathbb{N}$
is a computable surjection which has a computable right inverse, that
this holds uniformly in $k$, and from the standard fact that the
code of the composition of two functions can be computed given the
codes for those functions. 
\end{proof}

\subsection{\label{subsec:Presentation-and-co-presentation RPZ}Numbering and
representation associated to presentations}

We now introduce the presentation representation. For each $k$, we
have fixed a free group $\mathbb{F}_{k}$ and a basis $(x_{0},...,x_{k})$.
We use again the shortlex bijection $\theta_{k}:\mathbb{N}\rightarrow\mathbb{F}_{k}$. 

We then define the representation $\rho_{pres}:\mathbb{N}^{\mathbb{N}}\rightarrow\mathcal{G}$
associated to presentations, via: 
\[
\rho_{pres}(p)=(\mathbb{F}_{p_{0}}/\langle\langle\{\theta_{p_{0}}(p_{i}),i\ge1\rangle\rangle,(x_{0},...,x_{p_{0}})).
\]
In words: the first term in the name of a marked group gives the arity
of its marking, and the rest gives an infinite presentation via the
corresponding bijection $\theta_{k}$. 

We denote by $\nu_{pres}:\,\subseteq\mathbb{N}\rightarrow\mathcal{G}$
the associated numbering:
\[
\nu_{pres}(i)=\rho_{pres}(\varphi_{i})
\]
whenever $\varphi_{i}$ is a total computable function. Groups in
the image of $\nu_{pres}$ are the recursively presented groups. And
$\nu_{pres}$ is the numbering that allows to discuss decision problems
for groups given by recursive presentations. 

This representation will be useful in Section \ref{subsec:Virtually-cyclic-groups;}.
Note also that the final topology of $\rho_{pres}$ is the Scott topology
on the poset $(\mathcal{G},\rightarrow)$ of marked groups. The Scott
topology on $(\mathcal{G}^{k},\rightarrow)$ is connected, this was
used in \cite{Rauzy21} to prove the Rice Theorem for recursively
presented groups.

\section{\label{sec:Main-results-in}Introduction on some results in Markovian
computable analysis }

\subsection{Recursive Polish spaces}

This introduction follows mostly Kushner \cite{Kushner1984}, but
it is hopefully more accessible, since the constructivist setting
adds technical complications. Note that Section \ref{subsec:Markov's-Lemma-and}
follows closely Hertling \cite{Hertling2001}, who studies Banach-Mazur
computable functions. 

\subsubsection{The computable reals }

A precise definition of the set $\mathbb{R}_{c}$ of computable reals
first appeared in Turing's famous 1936 article \cite{Turing1936},
but the numbering type of computable real numbers which is best suited
to developing computable analysis was introduced one year later in
the corrigendum \cite{Turing1937}. 

We start by defining a numbering $c_{\mathbb{Q}}$ of the set of rationals.
The numbering $c_{\mathbb{Q}}$ is defined on $\mathbb{N}$. Given
a natural number $n$, seen as encoding a triple $n=\langle a,b,c\rangle$
via a pairing function, we put $c_{\mathbb{Q}}(n)=(-1)^{a}\frac{b}{c+1}$. 

We now define the Cauchy subnumbering $c_{\mathbb{R}}$ of $\mathbb{R}$.
Denote again by $(\varphi_{0},\varphi_{1},\varphi_{2}...)$ a standard
enumeration of all partial computable functions. 
\begin{defn}
The Cauchy subnumbering of $\mathbb{R}$ is defined by the formulas:
\[
\text{dom}(c_{\mathbb{R}})=\{i\in\mathbb{N},\,\exists x\in\mathbb{R},\,\forall n\in\mathbb{N},\,\left|c_{\mathbb{Q}}(\varphi_{i}(n))-x\right|<2^{-n}\};
\]
\[
\forall i\in\text{dom}(c_{\mathbb{R}}),\,c_{\mathbb{R}}(i)=\lim_{n\rightarrow\infty}(c_{\mathbb{Q}}(\varphi_{i}(n))).
\]
\end{defn}

Thus the description of a real number $x$ is a Turing machines that
produces a sequence $(u_{n})_{n\in\mathbb{N}}$ of rationals with
exponential convergence speed. 
\begin{defn}
The set of $c_{\mathbb{R}}$-computable real numbers is denoted $\mathbb{R}_{c}$,
the $c_{\mathbb{R}}$-computable reals are simply called the\emph{
computable real numbers}. 
\end{defn}

Several other definitions of the real numbers (decimal expansions,
Dedekind cuts), when rendered effective, yield subnumbering types
that define the same set of computable real numbers, but that are
not $\equiv$-equivalent to the Cauchy subnumbering type -they are
strictly stronger. See for instance \cite{Mostowski1979}. 
\begin{prop}
[Rice, \cite{Rice1954}]Addition, multiplication and divisions are
$(c_{\mathbb{R}}\times c_{\mathbb{R}},c_{\mathbb{R}})$-computable
functions defined respectively on $\mathbb{R}_{c}\times\mathbb{R}_{c}$,
$\mathbb{R}_{c}\times\mathbb{R}_{c}$ and $\mathbb{R}_{c}\times(\mathbb{R}_{c}\setminus\{0\})$. 
\end{prop}

The following is a well known proposition which follows easily from
Markov's Lemma, see Lemma \ref{lem:Markov}, this result is implicit
in \cite{Turing1936}. 
\begin{prop}
\label{prop:Equality-is-undecidable on R}Equality is undecidable
for computable reals. There is no algorithm that, given two computable
reals $x$ and $y$, chooses one of $x\le y$ or $y<x$ which is true. 
\end{prop}

\subsubsection{Recursive metric spaces}

We can now define what is a recursive metric spaces. We stick to the
old terminology that uses the term ``recursive'' instead of ``computable''
because the term ``computable metric space'' has a different meaning,
and it is not the case that every recursive metric space is a computable
metric space. In particular recursive metric spaces are countable
sets. \emph{Computable metric spaces} as studied in Type 2 computable
analysis \cite{Brattka2003} or \emph{recursively presented Polish
spaces} as studied in effective descriptive set theory \cite{Moschovakis1980}
can be uncountable sets. 
\begin{defn}
\emph{A recursive metric space (RMS)} is a countable metric space
$(X,d)$ equipped with a numbering $\nu:\,\subseteq\mathbb{N}\rightarrow X$,
and such that the distance function $d:X\times X\rightarrow\mathbb{R}$
is $(\nu\times\nu,c_{\mathbb{R}})$-computable. 
\end{defn}

Whether a triple $(X,d,\nu)$ is a RMS depends only on the $\equiv$-equivalence
class of the numbering $\nu$, we can thus define a RMS to be a metric
space equipped with a compatible numbering type. 
\begin{defn}
A function $f$ between recursive metric spaces $(X,d,\nu)$ and $(Y,d,\mu)$
is called \emph{effectively metric continuous }if given a $\nu$-name
$n$ of a point $x$ in $X$ and the $c_{\mathbb{R}}$-name of a computable
real number $\epsilon>0$, it is possible to compute the $c_{\mathbb{R}}$-name
of a number $\eta>0$ such that 
\[
\forall y\in X;\,d(x,y)<\eta\implies d(f(x),f(y))<\epsilon.
\]
\end{defn}

Note that the number $\eta$ is allowed to depend not only on $x$
and $\epsilon$, but also on the given names for those points. A program
that computes the $c_{\mathbb{R}}$-name for $\eta$ given $x$ and
$\epsilon$ is said to \emph{witness for the effective continuity}
\emph{of} $f$. 

We also introduce effective continuity for computably second countable
spaces. 
\begin{defn}
A function $f$ between second countable spaces $(X,\nu)$ and $(Y,\mu)$
with bases $(B_{i}^{X})_{i\in\mathbb{N}}$ and $(B_{i}^{Y})_{i\in\mathbb{N}}$
is called \emph{effectively continuous }if there is a program that
given $i\in\mathbb{N}$ produces the code of a c.e. set $A\subseteq\mathbb{N}$
such that 
\[
f^{-1}(B_{i}^{X})=\bigcup_{t\in A}B_{t}^{Y}.
\]
\end{defn}

\begin{prop}
\label{prop: Subset of RMS}If $(X,d,\nu)$ is a RMS, if $Y\subseteq X$,
then $(Y,d,\nu_{\vert Y})$ is also a RMS. 
\end{prop}

\begin{proof}
Obvious. 
\end{proof}
\begin{example}
The following spaces, equipped with their usual distances and numberings,
are recursive metric spaces: $\mathbb{N}$, $\mathbb{R}_{c}$, the
set $\{0,1\}_{c}^{\mathbb{N}}$, the set $\mathbb{N}_{c}^{\mathbb{N}}$.
Although details can be found in \cite[Chapter 9]{Kushner1984}, we
still include a proof for the Cantor space. 
\end{example}

\begin{prop}
\label{prop:The-Cantor-space RMS}The set $\{0,1\}_{c}^{\mathbb{N}}$
of computable points of the Cantor space equipped with its ultrametric
distance $d$ and its usual numbering $\nu_{\mathcal{C}}$ is a RMS. 
\end{prop}

\begin{proof}
Given the $\nu_{\mathcal{C}}$ name of two sequences $(u_{n})_{n\in\mathbb{N}}$
and $(v_{n})_{n\in\mathbb{N}}$, we show that we can compute arbitrarily
good approximations of their distance. To compute the distance between
$(u_{n})_{n\in\mathbb{N}}$ and $(v_{n})_{n\in\mathbb{N}}$ with an
error of at most $2^{-n}$, it suffices to enumerate the first $n$
digits of $(u_{n})_{n\in\mathbb{N}}$ and $(v_{n})_{n\in\mathbb{N}}$.
Then, if $(u_{n})_{n\in\mathbb{N}}$ and $(v_{n})_{n\in\mathbb{N}}$
agree on their first $n$ digits, it must be that $d((u_{n})_{n\in\mathbb{N}},(v_{n})_{n\in\mathbb{N}})<2^{-n}$,
and so $0$ constitutes a good approximation of $d((u_{n})_{n\in\mathbb{N}},(v_{n})_{n\in\mathbb{N}})$.
Otherwise, the first index after which the sequences $(u_{n})_{n\in\mathbb{N}}$
and $(v_{n})_{n\in\mathbb{N}}$ differ can be computed exactly, and
thus also the distance $d((u_{n})_{n\in\mathbb{N}},(v_{n})_{n\in\mathbb{N}})$. 

In both cases, the desired approximation of $d((u_{n})_{n\in\mathbb{N}},(v_{n})_{n\in\mathbb{N}})$
can be computed. 
\end{proof}
\begin{cor}
Let $\hat{\nu_{\mathcal{C}}}$ denote the natural numbering of a disjoint
union of Cantor spaces $\underset{i\in\mathbb{N}}{\bigcup}\{0,1\}^{\mathbb{N}}$
($\hat{\nu_{\mathcal{C}}}$ was defined in Section \ref{subsec:Numbering induced by Cantor space}).
Consider the metric $d$ that puts different copies of $\{0,1\}^{\mathbb{N}}$
at distance exactly $2$. Then the set $(\underset{i\in\mathbb{N}}{\bigcup}\{0,1\}^{\mathbb{N}})_{\hat{\nu_{\mathcal{C}}}}$
of computable points of $\underset{i\in\mathbb{N}}{\bigcup}\{0,1\}^{\mathbb{N}}$,
equipped with the distance $d$ and numbering $\hat{\nu_{\mathcal{C}}}$,
is a RMS. 
\end{cor}

\begin{proof}
Given the $\hat{\nu_{\mathcal{C}}}$ -names of two sequences, we decide
whether or not they belong to the same copy of the Cantor space. If
they don't, we know that their distance is $2$. Otherwise apply Proposition
\ref{prop:The-Cantor-space RMS}. 
\end{proof}
The following proposition is fundamental for us. Recall that $\mathcal{G}^{+}$
is the set of marked groups with solvable word problem. 
\begin{cor}
\label{cor:Space G RMS}The space of marked groups with solvable word
problem $\mathcal{G}^{+}$ equipped with its ultrametric distance
$d$ and with the numbering type $\Lambda_{WP}$ is a recursive metric
space. 
\end{cor}

\begin{proof}
This follows directly from the previous corollary, together with Proposition
\ref{prop: Subset of RMS}, which states that a subset of a RMS with
the induced numbering remains a RMS, and with the fact that the numbering
type $\Lambda_{WP}$ is the numbering induced on the space of marked
groups by the numbering type $\left[\hat{\nu_{\mathcal{C}}}\right]$
defined on a disjoint union of Cantor spaces, as detailed in Section
\ref{subsec:Numbering induced by Cantor space}. 
\end{proof}
The space of marked groups equipped with the distance $d_{Cay}$ defined
in Section \ref{sec:The-topological-space} and with the numbering
type $\Lambda_{WP}$ is also a RMS. We leave it to the reader to prove
this easy fact. 

\subsubsection{Effective completeness and effective separability }

Since the space of marked groups is a Polish space, it is natural
to ask whether $\mathcal{G}^{+}$ is a recursive Polish space, that
is, whether it is effectively complete and effectively separable. 

Here, we define those two notions, and give some properties that follow
from them. The importance of these notions lies in the facts that
Ceitin's Theorem is set on recursive Polish spaces. 
\begin{defn}
\label{def:eff Cauchy}A sequence $(u_{n})_{n\in\mathbb{N}}$ of computable
points in $X$ is called \emph{effectively convergent} if it converges
to a point $y\in X$, and if there exists a computable function $f:\mathbb{N}\rightarrow\mathbb{N}$
such that:
\[
\forall(n,m)\in\mathbb{N}^{2};n\ge f(m)\implies d(u_{n},y)\le2^{-m}.
\]

A sequence $(u_{n})_{n\in\mathbb{N}}$ of computable points in $X$
is called \emph{effectively Cauchy} if there exists a computable function
$f:\mathbb{N}\rightarrow\mathbb{N}$ such that:
\[
\forall(p,q,m)\in\mathbb{N}^{3};p,q\ge f(m)\implies d(u_{p},u_{q})\le2^{-m}.
\]

In both cases the function $f$ is called a \emph{regulator }for the
sequence $(u_{n})_{n\in\mathbb{N}}$. 
\begin{defn}
Let $(X,d,\nu)$ be a recursive metric space. An \emph{algorithm of
passage to the limit} (the name is from \cite{Kushner1984} and \cite{Spreen1998})
is an algorithm that takes as input a computable Cauchy sequence together
with a regulator for it, and produces the $\nu$-name of a point towards
which this sequence converges. 

A recursive metric space $(X,d,\nu)$ is\emph{ effectively complete}
if it admits an algorithm of passage to the limit. 
\end{defn}

\end{defn}

Note that even though we will in this paper focus exclusively on effectively
complete spaces when stating continuity theorems, a weaker condition
is sufficient to apply Markov's Lemma and the theorems of Ceitin and
Moschovakis: that there exist an algorithm of passage to the limit
which works only on converging sequences. Thus for instance the open
interval $(0,1)$ admits such an algorithm for $c_{\mathbb{R}}$,
even though it is not effectively complete, as the sequence $n\mapsto2^{-n}$
does not converge in $(0,1)$. As the space of marked groups is effectively
complete, we are not concerned here with theorems that do not rely
on effectively completeness. 

It is easy to see that any recursive metric space can be effectively
completed into an effectively complete metric space, we describe here
the construction. 

Recall that $(\varphi_{0},\varphi_{1},\varphi_{2},...)$ denotes an
effective enumeration of all partial computable functions. 
\begin{defn}
\label{def:Eff-Completion}Let $(X,d,\nu)$ be a recursive metric
space. Denote by $(\hat{X},d)$ the classical completion of $(X,d)$.
Denote by $j:X\hookrightarrow\hat{X}$ the embedding of $X$ into
$\hat{X}$. Define a subnumbering $\hat{\nu}$ of $\hat{X}$ by 
\[
\text{dom}(\hat{\nu})=\{i\in\mathbb{N},\,\forall p\in\mathbb{N},\forall q>p,\,d(\nu(\varphi_{i}(p)),\nu(\varphi_{i}(q)))<2^{-p}\};
\]
\[
\forall i\in\text{dom}(\hat{\nu}),\,\hat{\nu}(i)=\lim_{n\rightarrow\infty}j(\nu(\varphi_{i}(n))).
\]

We then obtain the effective completion of $(X,d,\nu)$ by restricting
our attention to the $\hat{\nu}$-computable points of $\hat{X}$:
the \emph{effective completion} of $(X,d,\nu)$ is the triple $(\hat{X}_{\hat{\nu}},d,\hat{\nu})$. 
\end{defn}

We do not prove the following well known result (see for instance
\cite[Corollary 2.13]{Spreen1998}): 
\begin{lem}
The numbering $\hat{\nu}$ defined above is the supremum for $\succeq$
of the sets of subnumberings of $\hat{X}$ that satisfy the following
conditions:
\begin{itemize}
\item the injection $j:X\hookrightarrow\hat{X}$ is $(\nu,\hat{\nu})$-computable;
\item the triple $(\hat{X}_{\hat{\nu}},d,\hat{\nu})$ is effectively complete. 
\end{itemize}
\end{lem}

\begin{prop}
Let $(X,d,\nu)$ be an effectively complete recursive metric space.
A closed subset $Y$ of $X$, together with the numbering induced
by $\nu$, is also an effectively complete metric space. 
\end{prop}

\begin{proof}
It suffices to notice that algorithm of passage to the limit for $X$
also works for $Y$. 
\end{proof}
The following proposition can be found in \cite{Kushner1984}: 
\begin{prop}
The set $\{0,1\}_{c}^{\mathbb{N}}$ is effectively complete. 
\end{prop}

This has the immediate corollary: 
\begin{cor}
\label{cor:G effectively complete}The recursive metric space $(\mathcal{G}^{+},d,\nu_{WP})$
is effectively complete. 
\end{cor}

This last fact could easily have been proved directly. We now describe
the effective notion associated to separability. 
\begin{defn}
A recursive metric space $(X,d,\nu)$ is called \emph{effectively
separable} if there exists a $\nu$-computable sequence $(u_{n})_{n\in\mathbb{N}}$
of points in $X$ that is dense in $X$. 
\end{defn}

Note that the numbering $\nu:\,\subseteq\mathbb{N}\rightarrow X$
is surjective, but not total. An equivalent formulation of the above
statement is the following: $(X,d,\nu)$ is effectively separable
if there exists a c.e. subset $A$ of $\mathbb{N}$ such that $A\subseteq\text{dom}(\nu)$
and $\nu(A)$ is dense in $X$. As the set $\text{dom}(\nu)$ could
be very complicated, there is no reason a priori for it to contain
even an infinite c.e. set. 

We can finally define recursive Polish spaces.
\begin{defn}
A recursive metric space $(X,d,\nu)$ which is both effectively complete
and effectively separable is a\emph{ recursive Polish space}. 
\end{defn}

Note that the above definition cannot be found in the modern literature
on computability of Polish spaces (see for instance \cite{Iljazovic2021}),
because the use of notions of ``presentations of metric spaces''
(see Definition \ref{def:RecPresPolish}) allows for a direct definitions
of computably Polish spaces. This approach does not permit to discuss
effective separability, because all defined spaces are effectively
separable. 
\begin{example}
The following spaces, equipped with their usual distances and numberings,
are recursive Polish spaces: $\mathbb{N}$, $\mathbb{R}_{c}$, the
set of computable points of the Cantor space $\{0,1\}_{c}^{\mathbb{N}}$,
the set of computable points of Baire Space $\mathbb{N}_{c}^{\mathbb{N}}$. 
\end{example}

The following proposition, while obvious, shall be very useful in
its group theoretical version. 
\begin{prop}
\label{prop:dense seq effective polish}Let $(X,d,\nu)$ be a recursive
metric space, and $Y$ be a $\nu$-c.e. set in $X$. Then any point
in the closure $\overline{Y}$ of $Y$ is the limit of a computable
sequence of points of $Y$ that converges effectively. 
\end{prop}

This proposition could have been phrased: any computable point in
the closure of $Y$ is automatically in its ``effective closure''. 
\begin{proof}
Given a point $x$ adherent to $Y$, we can define a computable sequence
$(v_{n})_{n\in\mathbb{N}}$ by: $v_{n}$ is the first element, in
a fixed enumeration of $Y$, which is proven to satisfy $d(x,v_{n})<2^{-n}$. 
\end{proof}
This proposition has the immediate corollary:
\begin{cor}
\label{prop:dense seq effective polish-1-1}In a recursive Polish
space with a dense and computable sequence $(u_{n})_{n\in\mathbb{N}}$,
each point is the limit of an effectively Cauchy and computable sequence
extracted from $(u_{n})_{n\in\mathbb{N}}$. 
\end{cor}

The following result is important in that it relates the less general
approach based on notions of recursive presentations of Polish spaces
(see Definition \ref{def:RecPresPolish}) to the approach based on
numberings. Indeed, using numberings, it is possible to define spaces
that are not necessarily effectively separable, and not effectively
complete. The idea behind the use of recursive presentations of Polish
spaces is to use the fact that the computable structure on a recursive
Polish space is entirely defined by the distance function between
elements of its dense sequence to define it in one single step. The
following theorem shows that we arrive to the same result via both
approaches.
\begin{thm}
\label{thm:Effective-Polish-def-seq}A recursive Polish space is computably
isometric to the effective completion of any of its computable and
dense sequences. 
\end{thm}

What this theorem means is in fact the following. Let $(X,d,\nu)$
be a recursive Polish space, and $(u_{n})_{n\in\mathbb{N}}$ a $\nu$-computable
and dense sequence. One can restrict $\nu$ to the dense sequence,
to obtain a recursive metric space $(\{u_{n},n\in\mathbb{N}\},d,\nu_{\vert(u_{n})_{n\in\mathbb{N}}})$.
One can then consider the effective completion of $(\{u_{n},n\in\mathbb{N}\},d,\nu_{\vert(u_{n})_{n\in\mathbb{N}}})$,
following Definition \ref{def:Eff-Completion}. 

One thus obtains a recursive Polish space $(A,d,\hat{(\nu_{\vert(u_{n})_{n\in\mathbb{N}}})})$.
The set $A$ can be seen as a subset of the abstract completion $\hat{X}$
of $X$. The numbering $\nu$ can also be seen as a subnumbering of
$\hat{X}$ via the embedding of $X$ into its classical completion.
What we then claim is: 
\[
\hat{(\nu_{\vert(u_{n})_{n\in\mathbb{N}}})}\equiv\nu.
\]

\begin{proof}
Let $(X,d,\nu)$ be a recursive Polish space, and $(u_{n})_{n\in\mathbb{N}}$
any $\nu$-computable and dense sequence. Denote $(\varphi_{0},\varphi_{1},\varphi_{2},...)$
an effective enumeration of all partial recursive functions. 

Denote by $\hat{X}$ the abstract completion of $X$, $X$ is seen
as a subset of $\hat{X}$. 

The effective completion (Definition \ref{def:Eff-Completion}) of
$(u_{n})_{n\in\mathbb{N}}$ is given by a subnumbering $\mu$ of the
abstract completion of the metric space $(\{u_{n},n\in\mathbb{N}\},d)$.
The embedding $\{u_{n},n\in\mathbb{N}\}\hookrightarrow X$ induces
an embedding of this completion inside $\hat{X}$. We thus see $\mu$
as a subnumbering of $\hat{X}$.

The fact that $(X,d,\nu)$ is effectively complete directly implies
that the image of $\mu$ in $\hat{X}$ lies in fact in $X$. We thus
consider that $\mu$ is a subnumbering of $X$. 

We show now that the subnumberings $\nu$ and $\mu$ are equivalent,
i.e. the identity on $X$ is both $(\nu,\mu)$-computable and $(\mu,\nu)$-computable. 

A $\mu$-name of a point $x$ in $X$ is the description of a computable
Cauchy sequence that converges to $x$, with $g:n\mapsto2^{-n}$ being
a regulator for this sequence. The algorithm of passage to the limit
of $(X,d,\nu)$ can thus be applied to this description with $g$
as regulator, and it yields precisely a $\nu$-name of $x$. This
shows that the identity on $X$ is $(\mu,\nu)$-computable. 

To show that it is also $(\nu,\mu)$-computable, one only has to notice
that the procedure described in the proof of Proposition \ref{prop:dense seq effective polish}
is uniform, in that it allows, given a $\nu$-name of a point $x$,
to produce a computable sequence extracted from $(u_{n})_{n\in\mathbb{N}}$
which converges to $x$ with the desired speed. This is precisely
a $\mu$-name for $x$. 
\end{proof}
The following corollary to Theorem \ref{thm:Effective-Polish-def-seq}
shows that when a Polish space can be equipped with a recursive Polish
space structure, this structure is unique. 
\begin{cor}
Given an abstract Polish space $(X,d)$ with a dense sequence $(u_{n})_{n\in\mathbb{N}}$,
there is at most one subnumbering type $\Lambda$ of $\mathcal{NT}_{X}$
which makes of $(X_{\Lambda},d,\Lambda)$ a recursive Polish space
with $(u_{n})_{n\in\mathbb{N}}$ as a computable and dense sequence. 
\end{cor}

\begin{proof}
This follows directly from Theorem \ref{thm:Effective-Polish-def-seq}. 
\end{proof}
Note that this corollary starts with an actual Polish space, possibly
uncountable. The countable set of points that are the $\Lambda$-computable
points is uniquely given by the theorem. Theorem \ref{thm:Effective-Polish-def-seq}
allows for a very simple definition of what is an ``effective Polish
space'', that relies only on the distance between the elements of
a dense sequence. Weihrauch and Moschovakis both used such definitions.
We give here a definition that mimics that of Moschovakis, see \cite{GREGORIADES2016}
for the complete definitions of Weihrauch and Moschovakis, and their
differences. Note however that the definition that we give is weaker
than the ones given by Weihrauch and Moschovakis -a space that admits
a recursive presentation in this sense also admits one following the
definitions of Weihrauch and Moschovakis. Note also that the following
definition depends on a choice of a distance for the considered Polish
space, and not only on its topology.
\begin{defn}
\label{def:RecPresPolish}A \emph{recursive presentation of a Polish
space} $(X,d)$ is a dense sequence $(u_{n})_{n\in\mathbb{N}}$ of
points of $X$ such that the function 
\[
\begin{aligned}\phi:\mathbb{N}\times\mathbb{N} & \rightarrow\mathbb{R}_{c}\\
(n,m) & \mapsto d(u_{n},u_{m})
\end{aligned}
\]
is $(\text{id}_{\mathbb{N}}\times\text{id}_{\mathbb{N}},c_{\mathbb{R}})$-computable.
\end{defn}

Note that the concept above applies to an actual Polish space, and
not to a countable set of points. 

The term presentation is from \cite{Moschovakis1980}, and has no
relation the notion of a presentation for a group. The following proposition
renders explicit the link between recursive presentations and recursive
Polish spaces. 
\begin{prop}
\label{prop:REC pres and EPR}A Polish space $(X,d)$ admits a recursive
presentation if and only if it admits a numbering $\nu$ with dense
image that makes of $(X_{\nu},d,\nu)$ a recursive Polish space. 
\end{prop}

\begin{proof}
If $\nu$ is a numbering of $X$ that makes of $(X_{\nu},d,\nu)$
a recursive Polish space, it means that there exists a $\nu$-computable
dense sequence $(u_{n})_{n\in\mathbb{N}}$. The function $\phi$ defined
by $\phi(n,m)=d(u_{n},u_{m})$ is then computable, and thus $(u_{n})_{n\in\mathbb{N}}$
defines a recursive presentation of $(X,d)$. 

Conversely, if $(u_{n})_{n\in\mathbb{N}}$ is a sequence which is
dense in $X$, the condition that $(n,m)\mapsto d(u_{n},u_{m})$ be
computable exactly asks that the function $$\begin{aligned} \mu : \mathbb{N} &\rightarrow X \\ n &\mapsto u_n \end{aligned}$$
define a numbering of $X$ which makes of $(X_{\mu},d,\mu)$ a recursive
metric space. Then, the effective completion of $\mu$ defines a numbering
$\nu$ of $X$, for which $(X_{\nu},d,\nu)$ is a recursive Polish
space. And $X_{\nu}$ is dense in $X$. 
\end{proof}
In Section \ref{sec:Limits-of-applicability}, we prove that the space
of marked group, associated to its usual ultrametric distance, does
not have a recursive presentation. We will also prove that the space
of marked groups does not contain dense and computable sequences of
groups described by word problem algorithms, but the former result
is more general, because we do not suppose \emph{a priori} that a
dense sequence should consist in groups described by word problem
algorithms. 

\subsection{\label{subsec:Markov's-Lemma-and}Markov's Lemma and abstract continuity }

We will give here a proof of the fact that Banach-Mazur computable
functions on a recursive Polish space are continuous, starting with
Markov's Lemma, which is both very useful and very simple to use,
and which will remain our main tool in the space of marked groups,
since the effective continuity theorems that we present in Section
\ref{subsec:KLS-Ceitin-Theorem,} are not applicable there. 

We fix a recursive metric space $(X,d,\nu)$ which we suppose effectively
complete. Denote by $\mathcal{A}_{lim}$ an algorithm of passage to
the limit for it. 
\begin{lem}
[Markov, \cite{Markov54}, English version: \cite{Markov1963}, Theorem
4.2.2]\label{lem:Markov} Suppose that a $\nu$-computable sequence
$(u_{n})_{n\in\mathbb{N}}$ effectively converges in $X$ to a point
$x$. Suppose additionally that for any $n$, $u_{n}\ne x$. Then
there is a $\nu$-computable sequence $(w_{p})_{p\in\mathbb{N}}$
of $X^{\mathbb{N}}$ such that: for each $p$, $w_{p}\in\{u_{n},n\in\mathbb{N}\}\cup\{x\}$,
and the set $\{p,w_{p}=x\}\subseteq\mathbb{N}$ is co-c.e. but not
c.e.. 
\end{lem}

\begin{proof}
Consider an enumeration of all Turing machines $M_{0}$, $M_{1}$,
... To the machine $M_{p}$, we associate a computable sequence $(x_{n}^{p})_{n\in\mathbb{N}}$
of points in $X$. To define $(x_{n}^{p})_{n\in\mathbb{N}}$, start
a run of the machine $M_{p}$ with no input. While it lasts, the sequence
$(x_{n}^{p})_{n\in\mathbb{N}}$ is identical to the sequence $(u_{n})_{n\in\mathbb{N}}$.
If, at some point, the machine $M_{p}$ stops, the sequence $(x_{n}^{p})_{n\in\mathbb{N}}$
becomes constant. 

To sum this definition up, $(x_{n}^{p})_{n\in\mathbb{N}}$ is defined
as follows: 

\textbf{While} ($M_{p}$ does not stop) enumerate $(u_{n})_{n\in\mathbb{N}}$. 

\textbf{If} ($M_{p}$ stops in $k$ computation steps), set $x_{n}^{p}=u_{k}$
for $n\ge k$. 

Each sequence $(x_{n}^{p})_{n\in\mathbb{N}}$ is Cauchy, and in fact
it converges at least as fast as the original sequence $(u_{n})_{n\in\mathbb{N}}$.
Thus the algorithm of passage to the limit $\mathcal{A}_{lim}$ can
be applied to any sequence $(x_{n}^{p})_{n\in\mathbb{N}}$, using
the regulator of convergence of $(u_{n})_{n\in\mathbb{N}}$. 

The sequence $(w_{p})_{p\in\mathbb{N}}$ is the sequence obtained
by using the algorithm of passage to the limit on each sequence $(x_{n}^{p})_{n\in\mathbb{N}}$,
for $p\in\mathbb{N}$. 

If follows directly from our definitions that if the machine $M_{p}$
is non-halting, the sequence $(x_{n}^{p})_{n\in\mathbb{N}}$ is identical
to $(u_{n})_{n\in\mathbb{N}}$, and thus $w_{p}$, which is its limit,
is equal to $x$. On the other hand, if $M_{p}$ halts in $k$ computation
steps, we have $w_{p}=u_{k}$ and thus $w_{p}$ is different from
$x$. 
\end{proof}
Lemma \ref{lem:Markov's-Lemma-for Groups Intro} is an immediate corollary
of the above together with Corollaries \ref{cor:Space G RMS} and
\ref{cor:G effectively complete}. 
\begin{cor}
\label{cor:MarkovCor1}Let $f$ be a Banach-Mazur computable function
between recursive metric spaces $X$ and $Y$, suppose that $X$ is
effectively complete, and let $(x_{n})_{n\in\mathbb{N}}$ be a computable
sequence that effectively converges to a point $x$ in $X$. Then
the sequence $(f(x_{n}))_{n\in\mathbb{N}}$ converges to $f(x)$. 
\end{cor}

\begin{proof}
We proceed by contradiction. Suppose that the sequence $(f(x_{n}))_{n\in\mathbb{N}}$
does not converge to $f(x)$. Then there must exist a subsequence
$(x_{\phi(n)})_{n\in\mathbb{N}}$ of $(x_{n})_{n\in\mathbb{N}}$ and
a rational $r>0$ such that 
\[
\forall n\in\mathbb{N},\,d(f(x_{\phi(n)}),f(x))>r.
\]
The existence of such a sequence, which need not a priori be computable,
implies that there must also exist such a sequence where, additionally,
the function $\phi:\mathbb{N}\rightarrow\mathbb{N}$ is computable. 

This follows from Proposition \ref{prop:dense seq effective polish},
as the set of terms of the sequence $(x_{n})_{n\in\mathbb{N}}$ for
which $d(f(x_{n}),f(x))>r$ holds is a c.e. set.

Finally, the function $f$ can be used to distinguish between the
elements of the sequence $(x_{\phi(n)})_{n\in\mathbb{N}}$ and its
limit $x$, as, given a computable point $u$ in $X$, it is possible
to chose one which is true between $d(f(u),f(x))>r$ and $d(f(u),f(x))<r$,
if we know a priori that $d(f(u),f(x))$ is not equal to $r$. This
contradicts Markov's Lemma, and thus $(f(x_{n}))_{n\in\mathbb{N}}$
must converge to $f(x)$. 
\end{proof}
We can use the previous corollary to prove that, under the additional
assumption that the space $X$ be a recursive Polish space, any Banach-Mazur
computable function $f:X\rightarrow Y$ is continuous -not necessarily
effectively so. This corollary of Markov's Lemma was first proven
by Mazur for functions defined on intervals of $\mathbb{R}_{c}$ (see
\cite{Mazur63}). See \cite{Hertling2001}. 
\begin{cor}
[Mazur's Continuity Theorem]\label{cor:Continuity-Theorem}Consider
a Banach-Mazur computable function $f:X\rightarrow Y$ between a recursive
Polish space $X$ and a recursive metric space $Y$. Then $f$ is
continuous. 
\end{cor}

\begin{proof}
Denote $(u_{n})_{n\in\mathbb{N}}$ an effective and dense sequence
of $X$. 

Suppose that $f$ is not continuous at a point $x$ of $X$. This
means that there exists a sequence $(x_{n})_{n\in\mathbb{N}}$ and
a real number $r>0$ such that $(x_{n})_{n\in\mathbb{N}}$ converges
to $x$, but for any $n\in\mathbb{N}$, $d(f(x_{n}),f(x))>r$. Any
point of $(x_{n})_{n\in\mathbb{N}}$ is the limit of an effective
subsequence of $(u_{n})_{n\in\mathbb{N}}$, by Proposition \ref{prop:dense seq effective polish}.
And thus, by Corollary \ref{cor:MarkovCor1}, for each point $x_{k}$
of the sequence $(x_{n})_{n\in\mathbb{N}}$, there must exist a point
$u_{\phi(k)}$ in the sequence $(u_{n})_{n\in\mathbb{N}}$, such that
both inequalities $d(u_{\phi(k)},x_{k})<2^{-k}$ and $d(f(u_{\phi(k)}),f(x))>r$
hold. 

Thus there exists a subsequence $(u_{\phi(n)})_{n\in\mathbb{N}}$
of $(u_{n})_{n\in\mathbb{N}}$, which converges to $x$ and such that
for any $n\in\mathbb{N}$, $d(f(u_{\phi(n)}),f(x))>r$. 

This subsequence is a priori not computable, but by Proposition \ref{prop:dense seq effective polish},
the abstract fact that such a sequence exists automatically implies
that there must also exist such a subsequence that is, in addition,
both computable and effectively converging to $x$. 

Finally, we conclude by applying Corollary \ref{cor:MarkovCor1} again. 
\end{proof}

\subsection{\label{subsec:Differences-with-Borel}Computable but discontinuous
functions and Kolmogorov complexity}

We give here an example of a Markov computable function that is not
continuous and of a semi-decidable set in a recursive Polish space
which is not open. The fact that those exist is well known, we explain
it here in terms of Kolmogorov complexity, following Hoyrup and Rojas
\cite{Hoyrup2016}. 

We also render explicit how the main theorem of \cite{Hoyrup2016}
gives necessary and sufficient conditions for a computable discontinuous
Markov computable function to exist on a computably second countable
space. 

\subsubsection{Brief definition of the Kolmogorov complexity }

Kolmogorov complexity formalizes the idea of ``size of the minimal
description of a string''. This is in fact not well defined for a
single string, but it is defined ``up to a constant''. 

See Chapter 1 of \cite{Shen2017} for a better introduction to Kolmogorov
complexity. We here just give the definitions and results we need.
Denote by $\{0,1\}^{*}$ is the set of binary strings. Denote by $l(x)$
the length of an element of $\{0,1\}^{*}$. 

A partial computable map $F:\,\subseteq\{0,1\}^{*}\rightarrow\{0,1\}^{*}$
is called a \emph{decompressor. }

We then define \emph{the} \emph{complexity with respect to $F$}:
\[
C_{F}(x)=\min\{l(y),\,y\in\{0,1\}^{*}\,\&\,F(y)=x\}.
\]
It is infinite if $x$ is not in the image of $F$.

A decompressor $F$ is called \emph{not worse} than another decompressor
$G$ if there exists a constant $k$ such that 
\[
\forall x\in\{0,1\}^{*},\,C_{F}(x)\le C_{G}(x)+k.
\]

\begin{thm}
[Solomonoff--Kolmogorov, see \cite{Shen2017}, Theorem 1, p. 3]
There is a decompressor not worse than every other decompressor. 
\end{thm}

Such a decompressor is called \emph{optimal}. Let $(x,y)\mapsto\langle x,y\rangle$
denote a computable bijection between $\{0,1\}^{*}$ and $\{0,1\}^{*}\times\{0,1\}^{*}$.
One obtains an optimal decompressor simply by using a universal Turing
machine: if $U$ is the function computed by a universal Turing machine
(which means that $U(x,y)$ interprets $x$ as the code of a Turing
machine and applies this machine to $y$), then $\langle x,y\rangle\mapsto U(x,y)$
is an optimal decompressor. 

We fix an optimal decompressor $F_{0}$ and define \emph{the Kolmogorov
complexity, }denoted $C$, to be $C_{F_{0}}$, the complexity with
respect to this optimal decompressor. 

The identity $\text{id}_{\{0,1\}^{*}}$ of $\{0,1\}^{*}$ is a decompressor
(which is not optimal), the complexity $C_{\text{id}_{\{0,1\}^{*}}}$
is just the length function. Thus there exists $k$ such that for
any $x$, $C_{F_{0}}(x)\le l(x)+k$. This shows that the length is
an upper bound to Kolmogorov complexity. 

Conversely, strings of maximal Kolmogorov complexity do exit. 
\begin{thm}
[See \cite{Shen2017}, Theorem 5, p.8]\label{prop:String Max K comp}For
any $n>0$, there is a string of length at most $n$ and of Kolmogorov
complexity $n$. 
\end{thm}

The sequence given by this theorem cannot be computable. 
\begin{thm}
[See \cite{Shen2017}, Theorem 10, p.22]\label{prop:K comp not comp}Kolmogorov
complexity is not computable. What's more, there is no computable
sequence of strings $(x_{n})_{n\in\mathbb{N}}$ such that $C(x_{n})>\frac{l(x_{n})}{2}$.
\end{thm}

Applying computable functions does not change the asymptotic Kolmogorov
complexity of sequences: 
\begin{thm}
[See \cite{Shen2017}, Theorem 3, p.5]\label{thm:Apply comp transformation K comp}For
any partial computable function $g:\,\subseteq\{0,1\}^{*}\rightarrow\{0,1\}^{*}$,
there is $k$ such that $C(g(x))\le C(x)+k$ for each $x$ in $\text{dom}(g)$. 
\end{thm}

We will finally need the following result:
\begin{thm}
[See \cite{Shen2017}, Theorem 8, p.19]\label{prop:K comp upper semi comp}Kolmogorov
complexity is upper semi-computable: there is a computable function
$F:\{0,1\}^{*}\times\mathbb{N}\rightarrow\mathbb{N}$ such that for
any $x$ in $\{0,1\}^{*}$, $n\mapsto F(x,n)$ is decreasing and converges
to $C_{F_{0}}(x)$. 
\end{thm}

\subsubsection{Computable but discontinuous function\label{subsec:Computable-but-Discontinuous}}

We first build a computable but discontinuous function. This is done
by considering a function defined on a peculiar domain. We set ourselves
in the Cantor space $\left\{ 0,1\right\} ^{\mathbb{N}}$, we will
consider a certain subset of it. 

Consider a sequence $\left(u_{n}\right)_{n\in\mathbb{N}}$ of finite
strings of zeroes and ones, such that the length of $u_{n}$ is $n$,
and which has linear asymptotic Kolmogorov complexity: $K(u_{n})\underset{n\rightarrow\infty}{\sim}n$.
This is given by Theorem \ref{prop:String Max K comp}. Consider now
the sequence $v_{n}=0^{n}1u_{n}00000.....$ of elements of $\left\{ 0,1\right\} ^{\mathbb{N}}$.
This sequence also has linear asymptotic Kolmogorov complexity: a
single Turing Machine can transform any element $v_{n}$ into the
corresponding $u_{n}$, thus we can apply Theorem \ref{thm:Apply comp transformation K comp}.

We call $\mathcal{A}$ the subset of $\left\{ 0,1\right\} ^{\mathbb{N}}$
consisting of the null sequence (which we denote by $0^{\omega}$)
and of the set $\left\{ v_{n},n\in\mathbb{N}\right\} $. 
\begin{prop}
The function $\delta_{0}:\mathcal{A}\rightarrow\left\{ 0,1\right\} $
which sends the null sequence to $1$ and all other sequences to $0$
is computable on $\mathcal{A}$. However, it is discontinuous. 
\end{prop}

\begin{proof}
Because $K(v_{n})\underset{n\rightarrow\infty}{\sim}n$, there must
exist an integer $b\in\mathbb{Z}$ such that for all $n$, $K(v_{n})>\frac{n}{2}+b$. 

We now show how to compute $\delta_{0}$. 

An element $x$ of $\mathcal{A}$ is given as input of $\delta_{0}$
via a finite description, which is the code $k$ of a Turing machine
that enumerates $x$. This code precisely constitutes an upper bound
to the Kolmogorov complexity of $x$. Thus either $x$ it is the null
sequence, or, if it can be written $v_{n}$ for some $n$, we have
$k>\frac{n}{2}+b$, and thus $n<2(k-b)$. 

But since the element $v_{n}$ agrees with the null sequence only
on its first $n$ terms, this means that if $x$ is not the null sequence,
one of its first $2(k-b)$ digits must be a one. This can be easily
checked, using the Turing Machine coded by $k$ until it has written
the first $2(k-b)$ digits of $x$.
\end{proof}

\subsubsection{A semi-decidable set that is not open }

The previous example was obtained by considering a function defined
on a set with bad properties. On a recursive Polish space, the decidable
sets must be clopen, since their characteristic functions must be
continuous. By Markov's Lemma, the semi-decidable sets cannot be ``effectively
not-open'': if $x$ is a point of a semi-decidable set $Y$, and
if $(u_{n})_{n\in\mathbb{N}}$ is a computable sequence that effectively
converges to $x$, then infinitely many elements of this sequence
must belong to $Y$. One might wonder whether this result can be strengthened
to: ``the $\nu$-semi-decidable sets on a recursive Polish space
$(X,d,\nu)$ are open''. An example of Friedberg \cite{Friedberg1958}
shows that this is not the case, we reproduce here the account of
this result from \cite{Hoyrup2016}, which renders explicit the role
of Kolmogorov complexity in the construction of this example. 

This example is set in the Cantor space $\{0,1\}^{\mathbb{N}}$. For
$w$ an element of $\{0,1\}^{*}$, denote by $[w]$ the clopen set
of all sequences that start by $w$. 
\begin{thm}
[Friedberg, see \cite{Hoyrup2016}, Theorem 4.1]On the Cantor space,
the set 
\[
A=\{0^{\omega}\}\cup\underset{\{n\,\vert\,K(0^{n})<\frac{\log(n)}{2}\}}{\bigcup}[0^{n}1]
\]
 is semi-decidable but not open.
\end{thm}

\begin{proof}
$A$ is not open, as it does not contain a neighborhood of $0^{\omega}$,
because infinitely often in $n$ one has $K(0^{n})\ge\frac{\log(n)}{2}$
(by Theorem \ref{prop:String Max K comp}, noting that the complexity
of the string $0^{n}$ is equivalent to the complexity of the binary
expansion of $n$, since it is possible to computably go from $0^{n}$
to the binary expansion of $n$, and conversely). 

We now show that $A$ is semi-decidable. 

There exists a program $T$ that maps any element $x$ of the Cantor
space that is different from $0^{\omega}$ to the number of zeroes
that appear at the beginning of $x$. 

We are now given a computable point $x$ of $\{0,1\}^{\mathbb{N}}$. 

The description of $x$ by a Turing machine that produces it gives
an upper bound $K_{0}$ on the Kolmogorov complexity of $x$. Noting
$l$ the length of the program $T$ defined above, one has that either
$x$ is $0^{\omega}$, or, if $x$ can be decomposed as $x=0^{n}1x'$,
it must be that $K(0^{n})\le K(x)+l\le K_{0}+l$. 

Start enumerating $x$. If $x$ starts with more than $2^{2(K_{0}+l)}$
zeros, then either it is the sequence $0^{\omega}$, or it can be
written as $x=0^{n}1x'$, with $\frac{\log(n)}{2}>K_{0}+l$, and thus
with $\frac{\log(n)}{2}>K(0^{n})$. In any case, we know that $x$
belongs to $A$, without having to compute $n$. 

If $x$ starts with less than $2^{2(K_{0}+l)}$ zeroes, a number $n$
such that $x$ can be rewritten as $x=0^{n}1x'$ can be effectively
found. From this, to determine whether $x$ belongs to $A$, one only
needs to check whether $K(0^{n})<\frac{\log(n)}{2}$ holds. This last
inequality is semi-decidable, as the Kolmogorov complexity is upper-semi-computable
(Theorem \ref{prop:K comp upper semi comp}). 
\end{proof}
Other examples of non-open but semi-decidable sets can be found in
\cite{Hoyrup2016}. It is however clear that those examples are artificially
built, and this justifies the heuristic which says that a\emph{ natural}
semi-decidable property can be expected to be open. 

Note finally that although we have just seen that a semi-decidable
subset of a recursive Polish space does not have to be open, it must
share the following property of open sets: it meets any computable
and dense sequence. 
\begin{prop}
[Moschovakis, \cite{Moschovakis1964}, Theorem 4] \label{prop: SD set meets dense seq }Let
$(X,d,\nu,(u_{n})_{n\in\mathbb{N}})$ be a recursive Polish space.
A non-empty $\nu$-semi-decidable subset of $X$ must intersect the
dense sequence \textup{$(u_{n})_{n\in\mathbb{N}}$}. 
\end{prop}

The proof of this result is in fact very close to that of Mazur's
Continuity Theorem, Corollary \ref{cor:Continuity-Theorem}. A consequence
of this fact, pointed out in \cite{Hoyrup2016}, is the following: 
\begin{cor}
[Hoyrup, Rojas, \cite{Hoyrup2016}, Section 4, Proposition 3]\label{cor:HoyrupRojas NonEmpty is SD}In
a recursive Polish space $(X,d,\nu,(u_{n})_{n\in\mathbb{N}})$, there
is an algorithm that takes as input the code for a $\nu$-semi-decidable
set $D$ and stops if and only if this set is non-empty. 

In case the set $D$ is non-empty, this algorithm will produce the
$\nu$-name for a point in it. 
\end{cor}

\begin{proof}
Just enumerate the sequence $(u_{n})_{n\in\mathbb{N}}$ in search
of a point in $D$, by Proposition \ref{prop: SD set meets dense seq },
$D$ is non-empty if and only if it contains a point from $(u_{n})_{n\in\mathbb{N}}$. 
\end{proof}
Note that the previous proposition contains an effective version of
the Axiom of Choice for recursive Polish spaces: a possible formulation
of AC is that for any set $X$, there exists a choice function that
maps a non-empty subset of $X$ to a point in this subset. The above
corollary shows that a computable choice function exists for semi-decidable
subsets of a recursive Polish space. We will give in Theorem \ref{thm:No-Completion-Theorem}
a strong negation of such an effective Axiom of Choice for the space
of marked groups: there does not exist a computable function that,
given a non-empty basic clopen set $\Omega_{r_{1},...,r_{m};s_{1},...,s_{m'}}^{k}$,
can produce a point in this set. 

\subsubsection{Necessary and sufficient condition for discontinuous functions to
exist\label{subsec:Necessary-and-sufficient-conditions for continuity }}

We show how the example given in Section \ref{subsec:Computable-but-Discontinuous}
above is essentially the only possible one on spaces that satisfy
a certain form of ``effective second countability'', which is easily
seen to be satisfied by the space of marked groups. 

Consider a countable $T_{0}$ topological space $X$, which is second
countable with a totally numbered basis $(B_{i})_{i\in\mathbb{N}}$.
We suppose w.l.o.g. that $(B_{i})_{i\in\mathbb{N}}$ is stable by
finite intersection and that the intersection map is computable. 

Equip $X$ with the numbering $\nu$ induced by the basis $(B_{i})_{i\in\mathbb{N}}$:
\[
\nu(i)=x\iff\text{dom}(\varphi_{i})=\{n\in\mathbb{N},\,x\in B_{n}\}.
\]
In other words, the name of a point in $X$ is a program that recognizes
the basic open sets to which this point belongs. These assumptions
amount to a form of ``effective second countability''. 

We say that points in $X$ have \emph{effective neighborhood bases
of co-semi-decidable sets }if for each point $x$ there is a neighborhood
basis $(A_{i})_{i\in\mathbb{N}}$ of $x$ consisting of co-semi-decidable
sets, and a computable map $F_{x}$ such that if $x\in B_{i}$, then
$x\in A_{F_{x}(i)}\subseteq B_{i}$. The typical example of this is
metric spaces, where closed balls are co-semi-decidable, and where
the map $F_{x}$ can be chosen to map the open ball of radius $r$
centered at $x$ to the closed ball of radius $r/2$ centered at $x$.
In the space of marked groups the basic clopen sets are naturally
decidable, so $F_{x}$ can be taken as the identity. 

The following result was indicated to me by Mathieu Hoyrup. It follows
from \cite{Hoyrup2016} but does not appear there. 
\begin{thm}
\label{thm:Continuity IFF Kolmogorov }If there exists a $\nu$-computable
and discontinuous function $f:X\rightarrow\{0,1\}$, then there exists
a non-isolated point $x$ and a computable function $g:\mathbb{N}\rightarrow\mathbb{N}$
such that for each $n$, $x$ is the only point of Kolmogorov complexity
less than $n$ in $B_{g(n)}$. If furthermore points in $X$ have
effective neighborhood bases of co-semi-decidable sets, then this
is actually an equivalence: any such $g$ yields a computable discontinuous
function $f:X\rightarrow\{0,1\}$. 
\end{thm}

\begin{proof}
By the main result of \cite{Hoyrup2016}, the function $f$ is $\nu$-computable
if and only if it is computable in a Type 2 model where points are
given with, as additional information, an upper bound on the Kolmogorov
complexity of the input point. This means that there is a machine
$M_{f}$ that computes $f$ which, instead of taking $\nu$-names
as input (these would be elements of $\mathbb{N}$), takes pairs $(k,p)$,
$k\in\mathbb{N}$ and $p\in\mathbb{N}^{\mathbb{N}}$ as input, where
a pair $(k,p)$ describes a point $x$ if: 
\begin{itemize}
\item $k$ is such that $x$ admits a $\nu$-name $i$ with $i\le k$; 
\item $p$ is a sequence consisting of the set $\{n\in\mathbb{N},\,x\in B_{n}\}$
in some order. This sequence is given on an oracle tape of the Turing
machine. 
\end{itemize}
Suppose that $f$ is discontinuous at $x$. Fix a computable sequence
$(p_{n})_{n\in\mathbb{N}}$ which is such that $\{p_{n},\,n\in\mathbb{N}\}=\{n\in\mathbb{N},\,x\in B_{n}\}$.
Let $k_{0}$ be a $\nu$-name of $x$. 

Now consider what the machine $M_{f}$ does on input a pair $(k,(p_{n})_{n\in\mathbb{N}})$,
for $k>k_{0}$. This pair is a description of $x$, thus $M_{f}$
will eventually answer $f(x)$. However, when doing so, the machine
$M_{f}$ inspects only a finite portion of the infinite sequence $(p_{n})_{n\in\mathbb{N}}$. 

The extent to which this sequence is inspected is a computable function
of $k$: there is $g_{1}$, computable, which is such that $g_{1}(k)=N$
if on input $(k,(p_{n})_{n\in\mathbb{N}})$, $M_{f}$ inspects only
the first $N$ elements of the sequence $(p_{n})_{n\in\mathbb{N}}$. 

It follows from this that if a point $y$ has Kolmogorov complexity
at most $k$, then either it is $x$, or it does not belong to $B_{p_{1}}\cap...\cap B_{p_{g_{1}(k)}}$.
We have supposed that intersections are computable for the basis $(B_{i})_{i\in\mathbb{N}}$,
and thus we can replace $g_{1}$ by $g$ as in the statement of the
theorem. 

For the converse, suppose that such a $g$ exists, and that $x$ has
a computable neighborhood basis of co-c.e. sets. Then the function
$\delta_{x}:X\rightarrow\{0,1\}$ which maps $x$ to $1$ and other
points to $0$ is computable. Indeed, given $y=\nu(n)$, we know that
either $y=x$, or $y\notin B_{g(n)}$. By hypothesis we can compute
a co-c.e. set $A_{m}$ such that $x\in A_{m}\subseteq B_{g(n)}$.
Then either $y\in B_{g(n)}$ or $y\notin A_{m}$, both conditions
are semi-decidable, and thus it can be determined which one of them
holds. 
\end{proof}
The following is the above result reformulated for metric spaces:

\begin{thm}
Suppose that $X$ has a computable metric, and that the basis $(B_{i})_{i\in\mathbb{N}}$
generates the metric topology (without necessarily being a set of
open balls). Suppose furthermore every point has a c.e. neighborhood
basis $(B_{\psi(i)})_{i\in\mathbb{N}}$ with $\text{diam}(B_{\psi(i)})<2^{-n}$. 

Then there exists a $\nu$-computable and discontinuous function $f:X\rightarrow\{0,1\}$
iff there is a point $x$, which is not isolated, and a computable
function $g$ s.t. for all $y\ne x$, if $y=\nu(n)$, then $d(y,x)>2^{-g(n)}$. 
\end{thm}

\begin{proof}
Similar to the proof of Theorem \ref{thm:Continuity IFF Kolmogorov },
using as neighborhood basis of $x$ the sequence $(B_{\psi(i)})_{i\in\mathbb{N}}$.
\end{proof}
The hypotheses of the above Theorem are true of the space of marked
groups. The basic clopen sets $\Omega_{R,S}^{k}$ are not open balls
in the space of marked groups (since they can be empty), yet an upper
bound to their diameter can easily be computed. 

Note that despite being far reaching, Theorem \ref{thm:Continuity IFF Kolmogorov }
does not completely solve the continuity problem. Indeed: 
\begin{itemize}
\item It relies on a hypothesis of effective second countability; 
\item It does not provide necessary and sufficient conditions for effective
continuity of the considered functions (see \cite{Slisenko1964} for
a continuous but not effectively continuous computable function);
\item It does not \emph{a priori} apply to Banach-Mazur continuity. It is
however possible that the methods of \cite{Hoyrup2016} do in fact
apply also to Banach-Mazur computability. 
\end{itemize}

\subsection{\label{subsec:KLS-Ceitin-Theorem,}Kreisel, Lacombe and Schoenfield
and Ceitin Theorems, Moschovakis' addendum }

\subsubsection{Theorems of Kreisel, Lacombe, Schoenfield, Ceitin}

The following theorem is one of the most important theorems in Markovian
computable analysis. It was first proved by Kreisel, Lacombe and Schoenfield
in 1957 in \cite{KLS57} in the case of functions defined on the Baire
space $\mathbb{N}^{\mathbb{N}}$, and obtained independently by Ceitin
in 1962 (English translation in \cite{Ceitin1967}), in the more general
setting of recursive Polish spaces. 
\begin{thm}
[Kreisel-Lacombe-Schoenfield, Ceitin] A computable function defined
on a recursive Polish space (with images in any RMS) is effectively
continuous. 

Moreover, for each pair constituted of a recursive Polish space and
a RMS, there is an algorithm that takes as input the description of
a computable function defined between those spaces, and produces a
program that will attest for the effective continuity of this function. 
\end{thm}

\subsubsection{Moschovakis' Theorem}

In 1964, Moschovakis gave a new proof of Ceitin's Theorem, at the
same time providing the only known effective continuity result for
metric spaces set in a more general context than that of recursive
Polish spaces. 

In what follows, $(X,d,\nu)$ denotes a recursive metric space. 
\begin{defn}
\label{def:Moschovakis (B) condition}We say that $(X,d,\nu)$ satisfies
\emph{Moschovakis' condition (B) }if there is an algorithm that, given
the code of a $\nu$-semi-decidable set $A\subseteq X$ and an open
ball $B(x,r)$ which intersects it, will produce the $\nu$-name of
a point $y$ in the intersection $A\cap B(x,r)$. 
\end{defn}

Note that in this definition, the algorithm is always given as input
a pair of intersecting sets, it then produces a point in the intersection.
This algorithm is not supposed to be able to determine whether or
not two given sets intersect. This definition asks for an effective
Axiom of Choice, similar to the one described in Proposition \ref{prop: SD set meets dense seq }.
And an easy consequence of Proposition \ref{prop: SD set meets dense seq }
is the following:
\begin{prop}
A recursive Polish space satisfies Moschovakis' condition (B). 
\end{prop}

We can now state Moschovakis' Theorem on the effective continuity
of computable functions. 
\begin{thm}
[Moschovakis, \cite{Moschovakis1964}, Theorem 3] \label{thm:Moschovakis}A
computable function defined on an effectively complete RMS that satisfies
Moschovakis' condition (B) is effectively metric continuous. 

Moreover, for each pair constituted of such a space and of any RMS,
there is an algorithm that takes as input the description of a computable
function defined between those spaces, and produces a program that
will attest for the effective continuity of this function. 
\end{thm}

We will prove in Corollary \ref{cor:MoschoFails} that the hypotheses
of this theorem fail for the space of marked groups, leaving open
the conjecture which says that computable functions on $\mathcal{G}^{+}$
are effectively continuous.

Note that several other results that can be found in \cite{Moschovakis1964}
can give rise to conjectures for the space of marked groups, in particular
the characterization of effective open sets as Lacombe sets, which
is Theorem 11 of \cite{Moschovakis1964}. See for instance Theorem
\ref{thm:LEF not co LACOMBE} and the discussion around it. 

\section{\label{sec:Limits-of-applicability}The space of marked groups as
a recursive metric space}

\subsection{Banach-Mazur computable but not Markov computable functions }
\begin{lem}
For every $k\ge2$, there exists a computable retraction from $\mathcal{G}_{k}$
onto a Cantor space: there exists a computable injection $i:\{0,1\}^{\mathbb{N}}\rightarrow\mathcal{G}_{k}$
which admits a computable right inverse $r:\mathcal{G}_{k}\rightarrow\{0,1\}^{\mathbb{N}}$. 
\end{lem}

\begin{proof}
Many constructions are possible. We use Philip Hall's 2-generated
center-by-metabelian group which admits $\bigoplus_{\mathbb{N}}\mathbb{Z}$
as center, which gave the first constructions of continuously many
2 generated solvable groups \cite{Hall1954}. Let $a$ and $b$ be
its generators. For each $n\in\mathbb{N}$, denote by $w_{n}$ a word
of $\{a,b,a^{-1},b^{-1}\}^{*}$ that defines in $G$ a generator of
the $n$-th central copy of $\mathbb{Z}$. The map $n\mapsto w_{n}$
can be taken computable (an explicit formula is given in \cite{Hall1954}).
Then map any marked group $(G,S)$ to the sequence $(p_{i})_{i\in\mathbb{N}}$
given by 
\[
p_{i}=1\iff w_{i}\text{ is a relation in \ensuremath{(G,S)}}.
\]
Because the relations $w_{i}$ are independent, this defines a surjection
$r:\mathcal{G}_{k}\rightarrow\{0,1\}^{\mathbb{N}}$ onto the Cantor
space. Because $n\mapsto w_{n}$ is computable, it is a computable
map. And the computable injection $i:\{0,1\}^{\mathbb{N}}\rightarrow\mathcal{G}_{k}$
is given by Hall's construction. 
\end{proof}
\begin{cor}
\label{cor:BM not Markov on G}For every $k\ge2$, there exists a
Banach-Mazur computable function defined on $\mathcal{G}_{k}$ which
is not Markov computable.
\end{cor}

\begin{proof}
If $f:\{0,1\}_{c}^{\mathbb{N}}\rightarrow\mathbb{N}$ is a Banach-Mazur
computable function which is not Markov computable, as given in \cite{Bauer2004},
and if $r$ is the map given in the previous lemma, then $f\circ r$
is a Banach-Mazur computable function $\mathcal{G}^{+}\rightarrow\mathbb{N}.$
It is not Markov computable, otherwise $f\circ r\circ i$ would be
as well, but this is $f$. 
\end{proof}

\subsection{Non effective separability of $\mathcal{G}^{+}$ }

\begin{thm}
\label{prop:No-re dense seq}No sequence of marked groups can be both
computable and dense in $\mathcal{G}^{+}$. 
\end{thm}

\begin{proof}
This is a simple application of Theorem \ref{thm:Boone-Rogers-reformulated-},
together with Corollary \ref{cor:HoyrupRojas NonEmpty is SD}. Corollary
\ref{cor:HoyrupRojas NonEmpty is SD} states that in a recursive Polish
space, there is an algorithm that stops exactly on semi-decidable
sets that are non-empty. 

The basic clopen sets $\Omega_{r_{i};s_{j}}^{k}$ are obviously semi-decidable
in $\mathcal{G}^{+}$, but a program that recognizes those basic clopen
sets that are non-empty would allow one to recognize wp-coherent sets
of relations and irrelations, contradicting Theorem \ref{thm:Boone-Rogers-reformulated-}. 
\end{proof}
It is interesting to interpret this proof as a variation on McKinsey's
theorem \cite{McKinsey1943} which states that finitely presented
residually finite groups have solvable word problem. Notice that if
$X$ is a set of marked groups which is dense in $\mathcal{G}^{+}$,
then every finitely presented group with solvable word problem is
``residually-$X$'', and a proof similar to McKinsey's would then
contradict the Boone and Rogers Theorem \cite{Boone1966}. 

This proposition directly implies the following:
\begin{cor}
The recursive metric space $(\mathcal{G}^{+},d,\nu_{WP})$ is effectively
complete but not effectively separable, and thus it is not a recursive
Polish space. 
\end{cor}

As we have already seen in the previous section, Mazur's Continuity
Theorem and Ceitin's Effective Continuity Theorem both apply to recursive
Polish spaces. This corollary thus shows that those theorems cannot
be directly applied to the space of marked groups. 

The following result is more general than the previous one, because
it does not necessarily applies to group given by word-problem algorithms.
Note that it is a result about $(\mathcal{G},d)$ up to isometry,
and not up to homeomorphism. 
\begin{thm}
\label{thm:NoRecPres}The metric space $(\mathcal{G},d)$ does not
have a recursive presentation in the sense of Definition \ref{def:RecPresPolish}. 
\end{thm}

\begin{proof}
Recall that a recursive presentation of $(\mathcal{G},d)$ consists
in a sequence $(u_{n})_{n\in\mathbb{N}}$, dense in $\mathcal{G}$,
and for which the distance between the $n$-th and $m$-th terms is
computable. 

Suppose that $\mathcal{G}$ admits a recursive presentation. Then
so does $\mathcal{G}_{2}$. 

Recall that we have defined an embedding $\Phi_{2}:\mathcal{G}_{2}\rightarrow\{0,1\}^{\mathbb{N}}$
by fixing a computable order on the rank two free group. Call a set
$(r_{1},...,r_{m};\,s_{1},...,s_{m'})$ of relations and irrelations
\emph{initial} if the $m+m'$ elements of the free group it contains
are exactly the first $m+m'$ elements of this order. The number $m+m'$
is called the length of this set of relations and irrelations. 

We will prove that the recursive presentation of $\mathcal{G}_{2}$
allows to compute, given an integer $n$, the number of initial coherent
sets of relations and irrelations that contain $n$ relations. 

We first show that this is sufficient to obtain a contradiction. 

There are exactly $2^{n}$ possible initial sets of relations and
irrelations of length $n$. Since the incoherent sets of relations
and irrelations form a c.e. set, if we had access to the number of
initial coherent sets of relations and irrelations of length $n$,
we would be able to compute exactly those sets, by starting with the
$2^{n}$ possible initial sets, and deleting incoherent ones until
the number of coherent sets is attained. 

But being able to compute the \emph{initial} coherent sets of relations
and irrelations in fact also allows one to compute all coherent sets
of relations and irrelations, because a set $(r_{1},...,r_{m};\,s_{1},...,s_{m'})$,
which is not initial, is coherent if and only if there is an initial
and coherent set which contains the elements $(r_{1},...,r_{m})$
as relations, and the elements $(s_{1},...,s_{m'})$ as irrelations.
Choosing $n$ big enough, it suffices to construct all initial sets
of relations and irrelations of length $n$ to determine whether $(r_{1},...,r_{m};\,s_{1},...,s_{m'})$
is coherent. And we have seen that this is impossible. 

Suppose that $(u_{n})_{n\in\mathbb{N}}$ defines a recursive presentation
of $\mathcal{G}_{2}$, we show how to compute the number of initial
coherent sets of relations and irrelations that contain $n$ relations.
Denote by $\lambda(n)$ this number. Again, because the incoherent
sets of relations and irrelations form a c.e. set, $\lambda$ is an
upper semi-computable function: there exists a computable function
$\lambda^{>}$ that, given $n$, produces a computable and decreasing
sequence of integers that converges to $\lambda(n)$. What we show
is that the existence of a recursive presentation of $\mathcal{G}_{2}$
implies that $\lambda$ is also lower semi-computable, meaning that
there exists a computable function $\lambda^{<}$ that, on input $n$,
produces an increasing sequence of integers which converges to $\lambda(n)$. 

Given $i$ and $n$ natural numbers, define $x_{i}^{n}$ to be the
maximal size of a subset of $\{u_{0},u_{1},...,u_{i}\}$ of which
any two elements are at least $2^{-n}$ apart. 

We claim that $(i,n)\mapsto x_{i}^{n}$ is a computable function,
and that, for any $n$, $i\mapsto x_{i}^{n}$ is an increasing function
that converges to $\lambda(n)$. 

Setting $\lambda^{<}(n)=(x_{i}^{n})_{i\in\mathbb{N}}$ will then conclude
the proof. 

As the distance function $d$ takes values only in $\left\{ 0\right\} \cup\left\{ 2^{-n},n\in\mathbb{N}\right\} $,
given the description of the distance $d(u_{i},u_{j})$ as a computable
real, one can always effectively choose one of $d(u_{i},u_{j})<3\times2^{-n-2}$
and $d(u_{i},u_{j})>3\times2^{-n-2}$ which holds, and thus decide
whether or not $u_{i}$ and $u_{j}$ are $2^{-n}$ apart. One can
thus check every subset of $\{u_{0},u_{1},...,u_{i}\}$ to find one
of maximal size, all the elements of which are $2^{-n}$ apart. Thus
$(i,n)\mapsto x_{i}^{n}$ is computable. 

The function $(i,n)\mapsto x_{i}^{n}$ is increasing in $i$ by definition. 

Finally, we show that $x_{i}^{n}$ goes to $\lambda(n)$ as $i$ goes
to infinity. Two points of $\mathcal{G}_{2}$ are at least $2^{-n}$
apart if and only if their binary expansion differ on one of their
first $n$ digits: those groups must be associated to different initial
sets of relations and irrelations of length $n$. Because the sequence
$(u_{n})_{n\in\mathbb{N}}$ is supposed to be dense in $\mathcal{G}_{2}$,
for any coherent initial set of relations and irrelations of length
$n$, there is a point of this sequence which satisfies those relations
and irrelations. 

And thus there must indeed exist a set of $\lambda(n)$ points in
the dense sequence which are pairwise $2^{-n}$ apart, and this number
is clearly maximal. 
\end{proof}

\subsection{\label{subsec:Optimal numbering}Optimality of the numbering type
$\Lambda_{WP}$}

A Polish space is a topological space, it is not attached to a particular
metric. The notion of recursive presentation of a Polish space is
attached to a certain metric. We have shown that $(\mathcal{G},d)$
does not have a recursive presentation, where $d$ is the usual ultrametric
distance of $\mathcal{G}$. 

We have no results that state that the space of marked groups, seen
purely as a topological space, cannot have a computably Polish model. 

In particular the above result on presentations, Theorem \ref{thm:NoRecPres},
could easily be extended to concrete examples of other metrics, but
it is unclear how to prove the same result for an arbitrary metric
that generates the right topology.

What our results show is that if we want to preserve slightly more
than just the topology of $\mathcal{G}$, then undecidability occurs.
We can obtain undecidability results whenever trying to preserve one
of the following three features of $\mathcal{G}$:
\begin{itemize}
\item One of the usual explicit metrics of $\mathcal{G}$, 
\item The explicit embeddings of each $\mathcal{G}_{k}$ into a Cantor space,
\item The numbering of the basis of clopen sets of $\mathcal{G}$ (sets
of the form $\Omega_{R;S}^{k}$, with the obvious numbering: a basic
set $\Omega_{R;S}^{k}$ is given by $k$ and the tuples $R$ and $S$). 
\end{itemize}
Each one of those points provides $\mathcal{G}$ with an extra structure,
in addition to its topology. 

The first point was discussed above, in this section we obtain undecidability
results concerning the second and third points. 

Note that the second and third points are related: the numbering of
the basis of $\mathcal{G}$ is the one induced by the embedding of
$\mathcal{G}$ into a countable union of Cantor spaces. 

The following theorem shows that, if $\hat{d}$ is any distance that
generate the topology $\mathcal{G}$, and $\mu$ any subnumbering
of $\mathcal{G}$, if there is an algorithm that testifies that the
open balls of $(\mathcal{G}_{\mu},\hat{d},\mu)$ are effectively open
with respect to the numbered basis of clopen sets of $\mathcal{G}$,
then $\mu$ provides more information than $\nu_{WP}$. 

\begin{thm}
\label{thm: same eff top >v_wp}Suppose that $\hat{d}$ is any distance
on $\mathcal{G}$ that generates the same topology as $d$, and that
$\mu$ is a subnumbering of $\mathcal{G}$ such that $(\mathcal{G}_{\mu},\hat{d},\mu)$
is a RMS. 

Suppose furthermore that there is an algorithm that takes as input
a $\mu$-name for a $k$-marked group $(G,S)$ and a $c_{\mathbb{R}}$-name
for a radius $r>0$, and produces a basic clopen set $\Omega_{R;T}^{k}$
such that: 
\[
(G,S)\in\Omega_{R;T}^{k};
\]
\[
\Omega_{R;T}^{k}\subseteq B_{\hat{d}}((G,S),r).
\]
Then one has $\mu\succeq\nu_{WP}$. And thus $\mathcal{G}_{\mu}\subseteq\mathcal{G}^{+}$,
and if $\mathcal{G}_{\mu}$ is dense in $\mathcal{G}^{+}$, then \textup{\emph{$(\mathcal{G}_{\mu},\hat{d},\mu)$
does not have }}a $\mu$-computable and dense sequence. 
\end{thm}

A \emph{discriminating family} of a group $G$ is a subset of $G$
which does not contain the identity element of $G$, and which intersects
any non-trivial normal subgroup of $G$. We will use Theorem 3.4 from
\cite{Cornulier2007}, which is an analysis of Kuznetsov's method
for solving the word problem in simple groups \cite{Kuznetsov1958}: 
\begin{thm}
[Cornulier, Guyot, Pitsch, \cite{Cornulier2007}]\label{thm:Cornulier,-Guyot,-Pitsch,}A
group has solvable word problem if and only if it is both recursively
presentable and recursively discriminable. 

And this statement is uniform: there is an effective method that allows,
given a recursive presentation and an algorithm that enumerates a
discriminating family in a marked group $(G,S)$, to find a word problem
algorithm for $(G,S)$. 
\end{thm}

We add the statement about the uniformity of this theorem, but it
is easy to see that the proof given in \cite{Cornulier2007} is uniform. 
\begin{proof}
[Proof of Theorem \ref{thm: same eff top >v_wp}]Consider a $\mu$-computable
marked group $(G,S)$. We show, given a $\mu$-name for $(G,S)$,
how to obtain a word problem algorithm for it. 

Using the algorithm given by the hypotheses of the theorem, consecutively
on each ball $B_{\hat{d}}(G,\frac{1}{n})$, we obtain a computable
sequence $(\Omega_{R_{n};T_{n}})_{n\in\mathbb{N}}$ of basic clopen
subsets, such that for each $n$ we have:
\[
G\in\Omega_{R_{n};T_{n}}\subseteq B_{\hat{d}}(G,\frac{1}{n}).
\]

It follows that $\bigcap\Omega_{R_{n};T_{n}}=\{G\}$, and thus that
the union $\underset{n\in\mathbb{N}}{\bigcup}R_{n}$ defines a computably
enumerable set of relations that defines $G$, and that the set $\underset{n\in\mathbb{N}}{\bigcup}T_{n}$
defines a computably enumerable discriminating family for $G$. 

We can thus apply Theorem \ref{thm:Cornulier,-Guyot,-Pitsch,}, which
indicates that a word problem algorithm for $G$ can be obtained from
this data. 
\end{proof}

\subsection{\label{subsec:Two Applications of Miller }Two applications of a
construction of Miller, failure of Moschovakis' (B) condition for
the space of marked groups }

In this section, we prove two important theorems that use variations
on Miller's example of a finitely presented group that is isolated
from groups with solvable word problem. 
\begin{thm}
\label{thm:Use Miller Wisely -1}No algorithm can stop exactly on
those sets of relations and irrelations which are not wp-coherent. 
\end{thm}

The following theorem is one of our most important results.
\begin{thm}
[Failure of an Effective Axiom of Choice for groups] \label{thm:No-Completion-Theorem}There
is no algorithm that, given a wp-coherent set of relations and irrelations,
produces a word problem algorithm for a marked group that satisfies
those relations and irrelations. 
\end{thm}

The proofs for those results will be similar: they rely on Miller's
constructions of a family of groups $L_{P,Q}$ indexed by two subsets
$P$ and $Q$ of $\mathbb{N}$. For each of those theorems, we will
find some conditions on the sets $P$ and $Q$ that are sufficient
for the groups $L_{P,Q}$ to provide proofs for the Theorems \ref{thm:Use Miller Wisely -1}
and \ref{thm:No-Completion-Theorem}, and then include a lemma to
prove that such sets $P$ and $Q$ do exist. 

We start by detailing Miller's construction. 

\subsubsection{Miller's construction }

We detail the construction of Miller as it was exposed in \cite{Miller1992}.
This construction was first introduced in \cite{Miller81}. 

\paragraph*{Step 1. }

Given two subsets $P$ and $Q$ of $\mathbb{N}$, we consider the
group $L_{P,Q}^{1}$ given by the following presentation:
\[
\langle e_{0},e_{1},e_{2},...\vert e_{0}=e_{i},\,i\in P,\,e_{1}=e_{j},\,j\in Q\rangle
\]
For simplicity, we shall always assume that $P$ contains $0$ and
$Q$ contains $1$. 

Notice that $L_{P,Q}^{1}$ is recursively presented with respect to
the family $(e_{i})_{i\in\mathbb{N}}$ if and only if $P$ and $Q$
are c.e. sets, and that $L_{P,Q}^{1}$ has solvable word problem with
respect to the family $(e_{i})_{i\in\mathbb{N}}$ if and only if $P$
and $Q$ are computable sets. 

In what follows, the sets $P$ and $Q$ will always be computably
enumerable, and thus $L_{P,Q}^{1}$ is recursively presented. 

\paragraph*{Step 2. }

Embed the recursively presented group $L_{P,Q}^{1}$ in a finitely
presented group $L_{P,Q}^{2}$ using some strengthening of Higman's
Embedding Theorem. For our purpose, we need to know that: 
\begin{itemize}
\item A finite presentation of $L_{P,Q}^{2}$ can be obtained from the recursive
presentation of $L_{P,Q}^{1}$;
\item If the group $L_{P,Q}^{1}$ has solvable word problem with respect
to the family $(e_{i})_{i\in\mathbb{N}}$, then the group $L_{P,Q}^{2}$
also has solvable word problem;
\item The embedding of $L_{P,Q}^{1}$ into $L_{P,Q}^{2}$ is effective,
i.e. there exists a computable function that maps a natural number
$n$ to a way of expressing the element $e_{n}$ as a product of the
generators of $L_{P,Q}^{2}$. 
\end{itemize}
Clapham's version of Higman's Embedding Theorem \cite{Clapham1967}
satisfies the required conditions for this step of the construction.
Clapham's Theorem is quoted precisely in Section \ref{subsec:Higman-Clapham-Valiev-Theorem-fo}. 

\paragraph*{Step 3. }

Embed the group $L_{P,Q}^{2}$ into a finitely presented group $L_{P,Q}^{3}$
with the following property: in any non-trivial quotient of $L_{P,Q}^{3}$,
the image of the element $e_{0}e_{1}^{-1}$ is a non-identity element. 

This is done as follows. 

Consider a presentation $\langle x_{1},...,x_{k}\vert r_{1},...r_{t}\rangle$
for $L_{P,Q}^{2}$, denote $w$ a word on $\{x_{1},...,\,x_{k},\,x_{1}^{-1},...,\,x_{k}^{-1}\}$
that defines the element $e_{0}e_{1}^{-1}$ in $L_{P,Q}^{2}$. The
group $L_{P,Q}^{3}$ is defined by adding to $L_{P,Q}^{2}$, in addition
to the generators $x_{1},...,x_{k}$ that are still subject to the
relations $r_{1},...r_{t}$, three new generators $a$, $b$ and $c$,
subject to the following relations:
\begin{enumerate}
\item $a^{-1}ba=c^{-1}b^{-1}cbc$
\item $a^{-2}b^{-1}aba^{2}=c^{-2}b^{-1}cbc^{2}$
\item $a^{-3}\left[w,b\right]a^{3}=c^{-3}bc^{3}$
\item $a^{-(3+i)}x_{i}ba^{(3+i)}=c^{-(3+i)}bc^{(3+i)},\,i=1..k$
\end{enumerate}
To use Miller's construction, we need to check the following points:
\begin{itemize}
\item If $w\neq1$ in $L_{P,Q}^{2}$, then $L_{P,Q}^{2}$ is embedded in
$L_{P,Q}^{3}$ via the map $x_{i}\mapsto x_{i}$. 
\item The presentation of $L_{P,Q}^{3}$ can be computed from the presentation
of $L_{P,Q}^{2}$ together with the word $w$. 
\item If $L_{P,Q}^{2}$ has solvable word problem, then so does the group
$L_{P,Q}^{3}$. 
\item The element $e_{0}e_{1}^{-1}$ has a non-trivial image in any non-trivial
quotient of $L_{P,Q}^{3}$. 
\end{itemize}
The second point is obvious. The last point is easily proven: remark
that the third written relation, together with $w=1$, implies the
relation $b=1$. This in turn implies that $c=1$ thanks to the first
relation, that $a=1$ thanks to the second relation, and then that
all $x_{i}$ also define the identity element because of the relations
of (4). 

The first and third points are proven using the fact that the group
$L_{P,Q}^{3}$ can be expressed as an amalgamated product.

Consider the free product $L_{P,Q}^{2}*\mathbb{F}_{a,b}$ of $L_{P,Q}^{2}$
with a free group generated by $a$ and $b$, and the free group $\mathbb{F}_{b,c}$
generated by $b$ and $c$. Then, provided that $w\ne1$ in $L_{P,Q}^{2}$,
the subgroup of $L_{P,Q}^{2}*\mathbb{F}_{a,b}$ generated by $b$
and the elements that appear to the left hand side in the equations
$(1)-(4)$ is a free group on $4+k$ generators, which we denote $A$,
and so is the subgroup $B$ of $\mathbb{F}_{b,c}$ generated by $b$
and the elements that appear to the right hand side in the equations
$(1)-(4)$. (These families are easily seen to be Nielsen reduced
\cite{Lyndon1977}: any product between distinct elements cancels
less than half of both words.)

Thus the given presentation of $L_{P,Q}^{3}$ shows that it is defined
as an amalgamated product of the form:
\[
\begin{array}{ccc}
(L_{P,Q}^{2}*\mathbb{F}_{a,b}) & * & \mathbb{F}_{b,c}\\
 & A=B
\end{array}
\]

This proves both the fact that $L_{P,Q}^{2}$ embeds in $L_{P,Q}^{3}$,
and that the word problem is solvable in $L_{P,Q}^{3}$ as soon as
it is in $L_{P,Q}^{2}$. Indeed, to solve the word problem in an amalgamated
product such as $L_{P,Q}^{3}$, it suffices to be able to solve the
membership problem for $A$ in $L_{P,Q}^{2}*\mathbb{F}_{a,b}$ and
for $B$ in $\mathbb{F}_{b,c}$. This can be done as soon as the word
problem is solvable in $L_{P,Q}^{2}$. Indeed, the fact that the generators
of the group $A$ give a Nielsen reduced family means that the word
length in $L_{P,Q}^{2}*\mathbb{F}_{a,b}$ of a product of $n$ generators
of $A$ is at least $n$. This implies immediately that the membership
problem to $A$ inside $L_{P,Q}^{2}*\mathbb{F}_{a,b}$ is solvable. 

Finally, we designate by $\Pi_{P,Q}$ the finite set of relations
and irrelations that is composed of the relations of $L_{P,Q}^{3}$,
and of a unique irrelation $w\ne1$, where $w$ is the word that defines
the element $e_{0}e_{1}^{-1}$ in $L_{P,Q}^{3}$. 

Note that the finite set $\Pi_{P,Q}$ can be effectively produced
from the codes for the c.e. sets $P$ and $Q$. 

This ends Miller's construction. 

\subsubsection{First application: Miller's Theorem}

We include here a proof of Miller's Theorem. 

A pair of disjoint subsets $P$ and $Q$ of $\mathbb{N}$ are said
to be \emph{computably inseparable }if there cannot exist a computable
set $H$ such that $P\subseteq H$ and $Q\subseteq H^{c}$, where
$H^{c}$ denotes the complement of $H$ in $\mathbb{N}$. 

We will need the following well known result (originally due to Kleene): 
\begin{lem}
\label{lem:Inseparable sets}There exists a pair $(P,Q)$ of disjoints
subsets of $\mathbb{N}$ that are computably enumerable and computably
inseparable. 
\end{lem}

\begin{proof}
Consider a standard enumeration $(\varphi_{0},\varphi_{1},\varphi_{2},...)$
of all computable functions. Consider the set $P=\{n\in\mathbb{N},\,\varphi_{n}(n)=0\}$
and the set $Q=\{n\in\mathbb{N},\,\varphi_{n}(n)=1\}$. Those sets
are obviously computably enumerable. Suppose now that some computable
set $H$ contains $P$ but does not intersect $H$. Consider an index
$n_{0}$ such that $\varphi_{n_{0}}$ is a total function that computes
the characteristic function of $H$. 

If $\varphi_{n_{0}}(n_{0})=0$, $n_{0}$ does not belong to $H$,
but it belongs to $P$, this is not possible. But if $\varphi_{n_{0}}(n_{0})=1$,
then $n_{0}$ belongs to $Q$ and to $H$, which is also impossible
because $H$ does not meet $Q$. 

This is a contradiction, and thus the sets $P$ and $Q$ are indeed
computably inseparable. 
\end{proof}
\begin{thm}
[Miller, \cite{Miller81}] \label{thm:Miller,}Suppose that $P$
and $Q$ are disjoints subsets of $\mathbb{N}$ that are computably
enumerable and computably inseparable. Then the set $\Pi_{P,Q}$ is
coherent, but not wp-coherent. 
\end{thm}

\begin{proof}
Suppose that a group $K$ satisfies the relations and irrelations
of $\Pi_{P,Q}$, and that is has solvable word problem. 

Using the word problem algorithm for $K$, given an integer $i$ in
$\mathbb{N}$, we can solve the questions ``is $e_{0}=e_{i}$ in
$K$'', since, by the properties of Miller's construction, an expression
of the element $e_{i}$ in terms of the generators of $L_{P,Q}^{3}$,
and thus of $K$, can be effectively found from $i$. 

The set $\{i\in\mathbb{N},\,e_{i}=e_{0}\}$ is thus a computable set
that contains $P$. And it is disjoint from $Q$, because we have
assumed that $e_{0}\ne e_{1}$ in $K$. 

This contradicts the fact that $P$ and $Q$ are computably inseparable.
\end{proof}

\subsubsection{Proof of Theorem \ref{thm:Use Miller Wisely -1}}

We first prove Theorem \ref{thm:Use Miller Wisely -1}:
\begin{thm*}
\label{thm:Use Miller Wisely -1-1}No algorithm can stop exactly on
those sets of relations and irrelations which are not wp-coherent. 
\end{thm*}
\begin{proof}
Given c.e. disjoint sets, we apply Miller's construction to obtain
the set $\Pi_{P,Q}$ of relations and irrelations. 

By Theorem \ref{thm:Miller,}, if the sets $P$ and $Q$ are computably
enumerable, computably inseparable sets, then $\Pi_{P,Q}$ is not
wp-coherent. 

On the contrary, if $P$ and $Q$ are both computable sets, we have
noted that $L_{P,Q}^{3}$ itself has solvable word problem, and thus
$\Pi_{P,Q}$ is wp-coherent. 

Because the set $\Pi_{P,Q}$ can be constructed from the codes for
$P$ and $Q$, an algorithm that stops exactly on those sets of relations
and irrelations which are not wp-coherent would produce, through Miller's
construction, an algorithm that, given a pair of c.e. sets $P$ and
$Q$ that are either computably inseparable or recursive, would stop
if and only if those sets are computably inseparable. We prove in
the next lemma, Lemma \ref{lem:rec inseparable sets}, that such an
algorithm cannot exit, this ends the proof of our theorem. 
\end{proof}
\begin{lem}
\label{lem:rec inseparable sets}There is no algorithm that, given
the code for two computably enumerable and disjoint subsets of $\mathbb{N}$,
that are either computable or computably inseparable, stops only if
they are computably inseparable. 
\end{lem}

\begin{proof}
Fix two computably enumerable and computably inseparable subsets $P$
and $Q$ of $\mathbb{N}$, that exist by Lemma \ref{lem:Inseparable sets}. 

Consider an effective enumeration $M_{0}$, $M_{1}$, $M_{2}$...
of all Turing machines. For each natural number $n$, define a pair
of computably enumerable sets $P_{n}$ and $Q_{n}$ defined as follows: 

To enumerate $P_{n}$, start a run of $M_{n}$. 

While this run lasts, an enumeration of $P$ gives the first elements
of $P_{n}$. If $M_{n}$ halts after $k$ computation steps, stop
the enumeration of $P$. 

Thus if $M_{n}$ halts, the set $P_{n}$ is a finite set. On the contrary,
if $M_{n}$ does not stop, $P_{n}$ is identical to $P$. 

The set $Q_{n}$ is defined similarly, replacing $P$ by $Q$ in its
definition. 

One then easily sees that the sets $P_{n}$ and $Q_{n}$ are uniformly
computably enumerable, and that $P_{n}$ and $Q_{n}$ are computably
inseparable if and only if $M_{n}$ does not halt.

Since no algorithm can stop exactly on the indices of non-halting
Turing machines, the lemma is proved. 
\end{proof}

\subsubsection{Proof of Theorem \ref{thm:No-Completion-Theorem}}

We now prove Theorem \ref{thm:No-Completion-Theorem}:
\begin{thm*}
[Failure of an effective Axiom of Choice for groups] There is no
algorithm that, given a wp-coherent set of relations and irrelations,
produces a word problem algorithm for a marked group that satisfies
those relations and irrelations. 
\end{thm*}
\begin{proof}
We will build in Lemma \ref{lem:Lemma for NCT} a pair of sequences
$(P_{n})_{n\in\mathbb{N}}$ and $(Q_{n})_{n\in\mathbb{N}}$ such that: 
\begin{itemize}
\item The sequences $(P_{n})_{n\in\mathbb{N}}$ and $(Q_{n})_{n\in\mathbb{N}}$
consist only of disjoint computable sets;
\item The sequences $(P_{n})_{n\in\mathbb{N}}$ and $(Q_{n})_{n\in\mathbb{N}}$
are uniformly c.e., but not uniformly computable;
\item For any sequence $(H_{n})_{n\in\mathbb{N}}$ of uniformly computable
sets, there must be some index $n_{0}$ such that either $H_{n_{0}}$
does not contain $P_{n_{0}}$, or $H_{n_{0}}^{c}$ does not contain
$Q_{n_{0}}$.
\end{itemize}
We apply Miller's construction to this sequence to obtain a computable
sequence $(\Pi_{P_{n},Q_{n}})_{n\in\mathbb{N}}$ of finite sets of
relations and irrelations. 

Suppose by contradiction that there is an algorithm $\mathcal{A}$
as in the theorem. For each natural number $n$, the set $\Pi_{P_{n},Q_{n}}$
is wp-coherent, because $P_{n}$ and $Q_{n}$ are computable. Thus
the algorithm $\mathcal{A}$ can be applied to $\Pi_{P_{n},Q_{n}}$,
to produce the word problem algorithm for a group that satisfies the
relations and irrelations of $\Pi_{P_{n},Q_{n}}$. Denote $G_{n}$
the group defined by this algorithm. 

For each $n$, the set $H_{n}=\{i\in\mathbb{N},e_{i}=e_{0}\text{ in }G_{n}\}$
is then a computable set, and this in fact holds uniformly in $n$.

Finally, one has the inclusions $P_{n}\subseteq H_{n}$ and $Q_{n}\subseteq H_{n}^{c}$.
This contradicts the properties of the sequences $(P_{n})_{n\in\mathbb{N}}$
and $(Q_{n})_{n\in\mathbb{N}}$. 
\end{proof}
\begin{lem}
\label{lem:Lemma for NCT}There exists a pair of sequences $(P_{n})_{n\in\mathbb{N}}$
and $(Q_{n})_{n\in\mathbb{N}}$ such that: 
\begin{itemize}
\item The sequences $(P_{n})_{n\in\mathbb{N}}$ and $(Q_{n})_{n\in\mathbb{N}}$
consist only of disjoint computable sets;
\item The sequences $(P_{n})_{n\in\mathbb{N}}$ and $(Q_{n})_{n\in\mathbb{N}}$
are uniformly c.e., but not uniformly computable;
\item For any sequence $(H_{n})_{n\in\mathbb{N}}$ of uniformly computable
sets, there must be a some index $n_{0}$ such that either $H_{n_{0}}$
does not contain $P_{n_{0}}$, or $H_{n_{0}}^{c}$ does not contain
$Q_{n_{0}}$.
\end{itemize}
\end{lem}

\begin{proof}
Fix a pair $(P,Q)$ of computably enumerable but computably inseparable
sets\footnote{Previous versions of this article contained a more complicated example.
I thank Laurent Bienvenu for showing me that this simple example works.}. 

Let $P_{n}=P\cap\{n\}$ and $Q_{n}=Q\cap\{n\}$. By construction these
sets are disjoint, computable (these are singletons), and the sequences
$(P_{n})_{n\in\mathbb{N}}$ and $(Q_{n})_{n\in\mathbb{N}}$ are uniformly
c.e.. 

Suppose $(H_{n})_{n\in\mathbb{N}}$ is a sequence of uniformly computable
sets that separate $P_{n}$ from $Q_{n}$. 

Consider $H$ given by: $n\in H\iff n\in H_{n}.$ Then $P\subseteq H\subseteq Q^{c}$,
and $H$ is a computable set. This contradicts the fact that $P$
and $Q$ are inseparable. 
\end{proof}
Theorem \ref{thm:No-Completion-Theorem} allows us to prove that the
space of marked groups does not satisfy Moschovakis' condition (B),
and thus that Theorem \ref{thm:Moschovakis} cannot be applied to
the space of marked groups. 
\begin{cor}
\label{cor:MoschoFails}The triple $(\mathcal{G}^{+},d,\nu_{WP})$
does not satisfy Moschovakis' condition (B). 
\end{cor}

\begin{proof}
A RMS $(X,d,\nu)$ satisfies Moschovakis' condition (B) if there exists
an algorithm $\mathcal{A}$ that takes as input the description of
a $\nu$-semi-decidable set $Y$ and the description of an open ball
$B(x,r)$ in $X$, such that those set intersect, and produces a point
that belongs to their intersection. 

In $\mathcal{G}_{k}$, apply such an algorithm to a basic clopen subset
$\Omega_{r_{i};s_{j}}^{k}$ and to an open ball that contains all
of $\mathcal{G}_{k}$ -any open ball of radius $r\ge1$. This yields
a program that takes as input a set of relations and irrelations that
is wp-coherent, and produces the $\nu_{WP}$-name of a point that
belongs to it. The existence of such an algorithm was proven impossible
in Theorem \ref{thm:No-Completion-Theorem}. 
\end{proof}

\section{\label{sec:First-results-for}Correspondence between Borel and effective
Borel hierarchies}

We now apply the results of the previous sections to decision problems
for groups described by word problem algorithms.

\subsection{\label{subsec:List-of-properties}List of properties for which the
correspondence holds}

Recall from the introduction that we have three distinct effective
versions of the Borel hierarchy: the one coming from Type 2 computability,
the one coming from Markov computability, and the one coming from
Banach-Mazur computability. 

We expect that most natural group properties will be given the same
classification in these four hierarchies (one classical and three
effective). We say that a property \emph{satisfies the correspondence
between the different hierarchies }if it is lies at the same position
in all four hierarchies.\emph{ }

We will now proceed to list group properties that satisfy the correspondence.
These properties are organized into a table, which can be found in
Section \ref{subsec:The Table}. 

Undecidability results which are required to establish correspondence
results are all obtained by one or two applications of the following
version of Markov's Lemma:
\begin{prop}
\label{prop:Eff not open =00003D> not SD}Suppose that a subset $A$
of $\mathcal{G}^{+}$ is effectively not open, i.e. that there is
a computable sequence of marked groups that do not belong to $A$,
which converges to a marked group in $A$. Then $A$ cannot be $\Lambda_{WP}$-semi-decidable. 
\end{prop}

In what follows, we will often use this proposition together with
the following proposition, which is a reformulation of Proposition
\ref{prop:dense seq effective polish}. 
\begin{prop}
\label{prop:Adherent to c.e. set}If a marked group $(G,S)$ with
solvable word problem is adherent to a c.e. set $\mathcal{C}$ of
marked groups, then there is a $\Lambda_{WP}$-computable sequence
of marked groups in $\mathcal{C}$ that effectively converges to $(G,S)$. 
\end{prop}

This proposition, while very easy, is often useful. It can be used
with $\mathcal{C}$ being: the class of finite groups, of hyperbolic
groups, of all the markings of a given abstract group $G$, etc. 

Those affirmations that appear in Section \ref{subsec:The Table}
which are not accompanied by either a reference, or for which a proof
does not appear in the following paragraphs, are left to the reader. 

\subsubsection{Residually finite groups }

A group is residually finite if any non-trivial element of it has
a non-trivial image in a finite quotient. Note that the space of marked
groups gives a topological characterization of residually finite groups
\cite{Champetier2005}: a group is residually finite if and only if
some (equiv. any) of its markings is adherent to the set of its marked
finite quotients. It is known that the set of residually finite groups
is not closed, because the closure of the set of finite groups (which
is known as the set LEF groups) strictly contains the set of residually
finite groups. The semi-direct product $\mathbb{Z}\ltimes\mathfrak{S}_{\infty}$
is the limit of the sequence of finite groups $\mathbb{Z}/n\mathbb{Z}\ltimes\mathfrak{S}_{n}$,
as $n$ goes to infinity; it is not residually finite since it contains
an infinite simple group. Note that the described sequence effectively
converges. (This example comes from \cite{VershikGordon1998}.)

It is also well known that the set of residually finite groups is
not open: non-abelian free groups are limits of non-residually finite
groups. For instance, it follows from well known results on the Burnside
problem that free Burnside groups of sufficiently large exponent are
infinite groups that are not residually finite, and that, if $S$
denotes a basis of a non-abelian free group, the sequence $(\mathbb{F}/\mathbb{F}^{n},S)$
is $\Lambda_{WP}$-computable and converges to $(\mathbb{F},S)$ as
$n$ goes to infinity. 

\subsubsection{Amenable groups}

Amenable groups do not form a closed set, since free groups are limits
of finite groups. They do not form an open set either. An interesting
example that proves this comes from \cite{Bartholdi2015}: there exists
a sequence of markings of $\mathbb{F}_{2}\wr\mathbb{Z}$ that converges
to a marking of $\mathbb{Z}^{2}\wr\mathbb{Z}$. (See Example 7.4 in
\cite{Bartholdi2015}.) This sequence is $\Lambda_{WP}$-computable. 

\subsubsection{Having sub-exponential growth}

The set of finitely generated groups with sub-exponential growth is
neither closed nor open. This was proved by Grigorchuk in \cite{Grigorchuk1985}:
there is constructed a family of groups $G_{\omega}$, $\omega\in\{0,1,2\}^{\mathbb{N}}$,
for which it is explained when $G_{\omega}$ has intermediate or exponential
growth, and for which it is proved that the convergence in the space
of marked groups of a sequence of groups $(G_{\omega_{n}})_{n\in\mathbb{N}}$
coincides with the convergence in $\{0,1,2\}^{\mathbb{N}}$ (for the
product topology) of the sequence $(\omega_{n})_{n\in\mathbb{N}}$.
The sequence $(G_{\omega_{n}})_{n\in\mathbb{N}}$ is $\Lambda_{WP}$-computable
when the sequence $(\omega_{n})_{n\in\mathbb{N}}$ is computable for
the usual numbering of the Cantor space. 

\subsubsection{Orderable groups, locally indicable groups}

A finitely generated group $G$ is \emph{left-orderable} if there
exists a total order $\le$ on $G$ which is compatible with left
multiplication:
\[
\forall x,a,b\in G,\,a\le b\implies xa\le xb.
\]

The fact that the set of left-orderable groups is closed was noticed
by Champetier and Guirardel in \cite{Champetier2005}. The following
well known characterization of left-orderable groups was already used
in \cite{Bilanovic_2019} and \cite{Goldbring2023} in order to study
the complexity of the property ``being left-orderable'': 
\begin{prop}
\label{prop:Left Orderable characterisation}A group $G$ is left-orderable
if and only if for any finite set $\{a_{1},a_{2},...,a_{n}\}$ of
non-identity elements of $G$, there are signs $(\epsilon_{1},...,\epsilon_{n})\in\{-1,1\}^{n}$,
such that the sub-semi-group generated by $\{a_{1}^{\epsilon_{1}},\,a_{2}^{\epsilon_{2}},...,a_{n}^{\epsilon_{n}}\}$
does not contain the identity of $G$. 
\end{prop}

It is straightforward to notice that this characterization provides
a way of recognizing word problem algorithms for non left-orderable
groups, and that it shows that the set of left-orderable groups is
closed. 
\begin{prop}
The set of left-orderable groups is $\rho_{WP}$-co-semi-decidable
and closed in $\mathcal{G}$.
\end{prop}

We note that characterizations similar to the one given in Proposition
\ref{prop:Left Orderable characterisation} are available for the
properties of being bi-orderable and locally indicable, see Fact 3.3.8
of \cite{Goldbring2023}. These yield: 
\begin{prop}
The sets of bi-orderable groups and of locally indicable groups are
$\rho_{WP}$-co-semi-decidable and closed in $\mathcal{G}$.
\end{prop}

\subsubsection{Diffuse and unique product groups}

A group $G$ has the \emph{unique product property} if for every finite
subsets $A$ and $B$ of $G$, there is some $a$ in $A$ and $b$
in $B$ whose product $ab$ can be written in a unique way as a product
of an element of $A$ by an element of $B$. This property is relevant
in that it implies a positive solution to the unit conjecture for
group rings. In \cite{Bowditch2000} Bowditch introduced the following
variation on the unique product property. A group is called \emph{diffuse}
if for every finite subset $A\subseteq G$ there is some $a\in A$
such that for every $g\in G$ either $ga\notin A$ or $g^{-1}a\notin A$. 

We immediately get: 
\begin{prop}
The sets of diffuse groups and of groups with the unique product property
are $\rho_{WP}$-co-semi-decidable and closed in $\mathcal{G}$.
\end{prop}

\subsubsection{\label{subsec:Virtually-cyclic-groups;}Virtually cyclic groups;
virtually nilpotent groups; polycyclic groups }

For the three properties of being virtually cyclic, virtually nilpotent
or polycyclic, we use the following lemma. In what follows, denote
by $\rho_{pres}$ the representation of $\mathcal{G}$ associated
to presentations. The numbering it induces is that associated to recursive
presentations. See Section \ref{subsec:Presentation-and-co-presentation RPZ}. 
\begin{lem}
\label{lem:SemiDecidable and quotients}Suppose that $P$ is a $\rho_{WP}$-semi-decidable
(resp. $\rho_{pres}$-semi-decidable) subset of $\mathcal{G}$, and
that $Q$ is a $\rho_{pres}$-semi-decidable subset of $\mathcal{G}$. 

Then the set of groups that have a finitely generated normal subgroup
in $P$, and such that the quotient by this normal subgroup is in
$Q$, is $\rho_{WP}$-semi-decidable (resp. $\rho_{pres}$-semi-decidable). 
\end{lem}

Note that in this statement, the normal subgroup should be finitely
generated as a group, and not only finitely generated as a normal
subgroup. 
\begin{proof}
Given a word problem algorithm for a group $G$ generated by a family
$S$, we proceed as follows. 

Enumerate all finite subsets of $G$. 

There is an effective procedure that recognizes those finite subsets
of $G$ that generate a normal subgroup. Indeed, consider a finite
set $A$ in $G$, and the subgroup $H$ it generates. The subgroup
$H$ is normal in $G$ if and only if for each $a$ in $A$ and each
$s$ in $S$, the elements $s^{-1}as$ and $sas^{-1}$ both belong
to $H$. An exhaustive search for ways of expressing $s^{-1}as$ and
$sas^{-1}$ as products of elements of $A$ will terminate if indeed
those elements belong to $H$. 

For each finite subset $A$ which generates a normal subgroup $H$
of $G$, we can obtain a word problem (resp. a presentation) for $H$
thanks to the word problem for $G$ (resp. thanks the presentation
of $G$), and a presentation for the quotient $G/H$. Indeed, an enumeration
of the relations of $G$ together with an enumeration of the elements
of $H$ yields a presentation of $G/H$. 

The hypotheses of the lemma then allow us to recognize when the group
$H$ is in $P$ and the quotient $G/H$ is in $Q$. 
\end{proof}
This lemma can be used directly to show that the properties of being
virtually nilpotent or virtually cyclic are $\rho_{WP}$-semi-decidable.
Note that virtually cyclic and virtually nilpotent groups are all
finitely presented and have solvable word problem, since the property
``being finitely presented and having solvable word problem'' is
stable under taking finite extensions -indeed, it is even invariant
under quasi-isometry (see for instance \cite{Sapir2011}). 
\begin{cor}
The set of virtually cyclic groups is open in $\mathcal{G}$ and $\rho_{WP}$-semi-decidable. 
\end{cor}

\begin{proof}
Apply Lemma \ref{lem:SemiDecidable and quotients} with $P$ being
the set of cyclic groups and $Q$ the set of finite groups, to prove
that the set of virtually cyclic groups $\rho_{WP}$-semi-decidable. 

Analyzing the way the algorithm that stops on word-problem algorithms
for virtually cyclic groups thus obtained works, one sees that the
proof it produces of the fact that a group is virtually cyclic consists
in finitely many relations. Those relations define an open set that
contains only virtually cyclic groups.
\end{proof}
\begin{cor}
The set of virtually nilpotent groups is open in $\mathcal{G}$ and
$\rho_{WP}$-semi-decidable. 
\end{cor}

\begin{proof}
Apply Lemma \ref{lem:SemiDecidable and quotients} with $P$ being
the set of nilpotent groups and $Q$ the set of finite groups. 
\end{proof}
Note that we have chosen these two properties because they both admit
equivalent definitions in terms of asymptotic geometry: by a well-known
theorem of Gromov \cite{Gromov1981}, a group is virtually nilpotent
if and only if it has polynomial growth, and a group is virtually
cyclic if and only if it has zero or two \emph{ends}, in the sense
of Stallings \cite{Hopf1943,Freudenthal1944}. It is thus remarkable
that both of these properties can be recognized thanks to only finitely
many relations. 
\begin{cor}
The set of polycyclic groups is open in $\mathcal{G}$ and $\rho_{WP}$-semi-decidable. 
\end{cor}

\begin{proof}
(Sketch) Iterate Lemma \ref{lem:SemiDecidable and quotients} with
$Q$ being the set of cyclic groups, and $P$ being the set polycyclic
groups with a subnormal series of length $n$, to obtain the result.
\end{proof}

\subsubsection{Groups with infinite conjugacy classes }

A group $G$ has \emph{infinite conjugacy classes} (ICC) if for each
non identity element $g$ of $G$ the conjugacy class $\{xgx^{-1},\,x\in G\}$
of $g$ is infinite. 
\begin{prop}
The set of ICC groups is $\rho_{WP}$-co-semi-decidable, and it is
closed in $\mathcal{G}$. 
\end{prop}

\begin{proof}
Given a word problem algorithm for a group $G$ over a generating
family $S$ and an element $g$ of $G$, it is possible to prove that
the conjugacy class of $g$ is finite: define a sequence of sets $(A_{n})_{n\in\mathbb{N}}$
by 
\[
A_{0}=\{g\},
\]
\[
A_{n+1}=\{s^{-1}xs;\,s\in S\cup S^{-1},\,x\in A_{n}\}.
\]
 The conjugacy class of $g$ is finite if and only if there exists
an integer $n$ such that $A_{n}=A_{n+1}$. A blind search for such
an integer will terminate if it exists. This shows that the set of
non-ICC groups is $\rho_{WP}$-semi-decidable.
\end{proof}

\subsubsection{Hyperbolic groups. }

Let $\delta$ be a positive real number. A marked group $G$ is $\delta$-hyperbolic
if the triangles (defined thanks to the word metric in $G$) are $\delta$-thin,
that is to say if for any three elements $g_{1}$, $g_{2}$ and $g_{3}$
of $G$, any geodesic that joins $g_{1}$ to $g_{2}$ stays in a $\delta$-neighborhood
of any pair of geodesics that join respectively $g_{2}$ and $g_{3}$
and $g_{1}$ and $g_{3}$. 
\begin{prop}
Being $\delta$-hyperbolic is $\rho_{WP}$-co-semi-decidable and thus
closed. 
\end{prop}

\begin{proof}
A marked group $(G,S)$ is not $\delta$-hyperbolic if it admits a
triangle that is not $\delta$-thin. The fact that this triangle is
not $\delta$-thin can be seen in a sufficiently large ball of the
Cayley graph of $(G,S)$, and thus any group whose Cayley graph corresponds
to that $(G,S)$ on this ball is also not $\delta$-hyperbolic. It
is easy to see that this condition can be effectively checked.
\end{proof}
Remark that being $\delta$-hyperbolic is a marked group property,
but not a group property, as can be seen from the fact that there
exists a sequence of markings of $\mathbb{Z}$ that converges to a
marking of $\mathbb{Z}^{2}$, which is not hyperbolic. A group is
Gromov hyperbolic if any of its marking is $\delta$-hyperbolic, for
some $\delta$ that can depend on the marking. The set of Gromov hyperbolic
groups is neither open nor closed in $\mathcal{G}$, but the previous
proposition implies that it is a union of closed sets. 
\begin{cor}
The set of hyperbolic groups is a $F_{\sigma}$ subset of $\mathcal{G}$,
and it is effectively not closed and effectively not open. 
\end{cor}

\subsection{\label{subsec:The Table}Table of results}

The following table gathers our examples. Remark that for each subset
of $\mathcal{G}$ that appears in this table, eight properties of
this set are expressed: whether or not it is open, whether or not
it is closed, whether or not it is $\rho_{WP}$-semi-decidable, whether
or not it is $\rho_{WP}$-co-semi-decidable (Type 2 characterizations),
and the corresponding results in terms of Type 1 and of Banach-Mazur
computability. 

$\,$

 \centerline{%
\begin{tabular}{|c|c|}
\hline 
\textbf{Clopen/decidable properties} & \textbf{Open/semi-decidable properties}\tabularnewline
\hline 
Being abelian; & Being nilpotent;\tabularnewline
\hline 
Being isomorphic to a given finite group; & Kazhdan's Property (T) (\cite[Theorem 6.7]{Shalom2000}, \cite{Ozawa2014});\tabularnewline
\hline 
Having cardinality at most $n$, $n\in\mathbb{N}^{*}$;  & Having a non-trivial center;\tabularnewline
\hline 
Being nilpotent of derived length $k>0$; & Being perfect;\tabularnewline
\hline 
Being a certain marked isolated group. & Having torsion;\tabularnewline
\hline 
\multicolumn{1}{c|}{} & Having rank at most $k$, $k\in\mathbb{N}^{*}$;\tabularnewline
\cline{2-2} 
\multicolumn{1}{c|}{} & Being virtually cyclic;\tabularnewline
\cline{2-2} 
\multicolumn{1}{c|}{} & Having polynomial growths;\tabularnewline
\cline{2-2} 
\multicolumn{1}{c|}{} & Being polycyclic;\tabularnewline
\cline{2-2} 
\multicolumn{1}{c|}{} & Counterexamples to one of Kaplansky's conjectures \cite{Gardam2021}.\tabularnewline
\cline{2-2} 
\multicolumn{1}{c}{} & \multicolumn{1}{c}{}\tabularnewline
\hline 
\textbf{Closed/co-semi-decidable properties} & \textbf{Neither closed nor open properties}\tabularnewline
\hline 
Being infinite; & Being solvable;\tabularnewline
\hline 
Being $k$-solvable, for a fixed $k>1$; & Being amenable;\tabularnewline
\hline 
Having finite exponent $k$, for a fixed $k>>1$; & Being simple;\tabularnewline
\hline 
Being a limit group (see Section \ref{subsec:elementary theory}); & Having sub-exponential growth;\tabularnewline
\hline 
Being left-orderable, bi-orderable or locally indicable; & Being finitely presented;\tabularnewline
\hline 
Having the unique product property; & Being hyperbolic;\tabularnewline
\hline 
Being diffuse; & Being residually finite.\tabularnewline
\hline 
Being $\delta$-hyperbolic, $\delta>0$;  & \multicolumn{1}{c}{}\tabularnewline
\cline{1-1} 
Having Infinite Conjugacy Classes (ICC).  & \multicolumn{1}{c}{}\tabularnewline
\cline{1-1} 
\end{tabular}}

\medskip

\section{\label{sec:Two-candidates-for failure}Differences between Borel
and effective Borel hierarchies}

In this section, we prove that the set of LEF groups is closed but
not effectively closed. Whether it can be $\Lambda_{WP}$-co-semi-decidable
is left open (thus the Type 2 undecidability result has yet to be
upgraded to a Type 1 or a Banach-Mazur result). 

We also give several other candidates of properties that could break
the correspondence. 

\subsection{Isolated Groups}

It does not seem possible to prove that a word problem algorithm belongs
to an isolated group (even with a partial algorithm). From this, we
conjecture that the set of isolated groups, while open, is not $\Lambda_{WP}$-semi-decidable.
The same conjecture goes for finitely presented simple groups, which
form a subset of the set of isolated groups. 
\begin{conjecture}
The sets of isolated group and of finitely presented simple groups
are not $\Lambda_{WP}$-semi-decidable. 
\end{conjecture}

Remark that the impossibility of partially recognizing isolated groups
is also an open problem for groups described by finite presentations. 
\begin{conjecture}
The set of isolated group is not $\Lambda_{FP}$-semi-decidable. 
\end{conjecture}

While the Adian-Rabin theorem implies that no algorithm stops exactly
on finite presentations of non-isolated groups (since being isolated
is a Markov property, as any group with unsolvable word problem provides
a negative witness for it), it fails to prove that no algorithm stops
exactly on finite presentations of isolated groups. The problem of
proving that the set of finite presentations of simple groups is not
c.e. is, to the best of our knowledge, still open. (The question appears
for instance in \cite{Mostowski1973}.) 

This is related to the Higman-Boone conjecture, which asks whether
any group with solvable word problem embeds in a finitely presented
simple group. Indeed, a proof of the Higman-Boone conjecture would
imply that there exists simple groups of arbitrarily difficult (while
solvable) word problem, in terms of time complexity. By a simple diagonal
argument, we have:
\begin{lem}
[\cite{Rauzy21}, Proposition 26]Let $A$ be a set of finitely presented
groups with uniformly solvable word problem. If $A$ is $\Lambda_{FP}$-c.e.,
then there is a universal computable upper bound to the time complexity
for the word problem for groups in $A$. 
\end{lem}

And thus:
\begin{prop}
If the Higman-Boone conjecture holds, the set of finitely presented
simple groups is not $\Lambda_{FP}$-semi-decidable. 
\end{prop}

It is also unknown (the question appears in \cite{Cornulier2007})
whether isolated groups are dense in $\mathcal{G}^{+}$. If they were,
it would follow from Proposition \ref{prop:No-re dense seq} that
no sequence of word problem algorithms which contains each isolated
group can be enumerated. We could still arrive to that conclusion
if we knew that the word problem is not uniform on isolated groups,
that is to say, since all isolated groups are finitely presented,
if we knew that a solution to the word problem of an isolated group
cannot be retrieved from a finite presentation for this group. For
instance, it is well known that the word problem is uniform on the
set of simple groups \cite{Kuznetsov1958}, however, Kuznetsov's argument
fails if we add the trivial group to the set of simple groups. 

It would be very interesting to prove that the trivial group is unrecognizable
from simple groups, from the finite presentation description. This
would prove both that the word problem is not uniform on all isolated
groups, and that the set of finite presentations of simple groups
is not c.e.. However, too few finitely presented infinite simple groups
are known as of now to obtain such results. 

Jeandel has further investigated in \cite{Jeandel2017} the link between
Kuznetsov's Theorem and the space of marked groups, and given a general
framework where ``Kuznetsov type arguments'' do apply. 

\subsection{\label{subsec:elementary theory}LEF groups and the elementary theory
of groups }

In this paragraph, we use well known links between the elementary
theories of groups and the space of marked groups to study limit groups
and LEF groups. 

Limit groups are groups that have markings in the closure of the set
of free groups, while LEF groups are groups that have markings in
the closure of the set of finite groups. 

\subsubsection{Introduction on universal and existential theories of groups }

The space of marked groups was used by Champetier and Guirardel in
\cite{Champetier2005} in order to study limit groups, which play
an important role in the solution to Tarski's problem on the elementary
theory of free groups. We include here a paragraph that emphasizes
the links between our present study and the study of the universal
theories of various classes groups, and we point out some differences.
This will be the occasion to propose the set of LEF groups as another
candidate for the failure of the correspondence between the Borel
and arithmetical hierarchies. 

We do not want to include many definitions, and refer \cite{Champetier2005}
for precise definitions, and references. A formula is obtained with
variables, logical connectors ($\wedge$ is ``and'', $\vee$ is
``or'', and $\neg$ is ``not''), the equality symbol $=$, the
group law $\cdot$, the identity element $1$, and the group inverse
$^{-1}$, and the two quantifiers $\forall$ and $\exists$. We use
shortcuts where it is convenient (as the symbols $\ne$ or $\implies$),
and always use implicitly all group axioms. A \emph{sentence }is a
formula with no free variables. A \emph{universal sentence} is a sentence
that uses only the universal quantifier, and an \emph{existential
sentence} uses only the existential quantifier. 

For instance: 
\[
\forall x\forall y,\,x=y
\]
\[
\forall x\forall y\forall z,\,xy=yx\wedge yz=zy\wedge y\neq1\implies xz=zx
\]
\[
\exists x,\,x\neq1\wedge x^{2}=1
\]
 For a group $G$, let $T_{\forall}(G)$ denote the set of universal
sentences that are true in $G$, and $T_{\exists}(G)$ the set of
existential sentences that are true in $G$. For a class $\mathcal{C}$
of group we also write $T_{\forall}(\mathcal{C})$ and $T_{\exists}(\mathcal{C})$,
meaning the set of universal (resp. existential) sentences that hold
in \emph{all} groups of $\mathcal{C}$. 

In the space of marked groups, a universal sentence defines a closed
set, and the correspondence with the arithmetical hierarchy holds,
i.e., from a word problem algorithm, it is possible to prove that
a group does not satisfy a given universal sentence. Similarly, an
existential sentence defines an open set and the correspondence holds
for such sets. We will not be interested here in formulas with alternating
quantifiers. 

The following proposition of \cite{Champetier2005} follows directly
from the fact that universal sentences define closed sets: 
\begin{prop}
[\cite{Champetier2005}; Proposition 5.2]If a sequence of marked
groups $\left(G_{n}\right)_{n\in\mathbb{N}}$ converges to a marked
group $G$, then $\limsup(T_{\forall}(G_{n}))\subseteq T_{\forall}(G)$. 
\end{prop}

This proposition admits a converse, also due to Champetier and Guirardel,
which strengthens the relation between the space of marked groups
and the study of the elementary theory of groups. We reproduce its
proof here. 
\begin{prop}
[\cite{Champetier2005}; Proposition 5.3]\label{Prop Champ Equiv}
Suppose that two groups $G$ and $H$ satisfy $T_{\forall}(H)\subseteq T_{\forall}(G)$.
Then any marking of $G$ is a limit of markings of subgroups of $H$. 
\end{prop}

\begin{proof}
The proof in fact relies on the existential theories of the groups
$G$ and $H$, which satisfy the reversed inclusion: $T_{\exists}(G)\subseteq T_{\exists}(H)$.
Fix a generating family $S$ of $G$, and a radius $r$. Consider
the set $\left\{ w_{1},...,w_{k}\right\} $ of reduced words of length
at most $r$ on the alphabet $S\cup S^{-1}$. Consider the sets $J_{1}=\left\{ (i,j);w_{i}=_{G}w_{j}\right\} $
(where $=_{G}$ means that those words define identical elements of
$G$) and $J_{2}=\left\{ (i,j);w_{i}\neq_{G}w_{j}\right\} $. Then
$G$ satisfies the existential formula: 
\[
\exists S,\underset{(i,j)\in J_{1}}{\bigwedge}w_{i}=w_{j}\wedge\underset{(i,j)\in J_{2}}{\bigwedge}w_{i}\neq w_{j}
\]
By hypothesis, $H$ must satisfy it as well, which means precisely
that a subgroup of $H$ must have the same ball of radius $r$ as
$G$. 
\end{proof}
For a group $H$, denote by $\mathcal{S}(H)$ the set of all markings
of its subgroups. 
\begin{cor}
\label{cor to champ and guirardel}Let $G$ and $H$ be finitely generated
groups. The following are equivalent: 
\begin{itemize}
\item A marking of $G$ is adherent to the set $\mathcal{S}(H)$;
\item All markings of $G$ are adherent to the set $\mathcal{S}(H)$;
\item $T_{\forall}(H)\subseteq T_{\forall}(G)$. 
\end{itemize}
\end{cor}

We end this paragraph by using Markov's Lemma together with the above
result. 
\begin{lem}
[Markov's Lemma for Elementary Theories]Suppose that two groups
$G$ and $H$, with solvable word problem, satisfy $T_{\forall}(H)\subseteq T_{\forall}(G)$.

Then $\left[G\right]$ is not $\Lambda_{WP}$-semi-decidable inside
the set 
\[
\left[G\right]\cup\mathcal{S}(H).
\]
\end{lem}

\begin{proof}
This follows from Corollary \ref{cor to champ and guirardel}, Proposition
\ref{prop:Adherent to c.e. set} (the set $\mathcal{S}(H)$ is $\Lambda_{WP}$-c.e.),
and from Markov's Lemma for groups (Lemma \ref{lem:Markov's-Lemma-for Groups Intro}). 
\end{proof}

\subsubsection{Limit groups }

We will use the following definition for limit groups (those were
named in \cite{Sela2001}, see \cite{Champetier2005} for the equivalence
with other definitions): a group $G$ is a \emph{limit group} if some
(or all) of its markings are adherent to the set of marked free groups.
Note that if $G$ is a subgroup of a group $H$, every universal sentence
in $G$ holds in $H$. This implies that all non-abelian free groups
have the same universal theory, since each non-abelian free group
is a subgroup of each other non-abelian free group. 

Thus by Corollary \ref{cor to champ and guirardel}, a group $G$
is a limit group if and only if it satisfies $T_{\forall}(\mathbb{F}_{2})\subseteq T_{\forall}(G)$,
where $\mathbb{F}_{2}$ is the rank two free group. In fact, it is
known that if a group $G$ satisfies $T_{\forall}(\mathbb{F}_{2})\subseteq T_{\forall}(G)$,
then either it is abelian, and then it is free abelian, and $T_{\forall}(\mathbb{Z})=T_{\forall}(G)$,
or it has a free subgroup, which implies that $T_{\forall}(\mathbb{F}_{2})=T_{\forall}(G)$. 

The following proposition solves a decision problem for groups given
by word problem algorithms, while relying heavily on the study of
the elementary theory of groups. 
\begin{prop}
\label{Limit groups core-1}Being a limit group is $\Lambda_{WP}$-co-semi-decidable. 
\end{prop}

\begin{proof}
A group $G$ is a limit group if and only if it satisfies $T_{\forall}(\mathbb{F}_{2})\subseteq T_{\forall}(G)$.
A theorem of Makanin \cite{Makanin1985} states that the universal
theory of free groups is decidable, and thus that it is possible to
enumerate all universal sentences that hold in free groups. 

Since, given a word problem algorithm for a marked group $(G,S)$,
it is always possible to prove that a given universal sentence is
not satisfied in $G$, it is possible to detect groups that are not
limit groups by testing in parallel all sentences of the universal
theory of free groups. 
\end{proof}
This result is a slight improvement of a result in \cite{Groves2009},
where the same is obtained, but making use of both a finite presentation
and a word problem algorithm.

\subsubsection{LEF Groups}

The above result calls to our attention a second example of a natural
property for which the correspondence between the arithmetical hierarchy
and the Borel hierarchy might fail. Indeed, this last proof relies
heavily on Manakin's theorem. While the universal theory of free groups
is decidable, Slobodskoi proved in \cite{Slobodskoi1981} that the
universal theories of finite and torsion groups are unsolvable. 

Denote by $\mathcal{F}$ the set of marked finite groups, its closure
$\overline{\mathcal{F}}$ is the set of LEF groups. (LEF groups are
``Locally Embeddable into Finite groups'', they were first defined
in \cite{VershikGordon1998}, in terms of partial homomorphisms onto
finite group.)

We conjecture: 
\begin{conjecture}
\label{conj:LEF not CO SD}The sets of LEF groups with solvable word
problem is not $\nu_{WP}$-co-semi-decidable, nor Banach-Mazur $\nu_{WP}$-co-semi-decidable. 
\end{conjecture}

Note that, at first glance, Slobodskoi's Theorem does not seem to
be the sole thing preventing us from applying the proof of Proposition
\ref{Limit groups core-1} to LEF groups. Indeed, this proof relied
on the fact that a group $G$ is a limit group if and only if it satisfies
$T_{\forall}(\mathbb{F}_{2})\subseteq T_{\forall}(G)$, which in turn
used the fact that the inclusion $T_{\forall}(\mathbb{F}_{2})\subseteq T_{\forall}(G)$
is equivalent to the reverse inclusion $T_{\exists}(G)\subseteq T_{\exists}(\mathbb{F}_{2})$
(see Proposition \ref{Prop Champ Equiv}). This follows from the fact
that the elementary theory of a single group is \emph{complete}, i.e.
every sentence or its negation is in it. The theory of finite groups
is not complete, as the existential theory $T_{\exists}(\mathcal{F})$
contains only trivial sentences (they should hold in the trivial group),
yet the corresponding equivalence still holds. 
\begin{prop}
A group $G$ belongs to $\overline{\mathcal{F}}$ if and only if it
satisfies $T_{\forall}(\mathcal{F})\subseteq T_{\forall}(G)$. 
\end{prop}

\begin{proof}
We use the fact that there exists a group $K$ in $\overline{\mathcal{F}}$
that satisfies $T_{\forall}(K)=T_{\forall}(\mathcal{F})$. If a group
$G$ satisfies $T_{\forall}(K)=T_{\forall}(\mathcal{F})\subseteq T_{\forall}(G)$,
then it satisfies $T_{\exists}(K)\subseteq T_{\exists}(G)$, and by
Corollary \ref{cor to champ and guirardel}, $G$ is a limit of subgroups
of $K$. But $K$ and all its finitely generated subgroups are in
$\overline{\mathcal{F}}$, thus $G$ must also be a limit of markings
of finite groups. 

The group $K$ can be taken as the semi-direct product $\mathbb{Z}\ltimes\mathfrak{S}_{\infty}$,
where $\mathfrak{S}_{\infty}$ denotes the group of finitely supported
permutations of $\mathbb{Z}$, on which $\mathbb{Z}$ acts by translation.
This group is the limit of the finite groups $\mathbb{Z}/n\mathbb{Z}\ltimes\mathfrak{S}_{n}$,
as $n$ goes to infinity ($\mathfrak{S}_{n}$ is the group of permutation
over $\left\{ 1,...,n\right\} $). Since $K$ is in $\overline{\mathcal{F}}$,
$T_{\forall}(\mathcal{F})\subseteq T_{\forall}(K)$. However, because
it contains a copy of every finite group, one also has the reversed
inclusion. 
\end{proof}
Thanks to this proposition, we have: 
\begin{prop}
Conjecture \ref{conj:LEF not CO SD} implies Slobodskoi's Theorem. 
\end{prop}

\begin{proof}
Supposing that Slobodskoi's Theorem fails, one can reproduce the proof
of Proposition \ref{Limit groups core-1}, and prove that Conjecture
\ref{conj:LEF not CO SD} fails. 
\end{proof}
Other conjectures can be obtained, that are similar to Conjecture
\ref{conj:LEF not CO SD}: by a theorem of Kharlampovich \cite{Kharlampovich1983},
the universal theory of finite nilpotent groups is also undecidable,
and it is also known that the universal theory of hyperbolic groups
is undecidable (as proven by Osin in \cite{Osin2009}).
\begin{problem}
Is the closure $\bar{\mathcal{H}}$ of the set of hyperbolic groups
$\Lambda_{WP}$-co-semi-decidable? What of the closure of the set
of finite nilpotent groups? Or the closure of the set of torsion groups? 
\end{problem}

It was remarked by Bridson and Wilton in \cite{Bridson2015} that
Slobodskoi's proof from \cite{Slobodskoi1981} provides in fact more
than just unsolvability of the universal theory of finite groups.
Indeed it was shown there: 
\begin{thm}
[Slobodskoi, \cite{Slobodskoi1981}, unstated, \cite{Bridson2015},
Theorem 2.1]\label{thm:Slobodskoi,-,-unstated,} There exists a finitely
presented group $G=\langle S\,\vert\,R\rangle$ in which the problem:
``Given an element of $G$ as a word on the generators, decide whether
its image is trivial in all finite quotients of $G$'' is co-semi-decidable
but not decidable (i.e. there is a partial algorithm that stops when
an element indeed has a non trivial image in a finite quotient, but
no algorithm that proves the converse).
\end{thm}

An simple consequence of this theorem is the following:
\begin{cor}
\label{cor:Cor to Strong Slobodskoi}No algorithm can decide whether
or not a given set of relation and irrelations is satisfied by a finite
group. 

Or again: the problem ``Given a finite tuple of relations and irrelations
$(r_{1},...,r_{m};\,s_{1},...,s_{m'})$, decide whether $\Omega_{r_{1},...,r_{m};s_{1},...,s_{m'}}^{k}$
contains a finite group'' is semi-decidable but not decidable.
\end{cor}

\begin{proof}
This is proven just as Theorem \ref{thm:Boone-Novikov-reformulated}:
one considers basic clopen sets $\Omega_{R;\,w}^{k}$, where $R$
are the relations given in Slobodskoi's Unstated Theorem, and $w$
varies. 
\end{proof}
Using this corollary of Slobodskoi's Theorem, we can obtain more information
of the set of LEF groups. 

In \cite[Theorem 11]{Moschovakis1964}, Moschovakis proved that two
notions of ``effectively open sets'' agree on recursive Polish spaces:
computable unions of open balls, which are known as \emph{Lacombe
sets}, and ``recursively open sets'' \cite{Moschovakis1964}, which
are semi-decidable sets for which there is an algorithm that on input
a point in the set produces the radius of a ball that is still contained
in this set. This result cannot be applied to the space of marked
groups, and thus there are several competing notions of ``effectively
open sets''. It is an open problem whether Lacombe sets and recursively
open sets agree in the space of marked groups. However, it is always
the case that the set of Lacombe sets forms a subset of the set of
recursively open sets, thus being a Lacombe set is a priori more restrictive
than being a recursively open set. Note that many authors \emph{define}
the ``computably open sets'' to be the Lacombe sets \cite{Weihrauch2009}. 

We will now show: 
\begin{thm}
\label{thm:LEF not co LACOMBE}The set of LEF groups is not a co-Lacombe
set, that is to say its complement cannot be written as a computable
union of basic clopen sets. 
\end{thm}

\begin{proof}
Suppose that the set of non-LEF groups can be written as a computable
union of basic clopen sets:
\[
\overline{\mathcal{F}}^{c}=\bigcup_{i\in\mathbb{N}}\Omega_{i},
\]
where the map $i\mapsto\Omega_{i}$ is computable. We show that this
implies that there is an algorithm that stops exactly on basic clopen
sets that contain no finite groups, contradicting Corollary \ref{cor:Cor to Strong Slobodskoi}. 

If a basic clopen set $\Omega_{r_{1},...,r_{m};s_{1},...,s_{m'}}^{k}$
does not contain a finite group, then it is a subset of $\overline{\mathcal{F}}^{c}$.
By compactness, it is a subset of a finite union 
\[
\bigcup_{i\in\{0,...,n\}}\Omega_{i}.
\]
 By Lemma \ref{lem:Inclusion-between-basic sets is semi-decidable},
inclusion is semi-decidable between basic clopen sets and finite union
of basic clopen sets. The result then follows from a brute search
argument: if  $\Omega_{r_{1},...,r_{m};s_{1},...,s_{m'}}^{k}$indeed
is included in such a finite union, an exhaustive search will prove
it. 
\end{proof}

\subsection{Sofic groups }

The set of \emph{sofic groups} is known to be closed in $\mathcal{G}$.
However, whether it is all of $\mathcal{G}$ or a strict subset of
$\mathcal{G}$ is still an open problem. See \cite{Pestov2008} for
an introduction to Sofic groups. 

Let $\mathfrak{S}_{n}$ be the group of permutations on $\{1,...,n\}$.
The \emph{Hamming distance} in given by 
\[
d_{H}(\sigma_{1},\sigma_{2})=\frac{1}{n}\vert\{i:\sigma_{1}(i)\neq\sigma_{2}(i)\}\vert.
\]

\begin{defn}
A marked group $(G,S)$ is \emph{sofic} if for every $n\in\mathbb{N}$
and every $\epsilon>0$, there exists $N\in\mathbb{N}$ and a map
$\phi:B_{(G,S)}(n)\rightarrow\mathfrak{S}_{N}$ such that 
\begin{itemize}
\item $\phi(e)=\text{id}$;
\item $d_{H}(\phi(gh),\phi(g)\phi(h))<\epsilon$ for all $g$, $h$ in $B_{\mathcal{C}ay(G,S)}(n)$
whose product is also in $B_{\mathcal{C}ay(G,S)}(n)$;
\item $\phi(g)$ is fixed point free for every $g\neq e$. 
\end{itemize}
\end{defn}

What we want to note here is that it is possible that a phenomenon
similar to what happens for LEF groups also happens for sofic groups:
both definitions involve a $\forall\exists$ statement for which the
second existential statement is in fact bounded: there exists a universal
function $f$ such that the statement could equivalently be written
$\forall n\exists k\le f(n),P(n,k)$. In the case of LEF groups we
know by Slobodskoi's theorem that such a function $f$ cannot be asymptotically
below a computable function. It is natural to ask wether the same
holds for sofic groups. 
\begin{conjecture}
The set of sofic groups is a closed set which is computably a $G_{\delta}$
but not computably closed. 
\end{conjecture}

\section{\label{sec:Witnessing-results-in FP groups}Subgroups of finitely
presented groups with solvable word problem }

\subsection{\label{subsec:Higman-Clapham-Valiev-Theorem-fo}Higman-Clapham-Valiev
Theorem for groups with solvable word problem}

After Higman's proof of his famous Embedding Theorem \cite{Higman1961},
several theorems that resemble it were obtained.

In particular, it was remarked that the theorem is effective, meaning
that it provides an algorithm that takes as input a recursive presentation
for a group $G$, and outputs a finite presentation for a group $H$,
together with a finite family of elements of $H$ that generate $G$.
Note that in terms of numbering types, this implies that the numbering
type $\Lambda_{r.p.}$, associated to recursive presentation, is equivalent
to the numbering type associated to the idea ``a marked group $(G,S)$
is described by a finite presentation of an overgroup of $G$ together
with words that define the elements of $S$ in that overgroup''.
We leave out the details.

We will be interested here in the version of Higman's Theorem that
preserves solvability of the word problem \cite{Valiev1975,Clapham1967}.
This theorem is known as the Higman-Clapham-Valiev Theorem. 

Historical remarks about these results can be found in \cite{Olshanskii2004}.
The following formulation of the Higman-Clapham-Valiev Theorem can
also be found in \cite{McCool_1970}. 
\begin{thm}
[Higman-Clapham-Valiev, I]There exists a procedure that, given a
recursive presentation for a marked group $(G,S)$, produces a finite
presentation for a group $H$, together with an embedding $G\hookrightarrow H$
described by the images of the generators of $G$, and such that if
the word problem is solvable in $G$, then it is also solvable in
$H$. 
\end{thm}

One can also check that if one has access to a word problem algorithm
for the group given as input to this procedure, one can obtain a word
problem algorithm for the constructed finitely presented group. This
yields:
\begin{thm}
[Higman-Clapham-Valiev, II]There exists a procedure that, given
a word problem algorithm for a finitely generated group, produces
a finite presentation of a group in which it embeds, together with
a word problem algorithm for this new group, and a set of elements
that generate the first group. 
\end{thm}

This proves that, in general, the description of a group by its word
problem algorithm, or by a finite generating family inside a finitely
presented group with solvable word problem, are equivalent (we leave
it to the reader to render this statement precise: define a numbering
of $\mathcal{G}$ associated to the idea ``a group is given as a
subgroup of a group described by a finite presentation together with
a word problem algorithm'', thus using the numbering $\nu_{WP}\wedge\nu_{FP}$
to describe the overgroup, the Higman-Clapham-Valiev Theorem implies
that this numbering is equivalent to $\nu_{WP}$). 

Thus the study of algorithmic problems that can be solved from the
word problem description is identical to the study of decision problems
about subgroups of finitely presented groups with solvable word problem.
We now prove Theorem \ref{thm: Theorem A Banach Mazur Iff }. 
\begin{thm}
\label{thm:BanachMazurIFF}Let $P$ be a marked group property. The
following are equivalent: 
\begin{itemize}
\item There exists a finitely presentable marked group $(G,S)$ with solvable
word problem where the problem ``given elements $(w_{1},...,w_{k})$
in $G$, decide whether the marked group they generate has $P$''
is not semi-decidable;
\item The property $P$ is not $\Lambda_{WP}$-Banach-Mazur semi-decidable. 
\end{itemize}
\end{thm}

\begin{proof}
Suppose that $P$ is $\Lambda_{WP}$-Banach-Mazur semi-decidable.
Fix a marked group $(G,S)$ with solvable word problem. Consider an
effective enumeration $\{T_{1},T_{2},...\}$ of all finite sets of
words over $(S\cup S^{-1})^{*}$. The sequence of marked groups $((\langle T_{i}\rangle,T_{i}))_{n\in\mathbb{N}}$
is a computable sequence. Thus if $P$ is Banach-Mazur semi-decidable
the set $\{n\in\mathbb{N},P((\langle T_{i}\rangle,T_{i}))\}$ is a
c.e. subset of $\mathbb{N}$. And thus it is possible to semi-decide
$P$ on subgroups of $G$. 

Suppose that $P$ is not $\Lambda_{WP}$-Banach-Mazur semi-decidable.
There must exist a $\Lambda_{WP}$-computable sequence of marked groups
$((G_{n},S_{n}))_{n\in\mathbb{N}}$ such that $\{n\in\mathbb{N},P((G_{n},S_{n}))\}$
is not a c.e. subset of $\mathbb{N}$. Embedding the restricted direct
product of the $(G_{n},S_{n})$ in a finitely generated group with
solvable word problem (see for instance \cite{Darbinyan2015}) then
in a finitely presented group with solvable word problem via Higman's
embedding theorem yields the desired finitely presented group. 
\end{proof}
The following theorem is a direct consequence of a joint application
of the the above result with Markov's Lemma:
\begin{thm}
\label{thm:MarkovClaphamHigman}Suppose that a $\Lambda_{WP}$-computable
sequence $(G_{n})_{n\in\mathbb{N}}$ of $k$-marked groups effectively
converges to a $k$-marked group $H$, and suppose that $H\notin\{G_{n},n\in\mathbb{N}\}$.
Then there exists a finitely presented group $\Gamma$, with solvable
word problem, in which no algorithm can, given a tuple of elements
of $\Gamma$ that defines a marked group of \textup{$\{G_{n},n\in\mathbb{N}\}\cup\{H\}$},
stop if and only if this tuple defines $H$. 
\end{thm}

Note that when the conjugacy problem is uniformly solvable for the
groups in $\left(G_{k}\right)_{k\in\mathbb{N}}$, we may want to apply
the version of Higman's Theorem due to Alexander Olshanskii and Mark
Sapir (\cite{Olshanskii2004}, and \cite{OLSHANSKII2005} for non-finitely
generated groups) that preserves its solvability. 

\subsection{Some examples }

We now give some examples of possible applications of Theorem \ref{thm:MarkovClaphamHigman}.
The \emph{order problem }in a marked group asks for an algorithm that,
given an element of the group, determines what is its order. In a
group with solvable word problem, this is equivalent to being able
to decide whether or not an element has infinite order. The \emph{power
problem} asks for an algorithm which, given two elements of a marked
group, decides whether or not the first one is a power of the second
one. (It is thus the subgroup membership problem restricted to cyclic
subgroups.)

The following two propositions were proven by McCool in \cite{McCool_1970}. 
\begin{prop}
There exists a finitely presented group with solvable word problem,
but unsolvable order problem. 
\end{prop}

\begin{proof}
Apply Theorem \ref{thm:MarkovClaphamHigman} to a sequence of finite
cyclic groups that converges to $\mathbb{Z}$. This yields a finitely
presented group with solvable word problem in which one cannot decide
whether a given element generates a subgroup isomorphic to $\mathbb{Z}$
or to a finite cyclic group. This is precisely a finitely presented
group with solvable word problem, but unsolvable order problem. 
\end{proof}
\begin{prop}
There exists a finitely presented group with solvable word problem,
but unsolvable power problem.
\end{prop}

\begin{proof}
Apply Theorem \ref{thm:MarkovClaphamHigman} to the sequence of $2$-markings
of $\mathbb{Z}$ defined by the generating families $(1,k)$, $k\in\mathbb{N}^{*}$,
which converges to (the only $2$-marking of) $\mathbb{Z}^{2}$ when
$k$ goes to infinity (see \cite{Champetier2005}). This yields a
finitely presented group with solvable word problem where, given a
pair of commuting elements, one cannot decide whether they generate
$\mathbb{Z}^{2}$, or if one of these elements is a power of the other:
this is a group with unsolvable power problem. 
\end{proof}
We can also use this theorem to strengthen a result that was recently
obtained in \cite{Duda_2022}. 
\begin{thm}
There is a finitely presented group with solvable word problem in
which the problem of deciding whether a given subgroup is amenable
is neither semi-decidable nor co-semi-decidable. 
\end{thm}

\begin{proof}
This is proven by using both a sequence of marked amenable groups
which converges to a non-amenable group and a sequence of non-amenable
marked groups that converges to an amenable marked group. Such examples
were given in Section \ref{subsec:List-of-properties}. 
\end{proof}
The ``not semi-decidable'' half of this result is Theorem 6 in \cite{Duda_2022}.

Theorem \ref{thm:MarkovClaphamHigman} can be applied to all the properties
that appeared in Section \ref{subsec:List-of-properties} to produce
results similar to this one. The three results given above were well
known, but explaining them in terms of convergence in the space of
marked groups unifies several existing constructions. What's more,
for the rest of the effectively non-closed/non-open group properties
presented in Section \ref{subsec:List-of-properties}, the result
obtained by applying Theorem \ref{thm:MarkovClaphamHigman} are new.

\bibliographystyle{alpha}
\bibliography{TTEbiblio}

\end{document}